\documentclass[a4paper,10pt]{amsart}
\usepackage[top=0.8in, bottom=0.8in, left=1.25in, right=1.25in]{geometry}
\usepackage{amsmath}
\usepackage{amsfonts}
\usepackage{amssymb}
\usepackage{mathrsfs}
\usepackage{pb-diagram}
\usepackage{epstopdf}

\usepackage{amscd,amsthm,curves,enumerate,latexsym,multibox}
\usepackage[all]{xypic}
\usepackage[all]{xy}
\usepackage{epstopdf}
\newtheorem{theorem}{Theorem}[section]
\newtheorem{lemma}[theorem]{Lemma}
\newtheorem{proposition}[theorem]{Proposition}
\newtheorem{corollary}[theorem]{Corollary}
\theoremstyle{definition}
\newtheorem{definition}[theorem]{Definition}
\newtheorem{remark}[theorem]{Remark}

\newtheorem{axiom}[theorem]{Axiom}
\newtheorem{assumption}[theorem]{Assumption}
\newtheorem{question}[theorem]{Question}

\begin{document}
\title{Donaldson-Thomas Theory For Calabi-Yau 4-Folds}
\author{Yalong Cao}
\address{The Institute of Mathematical Sciences and Department of Mathematics, The Chinese University of Hong Kong, Shatin, Hong Kong}
\email{ylcao@math.cuhk.edu.hk}

\author{Naichung Conan Leung}
\address{The Institute of Mathematical Sciences and Department of Mathematics, The Chinese University of Hong Kong, Shatin, Hong Kong}
\email{leung@math.cuhk.edu.hk}

\maketitle

\begin{abstract}
{Let $X$ be a compact complex Calabi-Yau 4-fold.
Under certain assumptions, we define Donaldson-Thomas type deformation invariants ($DT_{4}$ invariants) by studying moduli spaces of solutions to the Donaldson-Thomas equations on $X$. We also study sheaves counting problems on local Calabi-Yau 4-folds. We relate
$DT_{4}$ invariants of $K_{Y}$ to the Donaldson-Thomas invariants of the associated Fano 3-fold $Y$. When the Calabi-Yau 4-fold is toric, we adapt the virtual localization formula to define the corresponding equivariant $DT_{4}$ invariants. We also discuss the non-commutative version of $DT_{4}$ invariants for quivers with relations. Finally, we compute $DT_{4}$ invariants for certain Calabi-Yau 4-folds when moduli spaces are smooth and find a $DT_{4}/GW$ correspondence for $X$. Examples of wall-crossing phenomenon in $DT_{4}$ theory are also given.}
\end{abstract}

\tableofcontents
\newpage
\textbf{Notations and conventions}. Throughout this paper, unless specified otherwise, $(X,\mathcal{O}_{X}(1))$ will be
a polarized compact complex Calabi-Yau 4-fold \cite{yau} equipped with a Ricci-flat K\"ahler metric $g$, a K\"ahler form $\omega$ and a holomorphic four-form $\Omega$ such that $\Omega\wedge\overline{\Omega}=dvol$ and $c_{1}(\mathcal{O}_{X}(1))=[\omega]$, where $dvol$ is the volume form of $g$. \\

We denote $(E,h)$ to be a complex vector bundle with a Hermitian metric over $X$ and $G$ to be the structure group of $E$ with center $C(G)$. We will restrict to the case when $G=U(r)$, where $r$ is the rank of $E$. \\

We denote $\mathcal{A}$ to be the space of all $L_{k}^{2}$ (Sobolev norm) unitary connections on $E$ and $\mathcal{G}$ to be the group of $L_{k+1}^{2}$ unitary gauge transformation, where $k$ is a large enough positive integer. $\Omega^{0,i}(X,EndE)_{k}$ is denote to be the completion of $\Omega^{0,i}(X,EndE)$ by $L_{k}^{2}$ norm. \\

We denote the space of irreducible $L_{k}^{2}$ unitary connections by
\begin{equation}\mathcal{A}^{*}=\{A\in \mathcal{A} \textrm{ } | \textrm{ } \Gamma_{A}=C(G)\}, \nonumber \end{equation}
where $\Gamma_{A}=\{u\in\mathcal{G} \textrm{ } | \textrm{ } u(A)=A \}$ is the isotropic group at $A$.
$\mathcal{A}^{*}$ is a dense open subset of $\mathcal{A}$ \cite{dk}.
Let $\mathcal{G}^{0}=\mathcal{G}/C(G)$ be the reduced gauge group. We know the action $\mathcal{G}^{0}$ on $\mathcal{A}^{*} $ is free and define $\mathcal{B}^{*}\triangleq\mathcal{A}^{*}/\mathcal{G}^{0}$, which is a Banach manifold \cite{d}, \cite{fu}. \\

We denote $\mathcal{M}_{c}(X,\mathcal{O}_{X}(1))$ or simply $\mathcal{M}_{c}$ to be the Gieseker moduli space of $\mathcal{O}_{X}(1)$-stable sheaves with given Chern character $c$.
We always assume $\mathcal{M}_{c}$ is compact, i.e. $\mathcal{M}_{c}=\overline{\mathcal{M}}_{c}$ ($\overline{\mathcal{M}}_{c}$ is the Gieseker moduli space of semi-stable sheaves) which is satisfied under the coprime condition on the degree and rank of coherent sheaves \cite{hl}. \\

We take $\mathcal{M}_{c}^{bdl}$ to be the analytic open subspace of $\mathcal{M}_{c}$ consisting of slope-stable holomorphic bundles which is possibly empty. We will not distinguish $\mathcal{M}_{c}^{bdl}$ with the moduli space of holomorphic Hermitian-Yang-Mills connections by Donaldson-Uhlenbeck-Yau's theorem \cite{d2}, \cite{UY}. \\

In this paper, when we say $\mathcal{M}_{c}$ is smooth, we always mean it in the strong sense, namely all Kuranishi maps are zero.

%\newpage
\section{Introduction}
In this paper, we study Donaldson-Thomas theory for Calabi-Yau 4-folds.
Originally, Floer studied Chern-Simons theory for closed oriented three-manifolds and
defined the instanton Floer homology generalizing the Casson invariant for moduli spaces of flat connections. For closed oriented four-manifolds, Donaldson defined polynomial invariants by studying moduli spaces of anti-self-dual connections on $SU(2)$ bundles over four-manifolds \cite{d},\cite{dk}.
Obviously, flat connections are anti-self-dual connections. The converse is also true if $ch_{2}(E)=0$.

Over complex-oriented manifolds, i.e. Calabi-Yau manifolds \cite{yau}, flat bundles are replaced by holomorphic bundles.
Thomas \cite{th} studied complex Chern-Simons gauge theory on Calabi-Yau 3-folds and defined the so-called
Donaldson-Thomas invariants for moduli spaces of stable holomorphic bundles (more generally, for Gieseker moduli spaces of stable sheaves). It was later generalized to semi-stable sheaves by Joyce and Song \cite{js}. A motivic version was proposed by Kontsevich and Soibelman \cite{ks},

As a complex analogue of Donaldson theory, we study the complex anti-self-dual equation on Calabi-Yau 4-folds which was written down by Donaldson and Thomas in \cite{dt}. Holomorphic Hermitian-Yang-Mills connections are complex anti-self-dual connections. The converse is also true if $ch_{2}(E)\in H^{2,2}(X)$. By Donaldson-Uhlenbeck-Yau's theorem \cite{d2}, \cite{UY}, the moduli space of holomorphic Hermitian-Yang-Mills connections is the moduli space of slope-stable bundles which has the natural Gieseker compactification. However, Gieseker moduli spaces on $CY_{4}$'s generally do not have perfect obstruction theory \cite{lt1}, \cite{bf} and do not obviously carry virtual fundamental classes. One of our main aims in this paper is to modify the obstruction theory of Gieseker moduli spaces, construct virtual fundamental classes and define the corresponding invariants under certain assumptions. \\

We start with a compact Calabi-Yau 4-fold $(X,g,\omega,\Omega)$ and define
\begin{equation}*_{4}: \Omega^{0,2}(X)\rightarrow \Omega^{0,2}(X),
\nonumber \end{equation}
\begin{equation}\alpha\wedge *_{4}\beta=(\alpha,\beta)_{g}\overline{\Omega}.
\nonumber \end{equation}
Coupled with bundle $(E,h)$, it is extended to
\begin{equation}*_{4}: \Omega^{0,2}(X,EndE)\rightarrow \Omega^{0,2}(X,EndE)
\nonumber \end{equation}
with $*_{4}^{2}=1$ \cite{dt}. Then we use this $*_{4}$ operator define the (anti-) self-dual subspace of $\Omega^{0,2}(X,EndE)$. In fact, $*_{4}$ also splits the corresponding harmonic subspace into self-dual and anti-self-dual parts.

The $DT_{4}$ equations are defined to be
\begin{equation} \left\{ \begin{array}{l}
  F^{0,2}_{+}=0 \\ F\wedge\omega^{3}=0 ,    % uses matrix to express the cpx ASD equations
\end{array}\right.
\nonumber \end{equation}
where the first equation is $F^{0,2}+*_{4}F^{0,2}=0$ and we assume $c_{1}(E)=0$ for simplicity in the moment map equation $ F\wedge\omega^{3}=0$.

We denote $\mathcal{M}^{DT_{4}}(X,g,[\omega],c,h)$ or simply $\mathcal{M}^{DT_{4}}_{c}$ to be the space of gauge equivalence classes of solutions to the $DT_{4}$ equations on a complex vector bundle $E$ with Chern character $ch(E)=c$. \\

To define Donaldson type invariants using $\mathcal{M}^{DT_{4}}_{c}$, we need
\begin{center}(1) compactness, \quad (2) orientation,  \quad (3) transversality.  \end{center}
\textbf{Compactness}. To compactify $\mathcal{M}_{c}^{DT_{4}}$, we start with its local Kuranishi structure.
\begin{theorem}\label{mo mDT4}(Theorem \ref{Kuranishi str of cpx ASD thm}) We assume $\mathcal{M}_{c}^{bdl}\neq\emptyset$, then \\
(1) $\mathcal{M}_{c}^{bdl}\cong\mathcal{M}_{c}^{DT_{4}}$ as sets. \\
(2) We fix $d_{A}\in\mathcal{M}_{c}^{DT_{4}}$, then there exists a Kuranishi map $\tilde{\tilde{\kappa}}$ of $\mathcal{M}_{c}^{bdl}$ at $\overline{\partial}_{A}$ (the (0,1) part of $d_{A}$) such that $\tilde{\tilde{\kappa}}_{+}$ is a Kuranishi map of $\mathcal{M}_{c}^{DT_{4}}$ at $d_{A}$, where
\begin{equation} \xymatrix@1{
\tilde{\tilde{\kappa}}_{+}=\pi_{+}(\tilde{\tilde{\kappa}}): H^{0,1}(X,EndE) \ar[r]^{\quad \quad \quad \tilde{\tilde{\kappa}}}
& H^{0,2}(X,EndE)\ar[r]^{\pi_{+}} & H^{0,2}_{+}(X,EndE) }  \nonumber \end{equation}
and $\pi_{+}$ is projection to the self-dual forms. \\
(3) The closed imbedding between analytic spaces possibly with non-reduced structures $\mathcal{M}_{c}^{bdl}\hookrightarrow \mathcal{M}_{c}^{DT_{4}}$
%\begin{equation}\mathcal{M}_{c}^{bdl}\hookrightarrow \mathcal{M}_{c}^{DT_{4}}  \nonumber \end{equation}
is also a homeomorphism between topological spaces.
\end{theorem}
In general, we want to obtain a compactification of $\mathcal{M}_{c}^{DT_{4}}$, denoted to be $\overline{\mathcal{M}}_{c}^{DT_{4}}$, such that the homeomorphism $\mathcal{M}_{c}^{bdl}\cong\mathcal{M}_{c}^{DT_{4}}$ can be extended to $\mathcal{M}_{c}\cong\overline{\mathcal{M}}_{c}^{DT_{4}}$ while the local analytic structure of $\overline{\mathcal{M}}_{c}^{DT_{4}}$ is given by $\kappa_{+}^{-1}(0)$, where
\begin{equation}\kappa_{+}=\pi_{+}(\kappa): Ext^{1}(\mathcal{F},\mathcal{F})\rightarrow  Ext^{2}_{+}(\mathcal{F},\mathcal{F}), \nonumber \end{equation}
$\kappa$ is a Kuranishi map of $\mathcal{M}_{c}$ at $\mathcal{F}$ and $Ext^{2}_{+}(\mathcal{F},\mathcal{F})$ is a half-dimensional real subspace
of $Ext^{2}(\mathcal{F},\mathcal{F})$ on which the Serre duality quadratic form is real and positive definite.

Although beginning with connections on bundles, we notice that $\mathcal{M}_{c}$ may not contain any locally free sheaf (like the moduli space of ideal sheaves of points) and the above gluing approach to define an analytic space $\overline{\mathcal{M}}_{c}^{DT_{4}}$ still makes sense.
We then call $\overline{\mathcal{M}}_{c}^{DT_{4}}$ the generalized $DT_{4}$ moduli space (Definition \ref{generalized DT4}).
The name comes from the fact that it may not parameterize any locally free sheaf in general while the $DT_{4}$ moduli space consists of connections on bundles only.

It is then obvious that if $\mathcal{M}_{c}=\mathcal{M}_{c}^{bdl}\neq\emptyset$, $\overline{\mathcal{M}}_{c}^{DT_{4}}$ exists and $\overline{\mathcal{M}}_{c}^{DT_{4}}=\mathcal{M}_{c}^{DT_{4}}$.
We have the following less obvious gluing results.
\begin{proposition}\label{condition of vir gene DT4}(Proposition \ref{gene DT4 if Mc smooth}, \ref{ob=v+v*}) ${}$ \\
If (i) $\mathcal{M}_{c}$ is smooth or (ii) for any closed point $\mathcal{F}\in \mathcal{M}_{c}$, there exists a complex vector space $V_{\mathcal{F}}$ and a linear isometry
\begin{equation}(Ext^{2}(\mathcal{F},\mathcal{F}),Q_{Serre})\cong (T^{*}V_{\mathcal{F}},Q_{std})  \nonumber \end{equation}
such that the image of a Kuranishi map $\kappa$ of $\mathcal{M}_{c}$ at $\mathcal{F}$ satisfies
\begin{equation}\emph{Image}(\kappa)\subseteq V_{\mathcal{F}},  \nonumber \end{equation}
then the generalized $DT_{4}$ moduli space exists and $\overline{\mathcal{M}}^{DT_{4}}_{c}\cong\mathcal{M}_{c}$ as real analytic spaces.
$Q_{Serre}$ is the Serre duality pairing and $Q_{std}$ is the standard pairing between $V_{\mathcal{F}}$ and $V_{\mathcal{F}}^{*}$.
\end{proposition}
\begin{remark}
Under any one of the above assumptions, we will no longer need to assume $\mathcal{M}_{c}$ contains any stable bundle.
\end{remark}
${}$ \\
\textbf{Orientation}.
The orientability issue for $\mathcal{M}^{DT_{4}}_{c}$ is concerning the determinant line bundle $\mathcal{L}$ (a real line bundle) of the index bundle of twisted Dirac operators.
%\begin{equation}\mathcal{L}|_{E}=\big(\wedge^{top}H^{2}_{+}(X,adE)\otimes \wedge^{top}H^{0}(X,adE)\big)^{-1}\otimes \wedge^{top}H^{1}(X,adE). \nonumber \end{equation}
We recall that if $\mathcal{M}_{c}^{bdl}\neq\emptyset$,
then $\mathcal{M}_{c}^{DT_{4}}\cong\mathcal{M}_{c}^{bdl}$ as topological spaces (Theorem \ref{mo mDT4}). In this case,
\begin{equation}\mathcal{L}|_{E}\cong\big(\wedge^{top}Ext^{2}_{+}(E,E)\big)^{-1}\otimes \wedge^{top}Ext^{1}(E,E), \nonumber \end{equation}
where $Ext^{2}_{+}(E,E)$ is the self-dual subspace of $Ext^{2}(E,E)$.

While on $\mathcal{M}_{c}^{bdl}$, its determinant line bundle $\mathcal{L}_{\mathbb{C}}$ exists as a complex line bundle such that
\begin{equation}\mathcal{L}_{\mathbb{C}}|_{E}\cong\big(\wedge^{top}Ext^{even}(E,E)\big)^{-1}\otimes \wedge^{top}Ext^{odd}(E,E).
\nonumber \end{equation}
By Serre duality, $\mathcal{L}_{\mathbb{C}}$ is endowed with a non-degenerate quadratic form $Q_{Serre}$. By Remark \ref{remark1}, one can check
that $\mathcal{L}$ is a real form of $(\mathcal{L}_{\mathbb{C}},Q_{Serre})$, i.e.
\begin{equation}\mathcal{L}_{\mathbb{C}}\cong\mathcal{L}\otimes\mathbb{C}. \nonumber \end{equation}
By the theory of quadratic bundles (complex vector bundles with non-degenerate quadratic forms) \cite{eg}, $\mathcal{L}$ is orientable if and only if
the structure group of $(\mathcal{L}_{\mathbb{C}},Q_{Serre})$ can be reduced to $SO(1,\mathbb{C})$.
Meanwhile, the choice of a reduction corresponds exactly to a choice of an orientation of $\mathcal{L}$ on $\mathcal{M}_{c}^{DT_{4}}$.
\begin{definition}\label{nat cpx ori}
A reduction of the structure group of $(\mathcal{L}_{\mathbb{C}},Q_{Serre})$ to $SO(1,\mathbb{C})$
is called a choice of an orientation of $(\mathcal{L}_{\mathbb{C}},Q_{Serre})$.
If the corresponding real line bundle $\mathcal{L}$ is the determinant line bundle of some complex vector bundle,
$(\mathcal{L}_{\mathbb{C}},Q_{Serre})$ is said to have a natural complex orientation, denoted by $o(\mathcal{O})$.
\end{definition}
\begin{remark}
By the above discussion, a choice of an orientation of $(\mathcal{L}_{\mathbb{C}},Q_{Serre})$ is equivalent to a choice of an orientation of $\mathcal{L}$ on $\mathcal{M}_{c}^{DT_{4}}$. Note that $(\mathcal{L}_{\mathbb{C}},Q_{Serre})$ has the advantage over $\mathcal{L}$ by being well defined also on $\mathcal{M}_{c}$ whereas $\mathcal{L}$ can be defined even when $c\notin\bigoplus_{k}\textrm{ }H^{k,k}(X)$. \\
\end{remark}
To pick a coherent choice of orientations for all components of the moduli space, as in Donaldson theory \cite{dk}, we need to extend the index bundle and its determinant line bundle to some big connected space such that $\mathcal{M}_{c}^{DT_{4}}$ (or $\mathcal{M}_{c}$) embeds inside with induced index (or determinant line) bundle.
We know $\mathcal{M}_{c}^{DT_{4}}\hookrightarrow \mathcal{B}^{*}$, where the determinant line bundle $\mathcal{L}$ extends naturally.

For $\mathcal{M}_{c}$ with $Hol(X)=SU(4)$, by Joyce-Song's work \cite{js} on Seidel-Thomas twists \cite{st}, we can identify it with some component(s) of $\mathcal{M}_{si}$, a coarse moduli space of simple holomorphic bundles with fixed Chern classes, where the corresponding quadratic bundle $(\mathcal{L}_{\mathbb{C}},Q_{Serre})$ is identified. By choosing a Hermitian metric, we can further imbed $\mathcal{M}_{si}$ into the space $\mathcal{B}^{*}$ of gauge equivalence classes of irreducible unitary connections, where the determinant line bundle $\mathcal{L}$ of twisted Dirac operators mentioned before is defined. Note that, one choice of an orientation of $\mathcal{L}$ gives an orientation of $(\mathcal{L}_{\mathbb{C}},Q_{Serre})$ on $\mathcal{M}_{si}$.

By \cite{dk}, $\mathcal{B}^{*}$ is connected, there are only two orientations for any orientable bundle. We assume from now on the determinant line bundle $\mathcal{L}$ on $\mathcal{B}^{*}$ is orientable.
\begin{definition}\label{ori data}
An \textit{orientation data} of $DT_{4}$ theory, denoted by $o(\mathcal{L})$, is a choice
%if $\mathcal{M}_{c}=\mathcal{M}_{c}^{bdl}$ or $c\notin\bigoplus_{k}\textrm{ }H^{k,k}(X)$, $o(\mathcal{L})$ is a choice of an orientation of $\mathcal{L}$ on $\mathcal{M}_{c}^{DT_{4}}$ induced from an orientation of $\mathcal{L}$ on $\mathcal{B}^{*}$; \\
%if $\mathcal{M}_{c}\neq\emptyset$, $o(\mathcal{L})$ is a choice
of an orientation of $(\mathcal{L}_{\mathbb{C}},Q_{Serre})$ on $\mathcal{M}_{c}$ (Definition \ref{nat cpx ori}) which is induced from an orientation of $\mathcal{L}$ on $\mathcal{B}^{*}$ via Seidel-Thomas twists.
\end{definition}
\begin{remark} ${}$ \\
1. The orientation data may involve choices of Seidel-Thomas twists for $\mathcal{M}_{c}$. \\
2. Making a choice of orientation on $\mathcal{L}$ from the ambient space $\mathcal{B}^{*}$ is for the purpose of deformation invariance of the theory. If we have some natural orientation of $(\mathcal{L}_{\mathbb{C}},Q_{Serre})$ on $\mathcal{M}_{c}$, such as $(\mathcal{L}_{\mathbb{C}},Q_{Serre})$ has a natural complex orientation \footnote{We will see there are many such examples.} (Definition \ref{nat cpx ori}), then we will just use that natural orientation without referring to $\mathcal{B}^{*}$.
%and we assume the natural orientation can also be induced from an orientation of $\mathcal{L}$ on $\mathcal{B}^{*}$.
\end{remark}
The existence of the orientation data is partially given by the following theorem.
\begin{theorem}(Theorem \ref{orientablity theorem}) ${}$ \\
For any compact Calabi-Yau 4-fold $X$ such that $H_{odd}(X,\mathbb{Z})=0$, and $U(r)$ bundle $E\rightarrow X$, the determinant line bundle $\mathcal{L}$ of the index bundle of twisted Dirac operators over the space $\mathcal{B}^{*}$ is trivial.
\end{theorem}
\begin{theorem}\label{cor on ori data}(Corollary \ref{existence of ori data}) ${}$ \\
Let $X$ be a compact Calabi-Yau 4-fold with $Hol(X)=SU(4)$ and $H_{odd}(X,\mathbb{Z})=0$, $\mathcal{M}_{c}$ be a Gieseker moduli space of stable sheaves. Then the structure group of $(\mathcal{L}_{\mathbb{C}},Q_{Serre})$ can be reduced to $SO(1,\mathbb{C})$. Furthermore, the orientation data of $DT_{4}$ theory (Definition \ref{ori data}) exists in this case.
\end{theorem}
${}$ \\
\textbf{Transversality}.
Finally, we come to the transversality issue, i.e. making sense of the fundamental class of the moduli space despite it may contain many components of different dimensions.
Firstly, we show that when $\mathcal{M}^{DT_{4}}_{c}$ is compact, its virtual fundamental class exists.
\begin{theorem}(Theorem \ref{main theorem}) \\
We assume $\overline{\mathcal{M}}_{c}=\mathcal{M}_{c}^{bdl}\neq\emptyset$ and there exists an orientation data $o(\mathcal{L})$,
then $\mathcal{M}^{DT_{4}}_{c}$ is compact and its virtual fundamental class exists as a cycle,
\begin{equation}[\mathcal{M}^{DT_{4}}_{c}]^{vir}\in H_{r}(\mathcal{B}^{*},\mathbb{Z}), \nonumber \end{equation}
where $r=2-\chi(E,E)$ is the real virtual dimension and $\chi(E,E)=\sum_{i}(-1)^{i}h^{i}(X,EndE)$
\footnote{By the Hirzebruch-Riemann-Roch theorem, $r$ depends only on $c$ and the topology of $X$.}.

Furthermore, if the above assumptions are satisfied by a continuous family of Calabi-Yau 4-folds $X_{t}$ parameterized by $t\in [0,1]$, then the virtual cycle in $H_{r}(\mathcal{B}^{*},\mathbb{Z})$ is independent of $t$.
\end{theorem}
We also define virtual fundamental classes of $\overline{\mathcal{M}}^{DT_{4}}_{c}$'s (see Definition \ref{virtual cycle when Mc smooth}, \ref{virtual cycle when ob=v+v*}) for the above two gluable cases where $\overline{\mathcal{M}}^{DT_{4}}_{c}\cong\mathcal{M}_{c}$
as real analytic spaces. \\
${}$ \\
\textbf{Axioms of $DT_{4}$ invariants}.
Similar to the case of Donaldson theory for four-manifolds \cite{dk}, we can use $\mu$-maps to cut down degrees of virtual fundamental classes and define the corresponding $DT_{4}$ invariants (Definition \ref{DT4 inv of bundles}, \ref{DT4 inv of sheaves}).

At the moment, we only define $DT_{4}$ invariants in several cases under different assumptions, i.e. under any one of the following assumptions and the assumption on the existence of the orientation data $o(\mathcal{L})$ (which is partially solved in Corollary \ref{cor on ori data} and cases (iii) and (iii') are always naturally oriented), we can define $DT_{4}$ invariants,  \\
(i) if the Gieseker moduli space consists of slope-stable bundles only, i.e. $\overline{\mathcal{M}}_{c}=\mathcal{M}_{c}^{bdl}\neq\emptyset$, or \\
(ii) if the Gieseker moduli space is smooth and consists of stable sheaves only, i.e. $\overline{\mathcal{M}}_{c}=\mathcal{M}_{c}$ is smooth, or  \\
(iii) if the Gieseker moduli space of compactly supported sheaves $\overline{\mathcal{M}}_{c}(K_{Y},\pi^{*}\mathcal{O}_{Y}(1))$ consists of slope-stable sheaves only, where $(Y,\mathcal{O}_{Y}(1))$ is a polarized Fano 3-fold, or more generally, \\
(iii') if $\overline{\mathcal{M}}_{c}=\mathcal{M}_{c}$ and there exists a perfect obstruction theory \cite{bf}
\begin{equation}\phi: \quad \mathcal{V}^{\bullet}\rightarrow \mathbb{L}^{\bullet}_{\mathcal{M}_{c}},  \nonumber \end{equation}
such that
\begin{equation}H^{0}(\mathcal{V}^{\bullet})|_{\{\mathcal{F}\}}\cong Ext^{1}(\mathcal{F},\mathcal{F}), \nonumber \end{equation}
\begin{equation}H^{-1}(\mathcal{V}^{\bullet})|_{\{\mathcal{F}\}}\oplus
H^{-1}(\mathcal{V}^{\bullet})|_{\{\mathcal{F}\}}^{*}\cong Ext^{2}(\mathcal{F},\mathcal{F}),  \nonumber \end{equation}
and $H^{-1}(\mathcal{V}^{\bullet})|_{\{\mathcal{F}\}}$ is a maximal isotropic subspace of $Ext^{2}(\mathcal{F},\mathcal{F})$ with respect
to the Serre duality pairing. \\

To make all these cases consistent, we propose several axioms that $DT_{4}$ invariants should satisfy.
Axioms $(6)$-$(8)$ are showed in the paper and axioms $(1)$-$(5)$ are verified when we have definitions of virtual fundamental classes of (generalized) $DT_{4}$ moduli spaces.
\begin{axiom}Given a triple $(X,\mathcal{O}_{X}(1),c)$ and an auxiliary choice of an orientation data $o(\mathcal{L})$, where $(X,\mathcal{O}_{X}(1))$ is a polarized Calabi-Yau 4-fold,
$c\in H^{even}_{c}(X,\mathbb{Q})$ is a (compactly supported) cohomology class,
the $DT_{4}$ invariant (i.e. Donaldson-Thomas 4-fold invariant) of this quadruple, denoted by $DT_{4}(X,\mathcal{O}_{X}(1),c,o(\mathcal{L}))$ is a map
\begin{equation}DT_{4}(X,\mathcal{O}_{X}(1),c,o(\mathcal{L})): Sym^{*}\big(H_{*}(X,\mathbb{Z})\otimes \mathbb{Z}[x_{1},x_{2},...]\big)
\rightarrow \mathbb{Z}, \nonumber \end{equation}
($Sym$ means graded symmetric with respect to the parity of the degree of $H_{*}(X)$ ) such that:  \\
$\textbf{(1)}$ \textbf{Orientation reversed}
\begin{equation}DT_{4}(X,\mathcal{O}_{X}(1),c,o(\mathcal{L}))=-DT_{4}(X,\mathcal{O}_{X}(1),c,-o(\mathcal{L})),  \nonumber \end{equation}
where $-o(\mathcal{L})$ denotes the opposite orientation of $o(\mathcal{L})$. \\
$\textbf{(2)}$ \textbf{Deformation invariance}
\begin{equation}DT_{4}(X_{0},\mathcal{O}_{X_{0}}(1),c,o(\mathcal{L}_{0}))=DT_{4}(X_{1},\mathcal{O}_{X_{1}}(1),c,o(\mathcal{L}_{1})),  \nonumber \end{equation}
where $(X_{t},\mathcal{O}_{X_{t}}(1))$ is a continuous family of complex structures and $o(\mathcal{L}_{t})$ is an orientation data on the family determinant line bundle with $t\in [0,1]$. \\
$\textbf{(3)}$ \textbf{Vanishing for certain virtual dimensions  }
\begin{equation}DT_{4}(X,\mathcal{O}_{X}(1),c,o(\mathcal{L}))=0,  \nonumber \end{equation}
if $X$ is compact and any one of the following two conditions is satisfied, \\
(i) $\chi(\mathcal{F},\mathcal{F})>2$, or
(ii) $\chi(\mathcal{F},\mathcal{F})$ is odd and $H^{odd}(X,\mathbb{Z})=0$.
$\chi(\mathcal{F},\mathcal{F})$ is the holomorphic Euler characteristic uniquely determined by $c$ and the topology of $X$.  \\
$\textbf{(4)}$ \textbf{Vanishing for certain choices of $c$  }
\begin{equation}DT_{4}(X,\mathcal{O}_{X}(1),c,o(\mathcal{L}))=0,  \nonumber \end{equation}
if $X$ is compact and any one of the following two conditions is satisfied, \\
(i) $c|_{H^{4}(X,\mathbb{Q})}$ has no component in $H^{0,4}(X)$ and $c\notin \bigoplus_{i=0}^{4}H^{i,i}(X)$, or \\
(ii) $c\in \bigoplus_{i=0}^{4}H^{i,i}(X)$ and $\exists\textrm{ } \varphi\in H^{1}(X,TX)$ such that $\varphi\lrcorner\textrm{ } \big(c|_{H^{2,2}(X,\mathbb{Q})}\big)\neq0$ (Proposition \ref{vanishing for some c}). \\
$\textbf{(5)}$ \textbf{Vanishing for compact hyper-K\"ahler manifolds  }
\begin{equation}DT_{4}(X,\mathcal{O}_{X}(1),c,o(\mathcal{L}))=0, \footnote{Orientation data is not defined for hyper-K\"ahler manifolds. The statement is: with respect to any orientation of $(\mathcal{L}_{\mathbb{C}},Q_{Serre})$, the invariant is zero.}  \nonumber \end{equation}
when $X$ is compact hyper-K\"ahler (\ref{nu+}).  \\
%$\textbf{(4)}$ \textbf{Vanishing for $T^{*}S$ }
%\begin{equation}DT_{4}(X,\pi^{*}\mathcal{O}_{S}(1),c,o(\mathcal{L}))=0  \nonumber \end{equation}
%if $c=(0,0,c|_{H_{c}^{4}(X)}\neq 0,c|_{H_{c}^{6}(X)},c|_{H_{c}^{8}(X)})$ and $\mathcal{M}_{c}$ contains a slope-stable sheaf scheme theoretically supported on $S$, where $\pi: X=T^{*}S\rightarrow S$ and $S$ is a del-Pezzo surface. $\mathcal{O}_{S}(1)$ is an ample line bundle such that $\pi^{*}\mathcal{O}_{S}(1)$ is also ample. \\
$\textbf{(6)}$ \textbf{$DT_{4}/DT_{3}$ correspondence}
\begin{equation}DT_{4}(X,\pi^{*}\mathcal{O}_{Y}(1),c,o(\mathcal{O}))=DT_{3}(Y,\mathcal{O}_{Y}(1),c'), \nonumber \end{equation}
if $c=(0,c|_{H_{c}^{2}(X)}\neq 0,c|_{H_{c}^{4}(X)},c|_{H_{c}^{6}(X)},c|_{H_{c}^{8}(X)})$ and the Gieseker moduli space of compactly supported sheaves $\overline{\mathcal{M}}_{c}(X,\pi^{*}\mathcal{O}_{Y}(1))$ consists of slope-stable sheaves, where $\pi: X=K_{Y}\rightarrow Y$ is projection and $(Y,\mathcal{O}_{Y}(1))$ is a polarized compact Fano 3-fold.

In this setup, sheaves in $\overline{\mathcal{M}}_{c}(X,\pi^{*}\mathcal{O}_{Y}(1))$ are of type $\iota_{*}(\mathcal{F})$ where $\iota: Y\rightarrow K_{Y}$ is the zero section and $c'=ch(\mathcal{F})\in H^{even}(Y)$ is uniquely determined by $c$. $o(\mathcal{O})$ is the natural complex orientation of $(\mathcal{L}_{\mathbb{C}},Q_{Serre})$ over $\mathcal{M}_{c}$.

$DT_{3}(Y,\mathcal{O}_{Y}(1),c')$ is the $DT_{3}$ invariant of $(Y,\mathcal{O}_{Y}(1))$ with certain insertion fields (
see Theorem \ref{compact supp DT4}). \\
%$\textbf{(7)}$ \textbf{Invariance under Mukai flops}
%\begin{equation}DT_{4}(X,\mathcal{O}_{X}(1),c,o(\mathcal{L}))=DT_{4}(X^{+},\mathcal{O}_{X^{+}}(1),\phi^{*}c,o(\mathcal{L}^{+})), \nonumber \end{equation}
%if $c=(1,*,*,*,*)\in H^{even}(X)$, and $\phi: X^{+}\dashrightarrow X$ is a Mukai flop \cite{mukai} between Calabi-Yau 4-folds (where $\phi$ induces an
%isomorphism $H^{*}(X^{+},\mathbb{Z})\cong H^{*}(X,\mathbb{Z})$ of cohomology rings, see e.g. \cite{huzhang}), $\mathcal{O}_{X^{+}}(1)$ is any ample line bundle
%on $X^{+}$ and $o(\mathcal{L}^{+})$ is a suitable orientation data for $(X^{+},\mathcal{O}_{X^{+}}(1),\phi^{*}c)$
%(Conjecture \ref{DT4 inv under Mukai flops}). \\
$\textbf{(7)}$ \textbf{Normalization 1}
\begin{equation}DT_{4}(X,\mathcal{O}_{X}(1),c,o(\mathcal{L}))=DT_{4}^{\mu_{1}}(X,\mathcal{O}_{X}(1),c,o(\mathcal{L})), \nonumber \end{equation}
if $X$ is compact and $\overline{\mathcal{M}}_{c}=\mathcal{M}_{c}^{bdl}\neq\emptyset$.

$DT_{4}^{\mu_{1}}(X,\mathcal{O}_{X}(1),c,o(\mathcal{L}))$ is defined using the virtual fundamental class of $\mathcal{M}_{c}^{DT_{4}}$ and the corresponding $\mu$-map (\ref{mu map for bundles}). \\
$\textbf{(8)}$ \textbf{Normalization 2}
\begin{equation}DT_{4}(X,\mathcal{O}_{X}(1),c,o(\mathcal{L}))=DT_{4}^{\mu_{2}}(X,\mathcal{O}_{X}(1),c,o(\mathcal{L})), \nonumber \end{equation}
if $\overline{\mathcal{M}}_{c}=\mathcal{M}_{c}\neq\emptyset$ is smooth or satisfies the condition in Definition \ref{virtual cycle when ob=v+v*}.

$DT_{4}^{\mu_{2}}(X,\mathcal{O}_{X}(1),c,o(\mathcal{L}))$ is defined using the virtual fundamental class of $\overline{\mathcal{M}}_{c}^{DT_{4}}$ and the corresponding $\mu$-map (\ref{u2 map}).
\end{axiom}
In normalization axioms, the construction depends on the existence of virtual fundamental classes and $\mu$-map descendent
fields as mentioned above. Throughout the paper, we will often only mention $DT_{4}$ virtual cycles (virtual fundamental classes of generalized $DT_{4}$ moduli spaces) instead of using
the corresponding $DT_{4}$ invariants for convenience purposes. \\
${}$ \\
\textbf{Computational examples}.
Li-Qin \cite{lq} had provided examples when $\overline{\mathcal{M}}_{c}=\mathcal{M}_{c}^{bdl}\neq\emptyset$. By studying them,
we have the following wall-crossing phenomenon in $DT_{4}$ theory.
\begin{theorem}(Theorem \ref{liqin eg}) ${}$ \\
Let $X$ be a generic smooth hyperplane section in $\mathbb{P}^{1}\times\mathbb{P}^{4}$ of bi-degree $(2,5)$.
Let
\begin{equation}cl=[1+(-1,1)|_{X}]\cdot[1+(\epsilon_{1}+1,\epsilon_{2}-1)|_{X}],
\nonumber
\end{equation}
\begin{equation}k=(1+\epsilon_{1})\left(\begin{array}{l}6-\epsilon_{2} \\ \quad 4\end{array}\right), \quad \epsilon_{1},\epsilon_{2}=0,1.
\nonumber\end{equation}
We denote $\overline{\mathcal{M}}_{c}(L_{r})$ to be the moduli space of Gieseker $L_{r}$-semi-stable rank two
torsion-free sheaves with Chern character $c$ (which can be easily read from the total Chern class $cl$), where $L_{r}=\mathcal{O}_{\mathbb{P}^{1}\times\mathbb{P}^{4}}(1,r)|_{X}$. \\
$(1)$ If
\begin{equation}\frac{15(2-\epsilon_{2})}{6+5\epsilon_{1}+2\epsilon_{2}}<r<\frac{15(2-\epsilon_{2})}{\epsilon_{1}(1+2\epsilon_{2})},
\nonumber\end{equation}
then $\overline{\mathcal{M}}_{c}^{DT_{4}}$ exists and $\overline{\mathcal{M}}_{c}^{DT_{4}}\cong\overline{\mathcal{M}}_{c}(L_{r})\cong\mathbb{P}^{k}$, $[\overline{\mathcal{M}}_{c}^{DT_{4}}]^{vir}=[\mathbb{P}^{k}]$. \\
${}$ \\
$(2)$ If
\begin{equation} 0<r<\frac{15(2-\epsilon_{2})}{6+5\epsilon_{1}+2\epsilon_{2}},\nonumber\end{equation}
then $\overline{\mathcal{M}}_{c}^{DT_{4}}=\emptyset $ and $[\overline{\mathcal{M}}_{c}^{DT_{4}}]^{vir}=0$.
\end{theorem}
We also study the $DT_{4}/GW$ correspondence for compact Calabi-Yau 4-folds in some specific cases.
\begin{theorem}(Theorem \ref{DT=GW}, \ref{DT=GW2}) ${}$ \\
Let $X$ be a compact Calabi-Yau 4-fold. We assume $\mathcal{M}_{c}$ with given Chern character $c=(1,0,0,-PD(\beta),-1)$
is smooth and consists of ideal sheaves of smooth connected genus zero imbedded curves only, then $(\mathcal{L}_{\mathbb{C}},Q_{Serre})$
on $\mathcal{M}_{c}$ has a natural complex orientation $o(\mathcal{O})$.
Assume the $GW$ moduli space is isomorphic to the Gieseker moduli space, i.e.
$\overline{\mathcal{M}}_{0,0}(X,\beta)\cong \mathcal{M}_{c}$, then $\overline{\mathcal{M}}_{c}^{DT_{4}}$ exists and
$\overline{\mathcal{M}}_{c}^{DT_{4}}\cong \overline{\mathcal{M}}_{0,0}(X,\beta)$. Furthermore,

$(1)$ if $Hol(X)=SU(4)$, then $[\overline{\mathcal{M}}_{c}^{DT_{4}}]^{vir}=[\overline{\mathcal{M}}_{0,0}(X,\beta)]^{vir}$.

$(2)$ if $Hol(X)=Sp(2)$, i.e. irreducible hyper-K\"ahler, then $[\overline{\mathcal{M}}_{c}^{DT_{4}}]^{vir}=0$. \\
Furthermore, $[\overline{\mathcal{M}}_{c}^{DT_{4}}]^{vir}_{hyper-red}=[\overline{\mathcal{M}}_{0,0}(X,\beta)]^{vir}_{red}$ (see theorem \ref{DT=GW2} for its meaning).
\end{theorem}
When $X=Tot(\mathcal{O}_{\mathbb{P}^{2}}(-1)\oplus\mathcal{O}_{\mathbb{P}^{2}}(-2))$ is the total space of a rank two bundle over $\mathbb{P}^{2}$,
we consider counting torsion sheaves scheme theoretically supported on $\mathbb{P}^{2}$
(i.e. they are of type $\iota_{*}\mathcal{F}$, where $\iota:\mathbb{P}^{2}\rightarrow X$ is the zero section).
We denote
\begin{equation}\mathcal{M}_{c}^{\mathbb{P}^{2}_{cpn}}=\{\iota_{*}\mathcal{F}\textrm{ } |\textrm{ } \mathcal{F}\in \mathcal{M}_{c}(\mathbb{P}^{2})  \}
\cong\mathcal{M}_{c}(\mathbb{P}^{2}) \nonumber \end{equation}
to be the component(s) of a moduli of sheaves on $X$ which can be identified with the Gieseker moduli space $\mathcal{M}_{c}(\mathbb{P}^{2})$ for $\mathbb{P}^{2}$ with Chern character $c\in H^{even}(\mathbb{P}^{2})$ (we assume every Gieseker semi-stable sheaf is slope stable).
As $\mathcal{M}_{c}(\mathbb{P}^{2})$ is smooth, the obstruction sheaf of $\mathcal{M}_{c}^{\mathbb{P}^{2}_{cpn}}$ is a vector bundle endowed with
a non-degenerate quadratic form (Serre duality pairing).
The $DT_{4}$ virtual cycle $[\mathcal{M}_{c}^{\mathbb{P}^{2}_{cpn}}]^{vir}$ of $\mathcal{M}_{c}^{\mathbb{P}^{2}_{cpn}}$
is defined to be the Poincar\'{e} dual of the Euler class of a self-dual obstruction bundle.
\begin{proposition}(see Proposition \ref{rk 2 bundle over S} for a more general result) ${}$ \\
The $DT_{4}$ virtual cycle $[\mathcal{M}_{c}^{\mathbb{P}^{2}_{cpn}}]^{vir}$ satisfies \\
(i) $[\mathcal{M}_{c}^{\mathbb{P}^{2}_{cpn}}]^{vir}=0$, if $c|_{H^{0}(\mathbb{P}^{2})}\geq2$;  \\
(ii) $[\mathcal{M}_{c}^{\mathbb{P}^{2}_{cpn}}]^{vir}\in H_{0}(\mathcal{M}_{c}^{\mathbb{P}^{2}_{cpn}})$, if $c|_{H^{0}(\mathbb{P}^{2})}=1$.
Furthermore, with respect to the natural complex orientation, we have
\begin{equation}\sum_{n\geq0}[\mathcal{M}_{(1,0,-n)}^{\mathbb{P}^{2}_{cpn}}]^{vir}q^{n}=\prod_{k\geq1}(\frac{1}{1-q^{k}}). \nonumber \end{equation}
\end{proposition}
Lastly, for ideal sheaves of one point, we have
\begin{proposition}(Proposition \ref{moduli of one point}) ${}$ \\
If $Hol(X)=SU(4)$ and $c=(1,0,0,0,-1)$, then $\overline{\mathcal{M}}_{c}^{DT_{4}}$ exists and $\overline{\mathcal{M}}_{c}^{DT_{4}}\cong X$, $[\overline{\mathcal{M}}_{c}^{DT_{4}}]^{vir}=\pm PD(c_{3}(X))$.
\end{proposition}
${}$ \\
\textbf{Equivariant $DT_{4}$ invariants}.
We also study the equivariant $DT_{4}$ theory for ideal sheaves of surfaces, curves and points on toric $CY_{4}$. We adapt the virtual localization formula due to Graber and Pandharipande \cite{gp} to our case and define the corresponding equivariant $DT_{4}$ invariants. Because torus fixed points of corresponding Gieseker moduli spaces are discrete, we do not need $\overline{\mathcal{M}}_{c}^{DT_{4}}$ to define invariants. We get the definition without any constraint on moduli spaces. \\

By studying the simplest example for $\mathbb{C}^{4}$, we get
\begin{proposition}(Proposition \ref{equi for C4}) ${}$ \\
Let $X=\mathbb{C}^{4}$, for some choice of toric orientation data (Definition \ref{toric ori data}), we have
\begin{equation}Z_{DT_{4}}\big(X,1 \textrm{ } | (1,1)\big)_{0}=\frac{\sigma_{1}\sigma_{2}-\sigma_{3}}{\sigma_{1}\sigma_{3}}, \nonumber \end{equation}
where $\sigma_{i}$ is the $i$-th elementary symmetric polynomial of variables $\lambda_{1},\lambda_{2},\lambda_{3}$ and
$Z_{DT_{4}}\big(X,1 \textrm{ } | (1,1)\big)_{0}$ is the equivariant $DT_{4}$ invariant (Definition \ref{def of equi DT4}) of $\mathbb{C}^{4}$ for
ideal sheaves of one point without any insertion field.
\end{proposition}
${}$ \\
\textbf{Non-commutative $DT_{4}$ invariants}.
In the non-commutative world, where sheaves on manifolds are replaced by representations of algebras, we also have a definition of Donaldson-Thomas type theory for four dimensional Calabi-Yau algebras. We define the $NCDT_{4}$ invariants and compute some examples. The detail is left to Section 9. \\

${}$ \\
\textbf{Comparisons with Borisov-Joyce's work}.
A related work was done by Borisov and Joyce \cite{bj}. They used local 'Darboux charts' in the sense of Brav, Bussi and Joyce \cite{bbj}
(based on Pantev-T\"{o}en-Vaqui\'{e}-Vezzosi's theory of shifted symplectic geometry \cite{ptvv}),
the machinery of homotopical algebra and $C^{\infty}$-algebraic geometry to obtain a compact derived $C^{\infty}$-scheme with
the same underlying topological structure as the Gieseker moduli space of stable sheaves. In our language,
their results proved the existence of generalized $DT_{4}$ moduli spaces in general ($C^{\infty}$-scheme version).
Furthermore, they defined the virtual fundamental class of the derived $C^{\infty}$-scheme.

In the appendix, we will first give another proof of BBJ's Darboux theorem (the analytic version) for Gieseker moduli spaces of stable
sheaves using gauge theory and Seidel-Thomas twists \cite{js}, \cite{st}. We then introduce a weaker condition on their local 'Darboux charts'
to include local models induced from $DT_{4}$ equations (i.e. the map $\tilde{\tilde{\kappa}}$ in Theorem \ref{mo mDT4}).
It turns out that the weaker condition is already sufficient for the gluing requirement in Borisov and Joyce's work \cite{bj}
which then indicates the equivalence of their approach to virtual fundamental classes and our $DT_{4}$ virtual cycles.  \\
${}$ \\
\textbf{Content of the paper }: In section 2, we study the $*_{4}$ operator which is needed in the definition of $DT_{4}$ equations.
In section 3, we study local analytic structures of $DT_{4}$ moduli spaces. In section 4, we compactify $DT_{4}$ moduli spaces
under certain assumptions. Under the gluing assumption \ref{assumption on gluing}, which is verified in some cases, we define the
generalized $DT_{4}$ moduli space. In section 5, we construct $DT_{4}$ virtual cycles, study actions of global monodromy on them and
also give some vanishing results. In section 6, we study $DT_{4}$ invariants for compactly supported sheaves on local Calabi-Yau 4-folds.
In section 7, we compute some $DT_{4}$ invariants when $\mathcal{M}_{c}$'s are smooth. In section 8, we define equivariant $DT_{4}$ invariants
on toric $CY_{4}$. In section 9, we discuss a definition of non-commutative $DT_{4}$ invariants for $CY_{4}$ algebras.
In the last section, we list some useful facts as
appendix. In particular, we prove the existence of the orientation data of $DT_{4}$ theory in a large number of cases.
We also give another proof of 'Darboux theorem' (in the sense of Brav, Bussi and Joyce \cite{bbj} Corollary 5.20)
in the case when $\mathbb{K}=\mathbb{C}$ and $\mathcal{M}$ is the Gieseker moduli space of stable sheaves using gauge theory and
Seidel-Thomas twists \cite{js}, \cite{st}. \\
${}$ \\
\textbf{Acknowledgement}: The first author would like to express his deep gratitude to Huai-Liang Chang, Zheng Hua, Wei-Ping Li,
Dennis Sullivan and Jiu-Kang Yu for many useful discussions and encouragements. We are grateful to Dennis Borisov and Dominic Joyce
for their interests in this work and informing us their related work and thoughts. Special thanks to Dennis Borisov for
explaining his joint work with Joyce \cite{bj}. We also hope to express our deep gratitude to Simon Donaldson for
many useful discussions and inviting the first author to visit him at the Simons Center for Geometry and Physics on May 2014.
The work of the second author was substantially supported by a grant from the Research Grants Council of the
Hong Kong Special Administrative Region, China (Project No. CUHK401411).

\section{The $*_{4}$ operator}

\subsection{The $*_{4}$ operator for bundles}
In this section, we introduce the $*_{4}$ operator on the space of bundle valued differential forms which
is the key to the definition of $DT_{4}$ equations and the construction of $DT_{4}$ moduli spaces.
%As proposed by Donaldson and Thomas in \cite{dt}, we have a commutative diagram
%\begin{equation}   \xymatrix{
%                      \Omega^{0,k}(X) \ar[rr]^{*} \ar[dr]_{*_{4}}
%                        &  &    \Omega^{4,4-k}(X)  \ar[dl]^{\lrcorner\Omega }    \\
%                                  & \Omega^{0,4-k}(X).     } \nonumber\end{equation}
%The commutativity of the above diagram defines the  $*_{4}$ operator, i.e.
We define
\begin{equation}*_{4}: \Omega^{0,k}(X)\rightarrow\Omega^{0,4-k}(X),\nonumber \end{equation}
\begin{equation} \alpha\wedge*_{4}\alpha=|\alpha|^{2}\overline{\Omega},\nonumber \end{equation}
%The $*$ operator is the usual $\mathbb{C}$-anti-linear Hodge star and $\alpha$ is a $(0,k)$ form on $X$.
which satisfies $*_{4}^{2}=1$ when $k=2$. Note that $*_{4}$ is a complex anti-linear map.

Given any Hermitian metric on $E$, we can extend $*_{4}$ to $\Omega^{0,k}(X,EndE)$. Using $*_{4}^{2}=1$, we have the following orthogonal decomposition,
\begin{equation}\Omega^{0,2}(X,EndE)=\Omega^{0,2}_{+}(X,EndE)\oplus\Omega^{0,2}_{-}(X,EndE), \nonumber\end{equation}
given by $\alpha=\frac{1}{2}(\alpha+*_{4}\alpha)+\frac{1}{2}(\alpha-*_{4}\alpha)$, where $\Omega^{0,2}_{\pm}(X,EndE)=\{\alpha\in\Omega^{0,2}(X,EndE)\mid *_{4}\alpha=\pm\alpha\}$.
%We take a unitary frame $ \{e_{i}\}$ with respect to the Hermitian metric $h$ on $E$ with dual frame $ \{e^{i}\} $. We then define
%\begin{equation}*_{4}(\alpha\otimes e_{i} \otimes e^{j})=*_{4}(\alpha)\otimes e_{j} \otimes e^{i}.
%\nonumber \end{equation}
%This is a well defined operator which also satisfies the identity $*_{4}^{2}=1$.
The following lemma shows that $*_{4}$ descends to corresponding harmonic subspaces.
\begin{lemma}\label{lem descendence of star}
$\overline{\partial}_{E}^{*}=\overline{\partial}_{E}^{*_{4}}$ on $\Omega^{0,k}(X,EndE)$, where $\overline{\partial}_{E}$
is a holomorphic $(0,1)$ connection and  $\overline{\partial}_{E}^{*_{4}}\triangleq-*_{4}\overline{\partial}_{E}*_{4}$.
\end{lemma}
\begin{proof}
%\begin{eqnarray*}*\overline{\partial}_{E}*(\alpha\otimes e_{i} \otimes e^{j})&=&*\overline{\partial}_{E}(*\alpha\otimes e_{j} \otimes e^{i}) \\
%&=&*\big((\overline{\partial}_{E}*\alpha)\otimes e_{j} \otimes e^{i}+(-1)^{|\alpha |} *\alpha\wedge \overline{\partial}_{E}(e_{j} \otimes e^{i})\big) \\
%&=&*\overline{\partial}_{E}*(\alpha)\otimes e_{i} \otimes e^{j}+(-1)^{|\alpha |} *(*\alpha\wedge\sum\beta_{k,l}\otimes e_{k} \otimes e^{l})\\
%&=&*\overline{\partial}_{E}*(\alpha)\otimes e_{i} \otimes e^{j}+(-1)^{|\alpha |} \sum*(*\alpha\wedge\beta_{k,l})\otimes e_{l} \otimes e^{k}.
%\end{eqnarray*}
%Similarly, we have
%\begin{equation}*_{4}\overline{\partial}_{E}*_{4}(\alpha\otimes e_{i} \otimes e^{j})=*_{4}\overline{\partial}_{E}*_{4}(\alpha)\otimes e_{i}
%\otimes e^{j}+(-1)^{|\alpha |}\sum*_{4}(*_{4}\alpha\wedge\beta_{k,l})\otimes e_{l} \otimes e^{k} .\nonumber \end{equation}
This follows directly from $*=\Omega\wedge*_{4}$ and $\Omega$ being a norm one parallel form.
\end{proof}
\begin{corollary}\label{cor star4 cut coho into ASD}
%The $*_{4}$ operator splits the space $\Omega^{0,2}(X,EndE)$ into
%\begin{equation}\Omega^{0,2}(X,EndE)=\Omega^{0,2}_{+}(X,EndE)\oplus\Omega^{0,2}_{-}(X,EndE) \nonumber\end{equation}
The $*_{4}$ operator splits the space of harmonic forms $H^{0,2}(X,EndE)$ into
\begin{equation}H^{0,2}(X,EndE)=H^{0,2}_{+}(X,EndE)\oplus H^{0,2}_{-}(X,EndE)   \nonumber\end{equation}
according to $(\pm1)$ eigenvalues.
\end{corollary}
%\begin{proof}
%By the fact that $*_{4}^{2}=1$ and Lemma \ref{lem descendence of star}.
%\end{proof}
\begin{remark}\label{remark1}
%1. The above $*_{4}$ operator is naturally extended to the corresponding Banach spaces $\Omega^{0,\bullet}(X,EndE)_{k}$. \\
It is obvious that $\sqrt{-1}H^{0,2}_{+}=H^{0,2}_{-}$. Thus they have the same dimension.
%\begin{equation}v.d_{\mathbb{R}}\triangleq 2h^{0,1}(X,EndE)-h^{0,2}(X,EndE)=2h^{0}(X,EndE)-\chi(E,E)\nonumber\end{equation}
%If $E$ is simple, the $v.d_{\mathbb{R}}$ is a topological invariant by the Hirzebruch-Riemann-Roch theorem.
\end{remark}

\subsection{The $*_{4}$ operator for general coherent sheaves}
The construction of the $*_{4}$ operator on $H^{0,2}(X,EndE)=Ext^{2}(E,E)$ has a generalization to $Ext^{2}(\mathcal{F},\mathcal{F})$ for any coherent sheaf $\mathcal{F}$ on $X$: Since $X$ is smooth and projective, we can resolve $\mathcal{F}$ by a complex of holomorphic vector bundles
$(E^{\bullet},\delta)\rightarrow\mathcal{F}\rightarrow 0$.
We consider the following double complex \cite{hl},
\begin{equation}(\Omega^{0,q}(X,\mathcal{H}om^{p}(E^{\bullet},E^{\bullet})),\overline{\partial},\delta). \nonumber\end{equation}
%where $\overline{\partial}$ is defined in terms of holomorphic structures on $E^{*}$ and $\delta$ is induced from the differential in the above resolved complex \cite{hl}.
On the total complex $(C^{*},D)$, where $C^{n}=\bigoplus_{l\in\mathbb{Z}} \bigoplus_{p+q=n}\Omega^{0,q}(X,\mathcal{H}om(E^{l},E^{l+p}))$
%\begin{equation} C^{n}\triangleq\bigoplus_{k\in\mathbb{Z}} \bigoplus_{p+q=n}\Omega^{0,q}(X,\mathcal{H}om(E^{k},E^{k+p}))\nonumber \end{equation}
and $D=\overline{\partial}+(-1)^{q}\delta$, there exists two natural filtrations which induce two spectral sequences converging to $Ext^{*}(\mathcal{F},\mathcal{F})$.
%\begin{definition}\label{*4 on sheaves}
%We fix Hermitian metrics on $\{E^{k}\}$ and define the $*_{4}$ operator,
%\begin{equation}*_{4}: \Omega^{0,q}(X,\mathcal{H}om(E^{k},E^{k+p}))\rightarrow \Omega^{0,4-q}(X,\mathcal{H}om(E^{k+p},E^{k}))\nonumber \end{equation} such that $*_{4}(\alpha\otimes e^{i}\otimes f_{j})=*_{4}(\alpha)\otimes f^{j}\otimes e_{i}$, where $\alpha\in \Omega^{0,q}(X)$,
%$\{e_{i}\}, \{f_{j}\}$ are unitary frames of $E^{k},E^{k+p}$ and $\{e^{i}\}, \{f^{j}\}$ are the corresponding dual frames.
%We define the $*_{4}$ operator on $C^{*}$ by $\mathbb{R}$-linearly extending the definition on each grading piece.
%\end{definition}

We fix Hermitian metrics on $\{E^{l}\}$ and similarly define the $*_{4}$ operator
\begin{equation}*_{4}: \Omega^{0,q}(X,\mathcal{H}om(E^{l},E^{l+p}))\rightarrow \Omega^{0,4-q}(X,\mathcal{H}om(E^{l+p},E^{l}))\nonumber \end{equation} as above. We then $\mathbb{R}$-linearly extend $*_{4}$ to $C^{*}$.
\begin{lemma}\label{double cpx is elliptic}
$(C^{*},D)$ is an elliptic complex.
\end{lemma}
\begin{proof}
$D$ is a $C^{\infty}$-differential operator.
%Firstly, the operator $D$ is a $C^{\infty}$ differential operator. This is done by recalling
%$\delta(\varphi)=\delta_{E}\circ\varphi-(-1)^{deg\varphi}\varphi\circ \delta_{E}$ and $\delta_{E}:E^{k}\rightarrow E^{k+1}$
%is a morphism between bundles induced from the differential in the resolved complex. Then $\delta$ is a linear operator on $\Omega^{0,q}(X,\mathcal{H}om (E^{k}, E^{k+p}))$. Since $\overline{\partial}$ is obviously a differential operator, thus $D=\overline{\partial}+(-1)^{q}\delta$ is a differential operator.
The twisted Laplacian $\Delta_{D}=D^{*}D+DD^{*}$ has an expansion $\Delta_{D}=\Delta_{\overline{\partial}}$+ 1st order terms of
$\overline{\partial}$ + 0 order terms, which shows $D$ is elliptic \cite{wells}.
\end{proof}
\begin{corollary} \label{cutting for sheaves}
The above $*_{4}$ operator on $C^{*}$ descends to the $D$-harmonic subspace of $C^{*}$. Furthermore, it splits $H^{2}(C^{*},D)\cong
Ext^{2}(\mathcal{F},\mathcal{F})$ into
\begin{equation}Ext^{2}(\mathcal{F},\mathcal{F})=Ext^{2}_{+}(\mathcal{F},\mathcal{F})\oplus Ext^{2}_{-}(\mathcal{F},\mathcal{F})
\nonumber \end{equation}
according to $(\pm1)$ eigenvalues.
\end{corollary}
\begin{proof}
The operator $D^{*_{4}}\triangleq-*_{4}D*_{4}$ equals to $D^{*}$ as in Lemma \ref{lem descendence of star}. By a similar argument as in Corollary \ref{cor star4 cut coho into ASD}, we are done.
\end{proof}
\begin{remark} The subspaces $Ext^{2}_{\pm}(\mathcal{F},\mathcal{F})$ depend on the choice of metrics on $E^{\bullet}$.
\end{remark}

\section{Local Kuranishi structures of $DT_{4}$ moduli spaces}
To study gauge theory on a general K\"{a}hler manifold $X$, we consider the moduli space of poly-stable holomorphic structures on $E$.
By the renowned theorem of Donaldson-Uhlenbeck-Yau \cite{d2}, \cite{UY}, this space equals to the moduli space of holomorphic Hermitian-Yang-Mills connections on $E$, i.e. the space of gauge equivalence classes of solutions to
\begin{equation} \left\{ \begin{array}{l} \label{complex ASD equation}
  F^{0,2}=0 \\ F\wedge\omega^{3}=0 ,
\end{array}\right. \nonumber \end{equation}
where we have assumed $c_{1}(E)=0$ in the moment map equation $F\wedge\omega^{3}=0$ for simplicity.
However, the holomorphic HYM equations are overdetermined when $dim_{\mathbb{C}}(X)\geq3$. \\
${}$ \\
\textbf{The $DT_{4}$ moduli space}.
%We start with our basic data
%\begin{equation}
%\begin{array}{lll}
%      &  \textrm{ } (E,h)
%      \\  &  \quad \textrm{ } \downarrow \\   &   (X,g,\omega).
%      %  leave space using \textrm{ }
%\end{array}\nonumber\end{equation}
When $X$ is a $CY_{4}$, Donaldson and Thomas \cite{dt} used the calibrated form $\Omega$ to cut down the number of equations
in the holomorphic HYM equations to obtain an elliptic system, i.e. Donaldson and Thomas' complex ASD equations ($DT_{4}$ equations for short),
\begin{equation} \left\{ \begin{array}{l} \label{complex ASD equation}          % l rep one array , the # of line rep by \\
  F^{0,2}_{+}=0 \\ F\wedge\omega^{3}=0 ,    % uses matrix to express the cpx ASD equations
\end{array}\right.\end{equation}
where the first equation is $F^{0,2}+*_{4}F^{0,2}=0$.
\begin{remark} ${}$ \\
%1. For notation simplicity, we assume $c_{1}(E)=0$ in the moment map equation $F\wedge\omega^{3}=0$. \\
1. The group of unitary gauge transformations preserves the above equations.  \\
%3. We call the above equations the $DT_{4}$ equations.\\
2. From the viewpoint of deformation-obstruction theory, we will see that this defines a perfect-obstruction theory, which leads to the construction of a virtual fundamental class for its moduli.
\end{remark}
Then the definition of $DT_{4}$ moduli spaces follows from the above $DT_{4}$ equations.
\begin{definition}
The $DT_{4}$ moduli space $\mathcal{M}^{DT_{4}}_{c}$ is the space of gauge equivalence classes of solutions to
equations (\ref{complex ASD equation}).
\end{definition}
Note that $\mathcal{M}^{DT_{4}}_{c}\hookrightarrow \mathcal{B}^{*}=\mathcal{A^{*}}/\mathcal{G}^{0}$ embeds
as the zero loci of a section $s=(\wedge F,F^{0,2}_{+})$ of a Banach bundle $\mathcal{E}$ over $\mathcal{B}^{*}$,
where $\mathcal{E}=\mathcal{A^{*}}\times_{\mathcal{G}^{0}}(\Omega^{0}(X,g_{E})_{k-1}\oplus\Omega^{0,2}_{+}(X,EndE)_{k-1})$.
This gives the $DT_{4}$ moduli space a natural real analytic structure.
\begin{remark} The Banach manifold $\mathcal{A}^{*}/\mathcal{G}^{0}$ involves a choice of a large integer $k $ in the $L_{k}^{2}$ Sobolev norm completion. By the ellipticity of $DT_{4}$ equations and Proposition 4.2.16 \cite{dk}, $DT_{4}$ moduli spaces are independent of the choice of $k$. Thus we omit $k$ in the notation.
\end{remark}
%\textbf{ Relations between $\mathcal{M}^{DT_{4}}_{c}$ and $\mathcal{M}_{c}^{bdl}$ }.
By the definition of $\mathcal{M}^{DT_{4}}_{c}$, we have an obvious inclusion between two sets
\begin{equation}\mathcal{M}_{c}^{bdl}\rightarrow \mathcal{M}^{DT_{4}}_{c}. \nonumber \end{equation}
If $\mathcal{M}_{c}^{bdl}\neq\emptyset$, by Lemma \ref{C2 condition}, the inclusion is a bijection.   \\
${}$ \\
\textbf{Local structures of $DT_{4}$ moduli spaces}. Even though we have identified $\mathcal{M}^{DT_{4}}_{c}$ and $\mathcal{M}_{c}^{bdl}$ as sets,
they could have different possibly non-reduced analytic structures. We will show there exists a closed imbedding $\mathcal{M}_{c}^{bdl}\hookrightarrow \mathcal{M}^{DT_{4}}_{c}$ between these two real analytic spaces. \\

We start with an unitary connection $d_{A}\in \mathcal{M}^{DT_{4}}_{c}$, denote its $(0,1)$ part by $\overline{\partial}_{A}$ (we call $\overline{\partial}_{A}$ a $(0,1)$ connection). We have $F^{0,2}_{A}=0$ under the assumption $\mathcal{M}_{c}^{bdl}\neq\emptyset$.
%We assume without loss of generality that $d_{A}$ is a smooth connection by Proposition 4.2.16 \cite{dk}.

By the Hodge decomposition theorem, we have
\begin{equation}\Omega^{0,2}(X,EndE)_{k-1}=H^{0,2}(X,EndE)\oplus \overline{\partial}_{A}\Omega^{0,1}(X,EndE)_{k}
\oplus\overline{\partial}_{A}^{*}\Omega^{0,3}(X,EndE)_{k},  \nonumber \end{equation}
\begin{equation} I=\mathbb{H}^{0,2}+P_{\overline{\partial}_{A}}+ P_{\overline{\partial}_{A}^{*}}, \nonumber \end{equation}
where the RHS consists of projections to the corresponding components.

Meanwhile,
\begin{equation}*_{4}:\overline{\partial}_{A}\Omega^{0,1}(X,EndE)_{k}\cong\overline{\partial}_{A}^{*}\Omega^{0,3}(X,EndE)_{k}
\nonumber\end{equation}
and
\begin{equation}*_{4}: H^{0,2}(X,EndE)\cong H^{0,2}(X,EndE)  \nonumber\end{equation}
induce
\begin{equation}H^{0,2}(X,EndE)=H^{0,2}_{+}(X,EndE)\oplus H^{0,2}_{-}(X,EndE)  \nonumber\end{equation}
and
\begin{equation}F^{0,2}_{+}(d_{A}+a)=0 \Leftrightarrow a''\triangleq (a)^{0,1}\textrm{}\emph{satisfies the following} \nonumber\end{equation}
\begin{equation}\label{equation 3}\overline{\partial}_{A}a''+P_{\overline{\partial}_{A}}(a''\wedge a'')
+*_{4}P_{\overline{\partial}_{A}^{*}}(a''\wedge a'')=0 ,   \end{equation}
\begin{equation}\label{equation 4}\pi_{+}\circ\mathbb{H}^{0,2}(a''\wedge a'')=0. \end{equation}
Here $a\in\Omega^{1}(X,g_{E})_{k}$ and $a''\triangleq (a)^{0,1}\in\Omega^{0,1}(X,EndE)_{k}$ is the (0,1) part of $a$.  \\

Using the gauge fixing $d_{A}^{*}a=0$, a neighbourhood of $d_{A}\in\mathcal{M}^{DT_{4}}_{c}$ can be described as
\begin{equation} \{a\in\Omega^{1}(X,g_{E})_{k} \big{|}\textrm{ } \| a\|_{k} < \epsilon, \textrm{ } d_{A}^{*}a=0 ,
\textrm{ } d_{A}+a  \textrm{ }  \emph{satisfies}  \textrm{ } (\ref{complex ASD equation}) \}, \nonumber\end{equation}
where $\epsilon$ is a small positive number.
We introduce a new space $\mathcal{M}_{A}^{+}$ to help establish relations of local structures
between $\mathcal{M}_{c}^{DT_{4}}$ and $\mathcal{M}_{c}^{bdl}$.
\begin{equation}  \mathcal{M}_{A}^{+}\triangleq\{a''\in\Omega^{0,1}(X,EndE)_{k} \textrm{ }\big{|}\textrm{ }\| a''\|_{k} < \epsilon'',\textrm{ }
a'' \textrm{ } \emph{satisfies} \textrm{ } (\ref{equation 3}), \textrm{} (\ref{equation 4}), \textrm{} (\ref{equation 5})  \}
\nonumber  \end{equation}
where (\ref{equation 5}) is defined to be
\begin{equation}\label{equation 5}\overline{\partial}_{A}^{*}a''-\frac{i}{2}\wedge(a'\wedge a''+a''\wedge a')=0 ,  \end{equation}
$\epsilon''$ is a small positive number and $a'\triangleq(a)^{1,0}$ is the $(1,0)$ part of $a$. \\

In fact, $\mathcal{M}_{A}^{+}$ is locally isomorphic to a neighbourhood of $d_{A}\in\mathcal{M}^{DT_{4}}_{c}$.
\begin{lemma}\label{identify SA MA}
The map which takes unitary connections to their $(0,1)$ parts
\begin{equation}d_{A}+a\longmapsto \overline{\partial}_{A}+a'' \nonumber \end{equation}
induces a local isomorphism near the origin
\begin{equation}\{a\in\Omega^{1}(X,g_{E})_{k} \big{|}\textrm{ } \|a\|_{k} < \epsilon, \textrm{ } d_{A}^{*}a=0 , \textrm{ } d_{A}+a  \textrm{ }
\emph{satisfies} \textrm{ }
(\ref{complex ASD equation}) \}\cong \mathcal{M}_{A}^{+}.  \nonumber \end{equation}
\end{lemma}
\begin{proof}
Locally, we consider an ambient space of the $DT_{4}$ moduli space
\begin{equation} \{a\in\Omega^{1}(X,g_{E})_{k} \big{|}\textrm{ } \|a\|_{k} < \epsilon, \textrm{ } d_{A}^{*}a=0 ,
\textrm{ } \wedge(d_{A}a+a\wedge a)=0, \textrm{ }  a''  \textrm{ }  \emph{satisfies}  \textrm{ } (\ref{equation 3}) \}. \nonumber\end{equation}
By the isomorphism sending unitary connections to $(0,1)$ connections,
we identify it with an open subset of $Q_{A}$ by implicit function theorem ($\epsilon \ll 1$), where
\begin{equation}   Q_{A}=\{a''\in\Omega^{0,1}(X,EndE)_{k} \textrm{ }\big{|}\textrm{ }\|a''\|_{k} < \epsilon'',\textrm{ }
a'' \textrm{ } \emph{satisfies}  \textrm{ } (\ref{equation 3}), \textrm{} (\ref{equation 5})  \}. \nonumber  \end{equation}
Because $F^{0,2}_{+}=0$ is preserved by the map $d_{A}+a\longmapsto \overline{\partial}_{A}+a'' $,
we get isomorphic analytic subspaces after adding it to both of the above two ambient spaces.
\end{proof}
The ambient space $Q_{A}$ of $\mathcal{M}_{A}^{+}$ can be locally identified with $H^{0,1}(X,EndE)$.
\begin{lemma}\label{QA iso to H1}
We denote
\begin{equation}Q_{A}=\{a''\in\Omega^{0,1}(X,EndE)_{k} \textrm{ }\big{|}\textrm{ }\|a''\|_{k} < \epsilon'',\textrm{ }
a'' \textrm{ } \emph{satisfies}  \textrm{ } (\ref{equation 3}), \textrm{} (\ref{equation 5})  \}, \nonumber  \end{equation}
then the harmonic projection map
\begin{equation}\mathbb{H}^{0,1}: Q_{A}\rightarrow H^{0,1}(X,EndE)   \nonumber \end{equation}
is a local analytic isomorphism if $\epsilon''$ is small.
\end{lemma}
\begin{proof}
We define
\begin{equation}q: \Omega^{0,1}(EndE)_{k}\rightarrow H^{0,1}(EndE)\oplus {\overline{\partial}_{A}^{*}\Omega^{0,1}(EndE)}_{k}\oplus
{\overline{\partial}_{A}^{*}\Omega^{0,2}(EndE)}_{k-1},  \nonumber \end{equation}
\begin{equation}q(a'')=\bigg(\mathbb{H}(a''),\overline{\partial}_{A}^{*}a''-\frac{i}{2}\wedge(a'\wedge a''+a''\wedge a'),
\overline{\partial}_{A}^{*}\big(\overline{\partial}_{A}a''+P_{\overline{\partial}_{A}}(a''\wedge a'')
+*_{4}P_{\overline{\partial}_{A}^{*}}(a''\wedge a'')\big)\bigg).  \nonumber \end{equation}
We show $q$ is well-defined: Firstly, we have
\begin{eqnarray*}\wedge(\varphi)&=&\wedge(\mathbb{H}(\varphi))+\wedge(\overline{\partial}_{A}\overline{\partial}_{A}^{*}G\varphi)+
\wedge(\overline{\partial}_{A}^{*}\overline{\partial}_{A}G\varphi)  \\
&=& \wedge(\mathbb{H}(\varphi))+ \overline{\partial}_{A}\wedge(\overline{\partial}_{A}^{*}G\varphi)+
\overline{\partial}_{A}^{*}(\wedge\overline{\partial}_{A}G\varphi) \\
&=& \wedge(\mathbb{H}(\varphi))+0+\overline{\partial}_{A}^{*}(\wedge\overline{\partial}_{A}G\varphi),
\nonumber\end{eqnarray*}
where $\varphi=a'\wedge a''+a''\wedge a'\in\Omega^{1,1}(X,End_{0}E)$.

Meanwhile, we have
\begin{equation} \overline{\partial}_{A}\wedge\mathbb{H}(\varphi)=\wedge\overline{\partial}_{A}\mathbb{H}(\varphi)\pm
\overline{\partial}_{A}^{*}\mathbb{H}(\varphi)=0,  \quad \overline{\partial}_{A}^{*}\wedge\mathbb{H}(\varphi)=0. \nonumber\end{equation}
%\begin{equation} \overline{\partial}_{A}^{*}\wedge\mathbb{H}(\varphi)=0. \nonumber\end{equation}
Thus $\wedge\mathbb{H}(\varphi)\in H^{0}(End_{0}E)=0$ by the simpleness of
$(E,\overline{\partial}_{A})$.
Hence $q$ is well-defined.

We take the differentiation of $q$ at $0$,
\begin{equation}dq_{0}(v)=(\mathbb{H}(v),\overline{\partial}_{A}^{*}v,\overline{\partial}_{A}^{*}\overline{\partial}_{A}v ), \nonumber\end{equation}
which is a diffeomorphism whose inverse is given by
\begin{equation}dq_{0}^{-1}(u_{0},u_{1},u_{2})=u_{0}+G\overline{\partial}_{A}u_{1}+Gu_{2}. \nonumber \end{equation}
By the implicit function theorem, $q$ is a local analytic isomorphism near the origin.
\end{proof}
By Lemma \ref{identify SA MA}, $\mathcal{M}_{c}^{DT_{4}}$ is locally identified with $\mathcal{M}_{A}^{+}$.
%(we may need to shrink it by requiring $\epsilon''$ to be smaller)
To compare $\mathcal{M}_{A}^{+}$ and $\mathcal{M}_{c}^{bdl}$, we define,
\begin{equation}\label{M_A}\mathcal{M}_{A}=\{a''\in\Omega^{0,1}(X,EndE)_{k} \textrm{ }\big{|}\textrm{ }\| a''\|_{k} < \epsilon'',\textrm{ }
\mathbb{H}^{0,2}(a''\wedge a'')=0 ,\textrm{} a'' \textrm{ }
\emph{satisfies}  \textrm{ } (\ref{equation 3}), \textrm{} (\ref{equation 5}) \}.    \end{equation}
By Lemma \ref{QA iso to H1}, $\mathcal{M}_{A}\hookrightarrow Q_{A}$ embeds as a closed analytic subspace of a finite dimensional smooth manifold $Q_{A}$.
We will show $\mathcal{M}_{A}$ is isomorphic to an analytic neighbourhood of $\overline{\partial}_{A}$ in $\mathcal{M}_{c}^{bdl}$ with analytic topology.
To achieve this, we denote local analytic maps
\begin{equation}P:\Omega^{0,1}(X,EndE)_{k}\rightarrow \Omega^{0,2}(X,EndE)_{k-1}, \nonumber \end{equation}
\begin{equation}P(a'')\triangleq\overline{\partial}_{A}a''+a''\wedge a'' \nonumber \end{equation}
and
\begin{equation}\lambda: Q_{A}\rightarrow Q_{A}\times \Omega^{0,2}(X,EndE)_{k-1},\nonumber \end{equation}
\begin{equation}\lambda(a'')\triangleq(a'',P(a'')). \nonumber \end{equation}
%where
%\begin{equation}Q_{A}=\{a''\in\Omega^{0,1}(X,EndE)_{k} \textrm{ }\big{|}\textrm{ }\|a''\|_{k} < \epsilon'',\textrm{ }
%a'' \textrm{ } \emph{satisfies}  \textrm{ } (\ref{equation 3}), \textrm{} (\ref{equation 5})  \}. \nonumber  \end{equation}
Note that
\begin{equation}\label{QA def}  Q_{A}\cap P^{-1}(0)=\{a'' \textrm{ }\big{|}\textrm{ }\| a''\|_{k} < \epsilon'',\textrm{ }
F^{0,2}(\overline{\partial}_{A}+a'')=0 ,\textrm{}  a'' \textrm{ }
\emph{satisfies}  \textrm{ } (\ref{equation 5}) \}    \end{equation}
gives a neighbourhood of $\overline{\partial}_{A}$ in $\mathcal{M}_{c}^{bdl}$ and
$Q_{A}\cap P^{-1}(0)\hookrightarrow \mathcal{M}_{A} $ embeds as a closed analytic subspace.
We are thus left to show $Q_{A}\cap P^{-1}(0)=\mathcal{M}_{A}$. \\
%This will be enough to set up the relation of analytic structures between $\mathcal{M}_{c}^{DT_{4}}$ and $\mathcal{M}_{c}^{bdl}$.

The image of the above map $\lambda$ satisfies
\begin{lemma}\label{image of lamda}
\begin{equation}Im(\lambda)\subseteq F, \nonumber \end{equation}
where
\begin{equation}  F=\left\{ \begin{array}{lll}
   (a'',\theta)\in Q_{A}\times\Omega^{0,2}_{k-1}\textrm{ }\big{|} & \overline{\partial}_{A}^{*}\theta=
   \overline{\partial}_{A}^{*}(*_{4}\overline{\partial}_{A}^{*}G([a'',\theta])), &
   \overline{\partial}_{A}^{*}(\overline{\partial}_{A}\theta+[a'',\theta])=0  %  leave space using \textrm{ }
\nonumber\end{array}\right\}.\end{equation}
\end{lemma}
\begin{proof}
Let
\begin{eqnarray*}\theta &=&\overline{\partial}_{A}a''+a''\wedge a'' \\
&=&  \overline{\partial}_{A}a''+P_{\overline{\partial}_{A}}(a''\wedge a'')+\mathbb{H}^{0,2}(a''\wedge a'')
+P_{\overline{\partial}_{A}^{*}}(a''\wedge a'').
\end{eqnarray*}
By definition
\begin{equation}a''\in Q_{A}\Rightarrow \overline{\partial}_{A}\alpha+P_{\overline{\partial}_{A}}(\alpha\wedge\alpha)
+*_{4}P_{\overline{\partial}_{A}^{*}}(\alpha\wedge\alpha)=0.
\nonumber \end{equation}
Hence
\begin{eqnarray}\label{bianchi 0}\theta-\mathbb{H}^{0,2}(a''\wedge a'')&=&P_{\overline{\partial}_{A}^{*}}(a''\wedge a'')-*_{4}P_{\overline{\partial}_{A}^{*}}(a''\wedge a'') \\
&=& \overline{\partial}_{A}^{*}\overline{\partial}_{A}G(a''\wedge a'')-*_{4}\overline{\partial}_{A}^{*}\overline{\partial}_{A}G(a''\wedge a'').
\end{eqnarray}
Taking $\overline{\partial}_{A}$ to both sides of $\theta=\overline{\partial}_{A}a''+a''\wedge a''$, we get
\begin{equation}\overline{\partial}_{A}\theta=\overline{\partial}_{A}(a''\wedge a'').\nonumber \end{equation}
Combined with the Bianchi identity $\overline{\partial}_{A}\theta+[a'',\theta]=0$, we have
\begin{equation}\label{bianchi 1}\overline{\partial}_{A}(a''\wedge a'')=-[a'',\theta]. \end{equation}
Using (\ref{bianchi 0}) and (\ref{bianchi 1}), we have
\begin{equation}\theta-\mathbb{H}^{0,2}(a''\wedge a'')=
\overline{\partial}_{A}^{*}\overline{\partial}_{A}G(a''\wedge a'')+*_{4}\overline{\partial}_{A}^{*}G([a'',\theta]). \nonumber \end{equation}
After taking $\overline{\partial}_{A}^{*}$, we get
\begin{equation}\overline{\partial}_{A}^{*}\theta=\overline{\partial}_{A}^{*}\big(*_{4}\overline{\partial}_{A}^{*}G([a'',\theta])\big).
\nonumber \end{equation}
\end{proof}
%$F$ is a finite dimensional smooth manifold.
\begin{lemma}\label{F iso to H1 H2}
We denote
\begin{equation}  F=\left\{ \begin{array}{lll}
   (a'',\theta)\in Q_{A}\times\Omega^{0,2}_{k-1}\textrm{ }\big{|} & \overline{\partial}_{A}^{*}\theta=
   \overline{\partial}_{A}^{*}(*_{4}\overline{\partial}_{A}^{*}G([a'',\theta])), &
   \overline{\partial}_{A}^{*}(\overline{\partial}_{A}\theta+[a'',\theta])=0  %  leave space using \textrm{ }
\nonumber\end{array}\right\}, \end{equation}
then the harmonic projection map
\begin{equation}(\mathbb{H}^{0,1}\times \mathbb{H}^{0,2}): F\rightarrow H^{0,1}(X,EndE)\times H^{0,2}(X,EndE) \nonumber \end{equation}
is a local analytic isomorphism.
\end{lemma}
\begin{proof}
We take a map $f$
\begin{equation}f:Q_{A}\times\Omega^{0,2}(X,EndE)_{k-1}\rightarrow Q_{A}\times\Omega^{0,2}(X,EndE)_{k-3}, \nonumber \end{equation}
\begin{equation}f(a'',\theta)=\bigg(a'',\mathbb{H}^{0,2}(\theta)+\overline{\partial}_{A}\overline{\partial}_{A}^{*}\big(\theta-*_{4}
\overline{\partial}_{A}^{*}G([a'',\theta])\big)+\overline{\partial}_{A}^{*}(\overline{\partial}_{A}\theta+[a'',\theta])\bigg).
\nonumber \end{equation}
It is easy to check that $f$ is a local analytic isomorphism by using implicit function theorem and noticing
\begin{equation}df_{0,0}(v_{1},v_{2})=(v_{1},\mathbb{H}^{0,2}v_{2}+\overline{\partial}_{A}\overline{\partial}_{A}^{*}v_{2}
+\overline{\partial}_{A}^{*}\overline{\partial}_{A}v_{2}) \nonumber\end{equation}
whose inverse is given by
\begin{equation}df_{0,0}^{-1}(u_{1},u_{2})=(u_{1},\mathbb{H}^{0,2}u_{2}+Gu_{2}). \nonumber \end{equation}
Hence $F=f^{-1}(Q_{A}\times H^{0,2}(X,EndE))$ and the projection map
\begin{equation}(\mathbb{H}^{0,1}\times \mathbb{H}^{0,2}) : F\rightarrow H^{0,1}(X,EndE)\times H^{0,2}(X,EndE)\nonumber\end{equation}
gives a local chart of $F$.
%Using Hodge decomposition theorem of $\varphi$. We have
%\begin{equation}\overline{\partial}_{A}\overline{\partial}_{A}^{*}\varphi=0\Rightarrow \overline{\partial}_{A}^{*}\varphi=0
%\nonumber \end{equation}
%which establish the relations between $f$ and $F$.
\end{proof}
%We use Sobolev norms $(L_{k}^{\alpha},  k\geq 2)$ in the following Lemma,
\begin{lemma}\label{vanishing of prW2}
If $\mathbb{H}^{0,2}(\theta)=0$ , $(a'',\theta)\in F$ and $\|a''\|_{k}\ll1$, then $\theta=0$.
%Here \begin{equation}\mathbb{H}^{0,2}: \Omega^{0,2}(X,EndE)_{k-1}\rightarrow H^{0,2}(X,EndE)\nonumber\end{equation} is the harmonic projection map.
\end{lemma}
\begin{proof}
By the Hodge decomposition and $(a'',\theta)\in F$,
\begin{eqnarray*}\theta&=&\mathbb{H}^{0,2}(\theta)+\overline{\partial}_{A}^{*}\overline{\partial}_{A}G\theta+
\overline{\partial}_{A}\overline{\partial}_{A}^{*}G\theta \\&=&-G\overline{\partial}_{A}^{*}([a'',\theta])+
G\overline{\partial}_{A}\overline{\partial}_{A}^{*}(*_{4}\overline{\partial}_{A}^{*}G([a'',\theta])),\nonumber\end{eqnarray*}
then
\begin{equation}
\|\theta\|_{k-1}\leq C_{1}\|a''\|_{k}\|\theta\|_{k-1}+C_{2}\| a''\|_{k}\|\theta\|_{k-1}
=C\| a''\|_{k}\|\theta\|_{k-1}.\nonumber\end{equation}
$C$ is a constant independent of $a'',\theta$. Hence we can get $\theta=0$ if $\|a''\|_{k}\ll1 $ .
\end{proof}
\begin{corollary}
The following three analytic spaces are set theoretically identical
\begin{equation}\mathcal{M}_{A}=\{a''\in\Omega^{0,1}(X,EndE)_{k} \textrm{ }\big{|}\textrm{ }\|a''\|_{k} < \epsilon'',\textrm{ } \mathbb{H}^{0,2}(a''\wedge a'')=0 ,\textrm{} a'' \textrm{ } \emph{satisfies}  \textrm{ } (\ref{equation 3}), \textrm{} (\ref{equation 5}) \}, \nonumber \end{equation}
\begin{equation}  Q_{A}\cap P^{-1}(0)=\{a'' \textrm{ }\big{|}\textrm{ }\|a''\|_{k} < \epsilon'',\textrm{ }
F^{0,2}(\overline{\partial}_{A}+a'')=0 ,\textrm{} a'' \textrm{ } \emph{satisfies}  \textrm{ } (\ref{equation 5}) \}, \nonumber \end{equation}
\begin{equation}Q_{A}\cap P^{-1}(0)=\left\{ \begin{array}{lll}& \|a''\|_{k} < \epsilon'', \textrm{ } a'' \textrm{ } \emph{satisfies}  \textrm{ } (\ref{equation 5}) \\ a'' \textrm{ }\Bigg{|} & \overline{\partial}_{A}a''+P_{\overline{\partial}_{A}}(a''\wedge a'')=0   \\ & \mathbb{H}^{0,2}(a''\wedge a'')=0 \nonumber\end{array}\right\}.\end{equation}
We use the same notation for the second and the third spaces because they are isomorphic as analytic spaces by the standard Kuranishi theory.
\end{corollary}
\begin{proof}
We only need to show $Q_{A}\cap P^{-1}(0)$ contains $\mathcal{M}_{A}$, i.e.
\begin{equation}\forall\textrm{ }\overline{\partial}_{A}+a''\in \mathcal{M}_{A} \Rightarrow F^{0,2}(\overline{\partial}_{A}+a'')=0. \nonumber \end{equation}
Since $\mathcal{M}_{A}$ is a subset of $Q_{A}$, we can apply Lemma \ref{image of lamda} to its image under the map $\lambda$.
Combined with Lemma \ref{vanishing of prW2}, we are done.
\end{proof}
Furthermore, we can identify the above three spaces as analytic spaces possibly with non-reduced structures.
Let us first recall a lemma due to Miyajima \cite{m}.
\begin{lemma}\label{miyajima lemma}(Miyajima \cite{m}). Let $E$, $G$ be Banach spaces with direct sum decomposition $E=F_{1}+F_{2}$. If a local analytic map
\begin{equation}h: E\rightarrow G
\nonumber\end{equation}
vanishes identically on $F_{2}$, then there exists a local analytic map
\begin{equation}f: E \rightarrow L(F_{2},G),
\nonumber \end{equation}
such that $h(t,s)=<f(t,s),s>$.
\end{lemma}
\begin{proposition}\label{QA intesect P=0 equals NA}
We have the following identification
\begin{equation}\label{Kuranishi theorem} Q_{A}\cap P^{-1}(0)=\mathcal{M}_{A} \nonumber \end{equation}
as analytic spaces possibly with non-reduced structures.
\end{proposition}
\begin{proof}
Obviously, $Q_{A}\cap P^{-1}(0)\hookrightarrow\mathcal{M}_{A}$, we are left to show that up to change of variables the analytic map $P$ can be expressed analytically in terms of $\mathbb{H}^{0,2}\circ P$ and coordinates of $Q_{A}$.

We consider $\lambda:Q_{A}\rightarrow F$, $\lambda(\alpha)=(\alpha,P(\alpha))$.
By Lemma \ref{image of lamda}, it is well defined. By Lemma \ref{QA iso to H1} and \ref{F iso to H1 H2},
we have the following commutative diagram
\begin{equation}
\xymatrix{\ar @{} [dr] |{} Q_{A} \ar@/^3pc/[rr]^{P(\alpha)=\overline{\partial}_{A}\alpha+\alpha\wedge \alpha}
\ar[d]^{\mathbb{H}^{0,1}}_{\wr\mid} \ar[r]^{\lambda(\alpha)=(\alpha,P(\alpha))} & F \ar[d]^{\mathbb{H}^{0,1}\times \mathbb{H}^{0,2}}_{\wr\mid} \ar[r]^{\pi_{2}(\alpha,\theta)=\theta}
& \Omega^{0,2}(X,EndE)_{k-1} \\
H^{0,1}(X,EndE) \ar[r]^{\lambda^{'}\quad \quad \quad \quad}
\ar@/_5pc/[urr]_{P^{'}(t)=\pi_{2}^{'}\circ \lambda^{'}(t)=\pi_{2}^{'}(t,\mathbb{H}^{0,2}\circ P(\alpha(t)))}
& H^{0,1}(X,EndE)\times H^{0,2}(X,EndE) \ar[ur]_{\quad \quad\pi_{2}^{'}}  }.
\nonumber \end{equation}
%We are left to show ideals generated by analytic maps $P$, $\mathbb{H}^{0,2}\circ P$ in the ring of analytic functions on $Q_{A}$ satisfies
%$\mathcal{I}_{P}\subseteq \mathcal{I}_{\mathbb{H}^{0,2}\circ P}$ as $\mathcal{I}_{P}\supseteq \mathcal{I}_{\mathbb{H}^{0,2}\circ P}$ is obvious.
With respect to local charts $(Q_{A},\mathbb{H}^{0,1})$ and $(F,\mathbb{H}^{0,1}\times \mathbb{H}^{0,2})$, $\lambda$ is expressed
by $\lambda^{'}=\big(t,\mathbb{H}^{0,2}\circ P(\alpha(t))\big)$, $\pi_{2}$ is expressed by $\pi_{2}^{'}$ and $P$ is expressed by $P^{'}$,
where $t\in H^{0,1}(X,EndE)$ and $\alpha(t)=(\mathbb{H}^{0,1})^{-1}(t)$.

By Lemma \ref{vanishing of prW2}, $\mathbb{H}^{0,2}(\theta)=0$ and $(t,\theta)\in F$ imply $\theta=0$. Hence $\pi_{2}^{'}(t,0)=0$.
%$\mathbb{H}^{0,2}\circ P(\alpha(t))=0\Rightarrow P(\alpha(t))=0$ which obviously implies $P^{'}(t)=0$.
Applying Lemma \ref{miyajima lemma}, we get
\begin{eqnarray*}P^{'}(t)&=&\pi_{2}^{'}\big(t,\mathbb{H}^{0,2}\circ P(\alpha(t))\big) \\
&=&<\eta\big(t,\mathbb{H}^{0,2}\circ P(\alpha(t))\big),\textrm{ }\mathbb{H}^{0,2}\circ P(\alpha(t))>
\end{eqnarray*}
for some local analytic map $\eta: H^{0,1}\times H^{0,2} \rightarrow L(H^{0,2},\Omega^{0,2}_{k-1})$, where $L(H^{0,2},\Omega^{0,2}_{k-1})$ is the
space of analytic maps between Banach spaces $H^{0,2}$ and $\Omega^{0,2}_{k-1}$.
%Thus $\mathcal{I}_{P}\subseteq\mathcal{I}_{\mathbb{H}^{0,2}\circ P}$.
\end{proof}
We have proved that $\mathcal{M}_{A}$ is isomorphic to an analytic neighbourhood of $\overline{\partial}_{A}$ in $\mathcal{M}_{c}^{bdl}$. We note that the above isomorphism is between real analytic spaces. We make the following definition generalizing the definition of Kuranishi maps.
\begin{definition}\label{real kuranishi map}
Given an integrable $(0,1)$-connection $\overline{\partial}_{A}$ with trivial isotropy subgroup, a real analytic map
\begin{equation}\kappa: H^{0,1}(X,EndE)\rightarrow H^{0,2}(X,EndE)  \nonumber \end{equation}
is a Kuranishi map of the moduli space of holomorphic bundles if there exists an open analytic neighbourhood $U_{A}$ of $\overline{\partial}_{A}$ in the moduli space such that $\kappa^{-1}(0)\cong U_{A}$ locally as real analytic spaces possibly with non-reduced structures.
\end{definition}
Finally, we get the following local Kuranishi model of $\mathcal{M}_{c}^{DT_{4}}$.
\begin{theorem}\label{Kuranishi str of cpx ASD thm}We assume $\mathcal{M}_{c}^{bdl}\neq\emptyset$ and
fix $d_{A}\in\mathcal{M}_{c}^{DT_{4}}$, then there exists a Kuranishi map $\tilde{\tilde{\kappa}}$ of $\mathcal{M}_{c}^{bdl}$ at $\overline{\partial}_{A}$ (the (0,1) part of $d_{A}$) such that $\tilde{\tilde{\kappa}}_{+}$ is a Kuranishi map of $\mathcal{M}_{c}^{DT_{4}}$ at $d_{A}$, where
\begin{equation} \xymatrix@1{
\tilde{\tilde{\kappa}}_{+}=\pi_{+}(\tilde{\tilde{\kappa}}): H^{0,1}(X,EndE)\cap B_{\epsilon} \ar[r]^{\quad \quad \quad \quad  \kappa}
& H^{0,2}(X,EndE)\ar[r]^{\pi_{+}} & H^{0,2}_{+}(X,EndE) },  \nonumber \end{equation}
$B_{\epsilon}$ is a small open ball containing the origin of the deformation space and $\pi_{+}$ is projection to self-dual two forms.
\end{theorem}
\begin{proof}
By Lemma \ref{identify SA MA}, $\mathcal{M}_{c}^{DT_{4}}$ is locally isomorphic to $\mathcal{M}_{A}^{+}$. From the definition of $\mathcal{M}_{A}$ (\ref{M_A}), $\mathcal{M}_{A}={\tilde{\tilde{\kappa}}}^{-1}(0)$,
%\begin{equation}\mathcal{M}_{A}={\tilde{\tilde{\kappa}}}^{-1}(0), \nonumber \end{equation}
where $\tilde{\tilde{\kappa}}: H^{0,1}(X,EndE)\cap B_{\epsilon}\rightarrow H^{0,2}(X,EndE) $ is
\begin{equation}\tilde{\tilde{\kappa}}(\alpha)=\mathbb{H}^{0,2}\big(q^{-1}(\alpha)\wedge q^{-1}(\alpha)\big), \nonumber\end{equation}
where
\begin{equation}q: \Omega^{0,1}(EndE)_{k}\rightarrow H^{0,1}(EndE)\oplus {\overline{\partial}_{A}^{*}\Omega^{0,1}(EndE)}_{k}\oplus
{\overline{\partial}_{A}^{*}\Omega^{0,2}(EndE)}_{k-1},  \nonumber \end{equation}
\begin{equation}q(a'')=\bigg(\mathbb{H}(a''),\overline{\partial}_{A}^{*}a''-\frac{i}{2}\wedge(a'\wedge a''+a''\wedge a'),
\overline{\partial}_{A}^{*}\big(\overline{\partial}_{A}a''+P_{\overline{\partial}_{A}}(a''\wedge a'')+*_{4}P_{\overline{\partial}_{A}^{*}}(a''\wedge a'')\big)\bigg). \nonumber \end{equation}
By Proposition \ref{QA intesect P=0 equals NA}, $\tilde{\tilde{\kappa}}$ is a Kuranishi map of $\mathcal{M}_{c}^{bdl}$.
Composing with $\pi_{+}$, $\mathcal{M}_{A}^{+}=\big(\pi_{+}\tilde{\tilde{\kappa}}\big)^{-1}(0) $.
\end{proof}
\begin{remark} ${}$ \\
%(1) It can be showed that the Kuranishi structure for corresponding complex self-dual equations can be described
%as $\big(\pi_{-}\tilde{\tilde{\kappa}}\big)^{-1}(0)$. \\
1. Under the assumption $\mathcal{M}_{c}^{bdl}\neq\emptyset$, we have a bijective map $\mathcal{M}_{c}^{bdl}\rightarrow \mathcal{M}_{c}^{DT_{4}}$.
The map can be enhanced to be a closed imbedding between analytic spaces by Theorem \ref{Kuranishi str of cpx ASD thm}. Note that this map is then also a homeomorphism between topological spaces. \\
2. For simplicity, we will always restrict to a small neighbourhood of the origin in $Ext^{1}(\mathcal{F},\mathcal{F})$
when we talk about a Kuranishi map $\kappa: Ext^{1}(\mathcal{F},\mathcal{F})\rightarrow Ext^{2}(\mathcal{F},\mathcal{F})$
for any coherent sheaf $\mathcal{F}$ and abbreviate $B_{\epsilon}$ in the notation from now on.
\end{remark}

\section{The compactification of $DT_{4}$ moduli spaces}
We come to the issue of compactification of $DT_{4}$ moduli spaces.
As Uhlenbeck, or generally Tian \cite{t} have shown, we need to consider connections with singularities supported on codimension 4 subspaces to compactify moduli spaces of holomorphic HYM connections. This becomes very difficult when the real dimension of the underlying manifold is bigger than 4. Even if one could compactify it, as Tian showed in his paper,
one still does not understand the local analytic structure of the compactified moduli space very well.

Instead of using this compactification, our attempted approach here is algebro-geometric, using Gieseker moduli spaces of semi-stable sheaves.

\subsection{The stable bundles compactification of $DT_{4}$ moduli spaces}
In this subsection, assuming $\overline{\mathcal{M}}_{c}\neq\emptyset$ consists of slope-stable bundles only, we prove that $\mathcal{M}_{c}^{DT_{4}}$ is compact.

We take a connection on $E$ with curvature $F$. By Chern-Weil theory, we have
%\begin{equation} Tr(F^{2})=-8\pi^{2}ch_{2}(E). \nonumber \end{equation}
%Then
\begin{equation}
-8\pi^{2}\int ch_{2}(E)\wedge\Omega=\int Tr(F^{0,2}\wedge F^{0,2})\wedge\Omega \nonumber\end{equation}
\begin{equation}=\int Tr(F^{0,2}_{+}\wedge F^{0,2}_{+})\wedge\Omega+\int Tr(F^{0,2}_{-}\wedge F^{0,2}_{-})\wedge\Omega\nonumber\end{equation}
\begin{equation}+\int Tr(F^{0,2}_{+}\wedge F^{0,2}_{-})\wedge\Omega+\int Tr(F^{0,2}_{-}\wedge F^{0,2}_{+})\wedge\Omega\nonumber\end{equation}
%\begin{equation}=\int Tr(F^{0,2}_{+}\wedge *_{4}F^{0,2}_{+})\wedge\Omega-\int Tr(F^{0,2}_{-}\wedge *_{4}F^{0,2}_{-})\wedge\Omega\nonumber\end{equation}
%\begin{equation}+\int Tr(F^{0,2}_{+}\wedge F^{0,2}_{-})\wedge\Omega+\int Tr(F^{0,2}_{-}\wedge F^{0,2}_{+})\wedge\Omega \nonumber\end{equation}
\begin{equation}\label{plus norm equal minus norm}
=\int\mid F^{0,2}_{+}\mid^{2}\wedge\Omega\wedge\overline{\Omega} -\int\mid
F^{0,2}_{-}\mid^{2}\wedge\Omega\wedge\overline{\Omega}+\int\sqrt{-1}\chi\wedge\Omega\wedge\overline{\Omega},\end{equation}
where $\chi$ is some real valued function.
\begin{lemma}(Lewis \cite{lewis})\label{C2 condition}
If $ch_{2}(E)\in H^{2,2}(X,\mathbb{C})$ or has no component of type $(0,4)$,
then $F^{0,2}_{+}=0$ implies $F^{0,2}=0$.
\end{lemma}
\begin{proof}
Note that $\Omega$ is $(4,0)$ form and $\chi$ is a real valued function.
\end{proof}
\begin{corollary} \label{cptness by stable bdl}
If $\overline{\mathcal{M}}_{c}=\mathcal{M}_{c}^{bdl}\neq\emptyset$,
then $\mathcal{M}_{c}^{DT_{4}}$ is compact.
\end{corollary}
\begin{proof}
By the assumptions and Lemma \ref{C2 condition}.
\end{proof}
From the viewpoint of local Kuranishi models (i.e. Theorem \ref{Kuranishi str of cpx ASD thm}), Lemma \ref{C2 condition} says
\begin{equation}\pi_{+}\tilde{\tilde{\kappa}}=0 \Rightarrow \tilde{\tilde{\kappa}}=0, \nonumber\end{equation}
which gives restrictions to $\tilde{\tilde{\kappa}}$.
\begin{proposition} Given a map
\begin{equation}\kappa: H^{0,1}(X,EndE)\rightarrow H^{0,2}(X,EndE) \nonumber\end{equation}
such that $\kappa_{+}=0 \Rightarrow \kappa=0$ and $\kappa(0)=0$, where
\begin{equation}\kappa_{+}=\pi_{+}\circ\kappa : H^{0,1}(X,EndE)\rightarrow H^{0,2}_{+}(X,EndE), \nonumber\end{equation}
then the image of $\kappa$ can not be a neighbourhood of the origin.
\end{proposition}
\begin{proof}
By assumptions, $\kappa(U(0))\cap H^{0,2}_{-}(X,EndE)=\{0\}$.
% where $U(0)$ is a neighbourhood of the origin in $H^{0,1}(X,EndE)$.
\end{proof}
%\begin{remark}
%In fact, we can find local analytic maps such that $rank(\kappa)=ext^{2}(E,E)-1$
%(but so far we do not know any example of $\mathcal{M}_{c}$ whose local analytic structure is of this type).
%Since $\mathcal{M}_{c}^{bdl}\cong\mathcal{M}_{c}^{DT_{4}}$ as sets, the dimension of
%$\mathcal{M}_{c}^{DT_{4}}$ may be less than its virtual dimension.
%Hence it is possible that some local parts of $\mathcal{M}_{c}^{DT_{4}}$ do not contribute to the later defined $DT_{4}$ invariants.
%\end{remark}

\subsection{The attempted general compactification of $DT_{4}$ moduli spaces}
In this subsection, we propose an attempted approach to the general compactification of $\mathcal{M}_{c}^{DT_{4}}$. Under the gluing assumptions, we define the generalized $DT_{4}$ moduli space $\overline{\mathcal{M}}_{c}^{DT_{4}}$ by gluing local models.
We then show in some cases we can get rid of the gluing assumption and prove the existence of $\overline{\mathcal{M}}_{c}^{DT_{4}}$. \\

We recall that if we assume
$\mathcal{M}_{c}^{bdl}\neq\emptyset$, we have a homeomorphism
\begin{equation}\mathcal{M}_{c}^{bdl}\rightarrow \mathcal{M}_{c}^{DT_{4}},  \nonumber \end{equation}
which is a closed imbedding between analytic spaces possibly with non-reduced structures.
The idea of general compactification is to extend the above map to a homeomorphism
\begin{equation}\mathcal{M}_{c}\rightarrow \overline{\mathcal{M}}_{c}^{DT_{4}},  \nonumber \end{equation}
where $\overline{\mathcal{M}}_{c}^{DT_{4}}$ comes from gluing local models, i.e. locally at a stable sheaf $\mathcal{F}$, $\overline{\mathcal{M}}_{c}^{DT_{4}}$ is $\kappa_{+}^{-1}(0)$, where
\begin{equation}\kappa_{+}=\pi_{+}\circ\kappa:Ext^{1}(\mathcal{F},\mathcal{F})\rightarrow Ext^{2}_{+}(\mathcal{F},\mathcal{F}),
\nonumber \end{equation}
$\kappa$ is a Kuranishi map of $\mathcal{M}_{c}$ at $\mathcal{F}$ and
$\pi_{+}$ is the projection map. \\

However, the Kuranishi map $\kappa$ is unique only up to change of variables.
Meanwhile, the $*_{4}$ is a real operator and if we use different re-parametrization, the resulting models may be different in general, i.e.
\begin{equation}(\pi_{+}\circ\kappa_{1})^{-1}(0)\ncong (\pi_{+}\circ\kappa_{2})^{-1}(0)  \nonumber \end{equation}
for different $\kappa_{i}$, $i=1, 2$.

For the purpose of gluing, we need to pick a coherent choice of local Kuranishi models for $\mathcal{M}_{c}$. In the case when
$\mathcal{M}_{c}=\mathcal{M}_{c}^{bdl}$, the moment map equation in $DT_{4}$ equations (\ref{complex ASD equation}) gives such a choice.
In general, we need a similar moment map equation for $\mathcal{M}_{c}$. This is achieved by a quiver representation of $\mathcal{M}_{c}$
due to \cite{bchr}. We then proposed a candidate local model at each $\mathcal{F}\in \mathcal{M}_{c}$ based on their work. As we do not know how to glue such local models at the moment, we will not discuss them in detail. The interested reader could refer to the appendix of the first named author's master thesis \cite{cao}.  \\

We will always make the following assumptions if $\mathcal{M}_{c}\neq\mathcal{M}_{c}^{bdl}$.
\begin{assumption}\label{assumption on gluing}
We assume there exists a real analytic space $\overline{\mathcal{M}}^{DT_{4}}_{c}$ and a homeomorphism
\begin{equation}\mathcal{M}_{c}\rightarrow\overline{\mathcal{M}}^{DT_{4}}_{c}  \nonumber \end{equation}
such that at each closed point of $\mathcal{M}_{c}$, say $\mathcal{F}$, $\overline{\mathcal{M}}^{DT_{4}}_{c}$ is locally isomorphic to $\kappa_{+}^{-1}(0)$, where
\begin{equation}\kappa_{+}=\pi_{+}\circ\kappa:Ext^{1}(\mathcal{F},\mathcal{F})\rightarrow Ext^{2}_{+}(\mathcal{F},\mathcal{F}),
\nonumber \end{equation}
$\kappa$ is a Kuranishi map at $\mathcal{F}$
and $Ext^{2}_{+}(\mathcal{F},\mathcal{F})$ is a half dimensional real subspace of $Ext^{2}(\mathcal{F},\mathcal{F})$ on which the Serre duality quadratic form is real and positive definite.
\end{assumption}
\begin{definition}\label{generalized DT4}
Under the Assumption \ref{assumption on gluing}, we obtain a real analytic space $\overline{\mathcal{M}}^{DT_{4}}_{c}$ which is compact if $\overline{\mathcal{M}}_{c}=\mathcal{M}_{c}$. We call it the
generalized $DT_{4}$ moduli space.
\end{definition}
\begin{remark}In \cite{bj} Borisov and Joyce used local 'Darboux charts' in the sense of Brav, Bussi and Joyce \cite{bbj}, the machinery of homotopical algebra and $C^{\infty}$-algebraic geometry to construct a compact derived $C^{\infty}$-scheme with the same underlying topological structure as the Gieseker moduli space of stable sheaves in general. In our language, this $C^{\infty}$-scheme is the $C^{\infty}$-scheme (instead of the real analytic space) version of the hoped generalized $DT_{4}$ moduli space. Thus generalized $DT_{4}$ moduli spaces always exist at least as $C^{\infty}$-schemes by their gluing result.
\end{remark}
We show in several good cases, generalized $DT_{4}$ moduli spaces exist as real analytic spaces. The first case is
when $\mathcal{M}_{c}$ is smooth.
\begin{proposition}\label{gene DT4 if Mc smooth}
If the Gieseker moduli space $\mathcal{M}_{c}$ is smooth, the generalized $DT_{4}$ moduli space exists and
$\overline{\mathcal{M}}^{DT_{4}}_{c}\cong\mathcal{M}_{c}$ as real analytic spaces.
\end{proposition}
\begin{proof}
By the assumption, all Kuranishi maps are zero. The conclusion is obvious.
\end{proof}
There is another interesting case when we can get $\overline{\mathcal{M}}^{DT_{4}}_{c}$ as a real analytic space via gluing.
\begin{proposition}\label{ob=v+v*}
We assume for any closed point $\mathcal{F}\in \mathcal{M}_{c}$, there is a splitting of obstruction space
\begin{equation}Ext^{2}(\mathcal{F},\mathcal{F})=V_{\mathcal{F}}\oplus V_{\mathcal{F}}^{*}  \nonumber \end{equation}
such that $V_{\mathcal{F}}$ is its maximal isotropic subspace with respect to the Serre duality pairing and the image of a Kuranishi map $\kappa$ at $\mathcal{F}$ satisfies
\begin{equation}\emph{Image}(\kappa)\subseteq V_{\mathcal{F}}.  \nonumber \end{equation}
Then the generalized $DT_{4}$ moduli space exists and $\overline{\mathcal{M}}^{DT_{4}}_{c}\cong\mathcal{M}_{c}$ as real analytic spaces.
\end{proposition}
\begin{proof}
We pick a Hermitian metric $h$ on $V_{\mathcal{F}}$ which induces a Hermitian metric on $V_{\mathcal{F}}^{*}$. We abuse the notation $h$ for the direct sum Hermitian metric on $Ext^{2}(\mathcal{F},\mathcal{F})=V_{\mathcal{F}}\oplus V_{\mathcal{F}}^{*}$. We define $*_{4}: Ext^{2}(\mathcal{F},\mathcal{F})\rightarrow Ext^{2}(\mathcal{F},\mathcal{F})$ such that $Q_{Serre}(\alpha,*_{4}\beta)=h(\alpha,\beta)$, where $Q_{Serre}$ denotes the Serre duality pairing. Then for $\kappa(\alpha)\in V_{\mathcal{F}}$, we have $*_{4}\big(\kappa(\alpha)\big)\in V_{\mathcal{F}}^{*}$ which implies $\kappa_{+}=0\Rightarrow \kappa=0$ by the assumption $\emph{Image}(\kappa)\subseteq V_{\mathcal{F}}$.
\end{proof}
\begin{remark} We will see the above conditions are satisfied for compactly supported sheaves on certain local $CY_{4}$ manifolds.
\end{remark}

\section{Virtual cycle constructions}
In algebraic geometry, we can use GIT to construct moduli spaces.
If one wants to define invariants associated to them, we need to make sense of their fundamental classes.
However, because of the lack of transversality, moduli spaces are in general very singular and not of expected dimensions.
The way to obtain correct fundamental cycles (deformation invariant) originated from the idea of Fulton-MacPherson's localized top Chern class \cite{fulton}. It is generalized to Fredholm Banach bundles over Banach manifolds by many people (such as Brussee \cite{brussee}, Cieliebak-Mundet i Riera-Salamon \cite{cieliebak} in the equivariant case) and developed in moduli problems by Li-Tian \cite{lt1}, Behrend-Fantechi \cite{bf} in full generality. The equivalence of these works were proved in \cite{lt3}, \cite{kim}.

%The idea is to imbed the moduli space into a smooth ambient space (maybe infinite dimension) as a zero loci of some section of some natural bundle.
%Then we generically perturb the section to define the Euler class of that bundle.
%The correct cycle one may want to take is the Poincar\'{e} dual of that Euler class.
%However, the moduli space can only be imbedded into a finite dimensional smooth space locally in general.
%Then one may allow the ambient space to be a infinite dimension one (usually a canonical one with nice local structures),
%this is the idea of \cite{lt2}. If one does not want to encounter infinite dimensional stuff,
%one way of achieving this in algebro-geometric setting is to study the deformation-obstruction theory of the moduli problem.
%If the obstruction theory is perfect in the sense of Li-Tian \cite{lt1},
%one can construct a global cone of equal dimension over the moduli space sitting inside the locally free obstruction bundle.
%Intersecting the cone with the zero section of the obstruction bundle gives the virtual fundamental class of the moduli space.
%The equivalence of the above two approaches was proved in \cite{lt3}.
%Meanwhile, Behrend-Fantachi \cite{bf} studied the cotangent complex of the moduli space directly.
%They constructed a zero dimensional cone stack called intrinsic normal cone inside the intrinsic normal sheaf of the cotangent complex.
%Under perfectness assumption too, they defined the virtual fundamental class of the moduli space which turned out to be equivalent
%to Li-Tian's construction \cite{kim}.

\subsection{The virtual cycle construction of $DT_{4}$ moduli spaces }
We start with
\begin{equation}\label{Fredholm bundle}
\begin{array}{lll}
      & \mathcal{E} =& \mathcal{A^{*}}\times_{\mathcal{G}^{0}}(\Omega^{0}(X,g_{E})_{k-1}\oplus\Omega^{0,2}_{+}(X,EndE)_{k-1})
      \\  & \quad  & \qquad \downarrow \\ \mathcal{M}^{DT_{4}}_{c} & \hookrightarrow & \mathcal{B}^{*}=\mathcal{A^{*}}/\mathcal{G}^{0},
      %  leave space using \textrm{ }
\end{array}\end{equation}
where $\mathcal{M}^{DT_{4}}_{c}\hookrightarrow \mathcal{B}^{*}$ embeds as the zero loci of section $s=(\wedge F,F^{0,2}_{+})$ of $\mathcal{E}$.
We assume $\overline{\mathcal{M}}_{c}=\mathcal{M}_{c}^{bdl}\neq\emptyset$ to get compactness of $\mathcal{M}_{c}^{DT_{4}}$.
\begin{remark} The orientability of the Banach bundle is proved when $H^{odd}(X,\mathbb{Z})=0$ in Theorem \ref{orientablity theorem}. Then we can
choose an orientation data $o(\mathcal{L})$ (Definition \ref{ori data}) in this case.
\end{remark}
We check the Fredholm property of the above Banach bundle.
\begin{lemma}
The above Banach bundle $\mathcal{E}\rightarrow \mathcal{B}^{*}$ is a Fredholm bundle.
\end{lemma}
\begin{proof}
We take an open cover $\{U_{i}\}$ of $s^{-1}(0)$ in $\mathcal{A}^{*}/\mathcal{G}^{0}$, where
\begin{equation}U_{i}=\{d_{A_{i}}+a \textrm{ } \big{|} \textrm{ }  \|a\|_{k}< \epsilon ,  \textrm{ }  d_{A_{i}}^{*}a=0 \}. \nonumber\end{equation}
Note that
\begin{equation}E\mid_{U_{i}}=U_{i}\times(\Omega^{0}(X,g_{E})_{k-1}\oplus\Omega^{0,2}_{+}(X,EndE)_{k-1}). \nonumber\end{equation}
On the intersection of two charts, we have a commutative diagram
\begin{equation}\xymatrix{
  E\mid_{U_{i}} \ar[d]_{\pi_{i}} \ar[r]^{\Phi_{ij}} & E\mid_{U_{j}} \ar[d]^{\pi_{j}} \\
  U_{i} \ar[r]^{\phi_{ij}} & U_{j},  }
\nonumber\end{equation}
where $\phi_{ij}$ is the gauge transformation on $U_{ij}$, and $\Phi_{ij}$ is the adjoint action in the fiber direction.

The section $s$ near $d_{A}$ with $\overline{\partial}_{A}^{2}=0$ is given by
\begin{equation}  \mathcal{G}^{0}\curvearrowright \Omega^{1}(X,g_{E})_{k}\rightarrow \Omega^{0}(X,g_{E})_{k-1}\oplus\Omega^{0,2}_{+}(X,EndE)_{k-1},
\nonumber\end{equation}
\begin{equation}a=a^{0,1}+a^{1,0}\mapsto (\wedge F(d_{A}+a^{0,1}+a^{1,0}), F^{0,2}_{+}(\overline{\partial}_{A}+a^{0,1})),
\nonumber \end{equation}
where we identify unitary connections with $(0,1)$ connections $\Omega^{1}(X,g_{E})_{k}\cong\Omega^{0,1}(X,EndE)_{k}$.
After gauge fixing, we get $ker(d_{A}^{*})\subseteq\Omega^{1}(X,g_{E})_{k}$.
By the K\"ahler identity, $[\wedge,d_{A}]=i(\overline{\partial}_{A}^{*}-{\partial}_{A}^{*})$, we have
\begin{equation}ker(ds)\mid_{A}\cong H^{0,1}(X,EndE),\nonumber\end{equation}
\begin{eqnarray*}coker(ds)\mid_{A}&=&\frac{\Omega^{0}(X,g_{E})_{k-1}\oplus
\Omega^{0,2}_{+}(X,EndE)_{k-1}}{(i\overline{\partial}_{A}^{*}a^{0,1}-i\partial^{*}a^{1,0})\oplus\overline{\partial}_{A}^{+}a^{0,1}} \\
&=&\frac{\Omega^{0}(X,g_{E})_{k-1}\oplus \Omega^{0,2}_{+}(X,EndE)_{k-1}}{(2i\overline{\partial}_{A}^{*}a^{0,1})\oplus\overline{\partial}_{A}^{+}a^{0,1}} \\
&=& H^{0}(X,g_{E})\oplus H^{0,2}_{+}(X,EndE), \nonumber\end{eqnarray*}
where the second equality uses $d_{A}^{*}(a)=\overline{\partial}_{A}^{*}a^{0,1}+\partial_{A}^{*}a^{1,0}=0 $.
%Thus we have proved the Fredholm property of $s$ on $s^{-1}(0)$.
\end{proof}
Then by Proposition 14 of \cite{brussee}, the Euler class of the above Fredholm Banach bundle $e([s:\mathcal{B}^{*}\rightarrow\mathcal{E}])$ exists. We define it to be the virtual fundamental class of $\mathcal{M}^{DT_{4}}_{c}$.
\begin{definition}\label{virtual cycle when Mc=Mbdl}
We assume $\overline{\mathcal{M}}_{c}=\mathcal{M}_{c}^{bdl}\neq\emptyset$ and there exists an orientation data $o(\mathcal{L})$, then the virtual fundamental class of $\mathcal{M}^{DT_{4}}_{c}$ ($DT_{4}$ virtual cycle for short) is the Euler class of the above oriented Fredholm Banach bundle with $[\mathcal{M}^{DT_{4}}_{c}]^{vir}\in H_{r}(\mathcal{B}^{*},\mathbb{Z})$, where $r=2-\chi(E,E)$ is the virtual dimension.
\end{definition}
\begin{remark}\label{framed DT4}
One can similarly define the $DT_{4}$ virtual cycle as a homology class $[\widetilde{\mathcal{M}^{DT_{4}}_{c}}]^{vir}\in H_{*}(\widetilde{\mathcal{B}}^{*})$. Here, $\widetilde{\mathcal{M}^{DT_{4}}_{c}}$ is the framed $DT_{4}$ moduli space, i.e.
the zero loci of section $\tilde{s}=(\wedge F,F^{0,2}_{+})$ of the following Banach bundle
\begin{equation}
\begin{array}{lll}
      & \widetilde{\mathcal{E}} =& (\mathcal{A^{*}}\times Hom(U(r),P_{x_{0}}))\times_{\mathcal{G}}(\Omega^{0}(X,g_{E})_{k-1}\oplus\Omega^{0,2}_{+}(X,EndE)_{k-1})
      \\  & \quad  & \qquad \downarrow \\ \widetilde{\mathcal{M}^{DT_{4}}_{c}} & \hookrightarrow & \widetilde{\mathcal{B}}^{*}=(\mathcal{A^{*}}\times Hom(U(r),P_{x_{0}}))/\mathcal{G},\end{array}\nonumber \end{equation}
where $x_{0}\in X$ is a base-point, $P\rightarrow X$ is the principal $U(r)$-bundle associated with $E\rightarrow X$ and $\widetilde{\mathcal{B}}^{*}$ is the space of framed irreducible connections which admits a base-point $PU(r)$-fibration $\beta:\widetilde{\mathcal{B}}^{*}\rightarrow \mathcal{B}^{*}$ \cite{dk}. Furthermore, we have a $PU(r)$-bundle map $\widetilde{\mathcal{E}}\rightarrow \mathcal{E}$ covering $\beta$. Thus we can choose the representative sub-manifold of $[\widetilde{\mathcal{M}^{DT_{4}}_{c}}]^{vir}$ in $\widetilde{\mathcal{B}}^{*}$ to be $PU(r)$ equivariant and the free $PU(r)$ quotient will represent  $[\mathcal{M}^{DT_{4}}_{c}]^{vir}\in H_{r}(\mathcal{B}^{*},\mathbb{Z})$ \cite{cieliebak}.
\end{remark}
${}$ \\
\textbf{The deformation invariance}.
We show $[\mathcal{M}^{DT_{4}}_{c}]^{vir}$ is independent of the choice of \\
$(1)$ the holomorphic top form $\Omega$, \\
$(2)$ the Hermitian metric $h$ on $E$,  \\
$(3)$ the parameter $t$ of any deformation of complex structures $X_{t}$ when
$\overline{\mathcal{M}}_{c}=\mathcal{M}_{c}^{bdl}\neq\emptyset$ for all $X_{t}$.
%\begin{remark}
%We will show in detail the independence of $(1)$ here, $(2)$ and $(3)$ can be proved similarly. \\
%As $\Omega^{0}(X,g_{E})_{k-1}$ is independent of some choices, we will sometimes omit them for simplicity.
%\end{remark}
\begin{lemma}
$[\mathcal{M}^{DT_{4}}_{c}]^{vir}$ is independent of the choice of $\Omega$ and $h$.
\end{lemma}
%\textbf{Independence of $\Omega$}.
\begin{proof}
%We fix a Hermitian metric $h$, a complex structure of $X$ and an orientation data $o(\mathcal{L})$.
We choose two holomorphic top forms $\Omega$, $e^{i\theta}\Omega$ which give $*$, $*_{1}=e^{i\theta}*$ respectively
(we use $*$ to denote $*_{4}$ for simplicity here). As $\Omega^{0}(X,g_{E})_{k-1}$ is independent of the choice of $\Omega$, we omit it in the expression of the Banach bundle $\mathcal{E}$. There exists a bundle isomorphism
\begin{equation}
\xymatrix{
  \Omega^{0,2}_{+}(X,EndE)_{k-1} \ar[rr]^{f_{1}} \ar[dr]_{\pi}
                &  &    \Omega^{0,2}_{+_{1}}(X,EndE)_{k-1} \ar[dl]^{\pi_{1}}    \\
                & \mathcal{B}^{*}               }   \nonumber\end{equation}
where $f_{1}$ is fiberwise multiplication by $\frac{1}{2}(cos\theta+1+sin\theta\sqrt{-1})$ if $\theta\neq\pi$ (if $\theta=\pi$, $f_{1}$ is
defined to be multiplication by $\sqrt{-1}$).
We denote $s_{1}'=f_{1}^{-1}\circ s_{1}$, where $s_{1}=F^{0,2}_{+_{1}}$ is the complex ASD equation with respect to $e^{i\theta}\Omega$.
%By similar arguments as before, we know $s_{1}': \mathcal{B}^{*}\rightarrow\Omega^{0,2}_{+}(X,EndE)_{k-1}$ is also a Fredholm Banach bundle when adding back $\Omega^{0}(X,g_{E})_{k-1}$ with the moment map section.
By the functorial property of the Euler class, we are reduced to prove
\begin{equation}e([s_{1}':\mathcal{B}^{*}\rightarrow\Omega^{0,2}_{+}(X,EndE)_{k-1}])=e([s :\mathcal{B}^{*}\rightarrow \Omega^{0,2}_{+}(X,EndE)_{k-1}]), \nonumber \end{equation}
where $s=F^{0,2}_{+}$ is the complex ASD equation with respect to $\Omega$.

We consider a family of sections
\begin{equation}s_{t}': \mathcal{B}^{*}\rightarrow \Omega^{0,2}_{+}(X,EndE)_{k-1}, \nonumber \end{equation}
\begin{equation}s_{t}'=f_{t}^{-1}\circ s_{t}\triangleq\big(\frac{1}{2}(\sqrt{1-t^{2}sin^{2}\theta}+1+t\cdot sin\theta\sqrt{-1})\big)^{-1}\cdot s_{t}
\nonumber\end{equation}
and $s_{t}=F^{0,2}_{+_{t}}$ is the complex ASD equation with respect to
$(\sqrt{1-t^{2}sin^{2}\theta}+t\cdot sin\theta\sqrt{-1})\cdot\Omega$.

It is easy to check we have the following commutative relation
\begin{equation}f_{t}\circ *\circ F^{0,2}=*_{t}\circ f_{t}\circ F^{0,2}, \nonumber \end{equation}
where $*_{t}\triangleq(\sqrt{1-t^{2}sin^{2}\theta}+t\cdot sin\theta\sqrt{-1})\circ*$. Then, we get
\begin{equation}s_{t}'=f_{t}^{-1}\circ s_{t}=f_{t}^{-1}\circ \pi_{+_{t}}F^{0,2}=\pi_{+}\circ(f_{t}^{-1}F^{0,2}),\nonumber \end{equation}
which connects $s_{0}'=s$ and $s_{1}'$.
We define
\begin{equation}S: \mathcal{B}^{*}\times[0,1]\rightarrow \mathcal{A^{*}}\times_{\mathcal{G}^{0}}(\Omega^{0}(X,g_{E})_{k-1}\oplus\Omega^{0,2}_{+}(X,EndE)_{k-1}),
\nonumber \end{equation}
\begin{equation}S(A,t)=(\wedge F(A),s_{t}'), \nonumber \end{equation}
which is an oriented Fredholm Banach bundle of index $r+1$ with $S|_{\mathcal{B}^{*}\times 0}=s$
and $S|_{\mathcal{B}^{*}\times 1}=s_{1}'$.
Note that, by topological reasons, the above family miss the case when $*_{1}=-*$
which can be covered by moving $*$ in the $S^{1}$ family.
Then by \cite{brussee}, we are done.

The space of all Hermitian metrics on $E$ is connected, the above argument goes through.
%(actually we do not need the explicit expression of the isomorphism $f_{t}$ as stated above).
The only difference is we also need to identify $\mathcal{B}^{*}$ for different choices of $h$ which is standard.
\end{proof}
Similarly, we have
\begin{lemma}
$[\mathcal{M}^{DT_{4}}_{c}]^{vir}$ is a deformation invariant of $X$.
\end{lemma}
\begin{proof}
We fix a Hermitian metric, a continuous deformations of complex structures $J_{t}$ of $X$ and an orientation data $o(\mathcal{L})$ (it does not depend on $t$). We consider Fredholm Banach bundle
\begin{equation}s: \mathcal{B}^{*}\times [0,1]\rightarrow \mathcal{A^{*}}\times_{\mathcal{G}^{0}}(\Omega^{0}(X,g_{E})_{k-1}\oplus\Omega^{0,2}_{+}(X,EndE)_{k-1}),
\nonumber \end{equation}
\begin{equation}s_{t}=(\wedge F, f_{t}^{-1}\circ F^{0,2}_{+_{t}}),\nonumber \end{equation}
where $*_{t}$ is the $*_{4}$ operator with respect to the holomorphic structure $J_{t}$ and
\begin{equation}f_{t}:\Omega^{0,2}_{+}(X,EndE)_{k-1}\rightarrow \Omega^{0,2}_{+_{t}}(X,EndE)_{k-1}\nonumber \end{equation}
is a Banach bundle isomorphism which commutes with the adjoint action of $\mathcal{G}$.
$f_{t}$ exists because the complex structure only affects the differential forms part of the underling manifold, not the topological bundle,
while the unitary gauge transformations act on bundle $E$ only.

We have
\begin{equation}f_{t}(* F^{0,2})=*_{t}f_{t}(F^{0,2})\nonumber \end{equation}
by extending $f_{t}$ to $\Omega^{0,2}_{-}(X,EndE)_{k-1}$ using
$f_{t}(\sqrt{-1}\alpha)\triangleq\sqrt{-1}f_{t}(\alpha)$,
where $\alpha\in \Omega^{0,2}_{+}(X,EndE)_{k-1}$. Then by \cite{brussee}, we prove the deformation invariance.
\end{proof}
To sum up, we have the following result.
\begin{theorem}\label{main theorem}
We assume $\overline{\mathcal{M}}_{c}=\mathcal{M}_{c}^{bdl}\neq\emptyset$ and there exists an orientation data $o(\mathcal{L})$.
Then $\mathcal{M}^{DT_{4}}_{c}$ is compact and its virtual fundamental class exists as a cycle
$[\mathcal{M}^{DT_{4}}_{c}]^{vir}\in H_{r}(\mathcal{B}^{*},\mathbb{Z})$, where $r=2-\chi(E,E)$ is the virtual dimension.

Furthermore, if the above assumptions are satisfied by a continuous family of Calabi-Yau 4-folds $X_{t}$ parameterized by $t\in [0,1]$, then the cycle in $H_{r}(\mathcal{B}^{*},\mathbb{Z})$ is independent of $t$.
\end{theorem}
\begin{remark}\label{remark on spin(7) instantons} ${}$ \\
1. The Banach manifold $\mathcal{B}^{*}=\mathcal{A}^{*}/\mathcal{G}^{0}$ involves a choice of a large integer $k$ in $L_{k}^{2}$ norm completion. As stated before, the $DT_{4}$ moduli space is independent of the choice of $k$. Meanwhile, the homotopy-invariant properties of
$\mathcal{B}^{*}$ are insensitive to $k$ \cite{dk} and it is easy to show the virtual fundamental class does not depend on the choice of $k$. \\
2. By Donaldson-Thomas \cite{dt}, the Calabi-Yau 4-fold $X$ is also a $Spin(7)$ manifold ($SU(4)\subset Spin(7)$) and
\begin{equation} \Omega^{2}(X)=\Omega^{2}_{7}(X)\oplus\Omega^{2}_{21}(X),
\nonumber\end{equation}
\begin{equation} \Omega^{2}(X)\otimes_{\mathbb{R}}\mathbb{C}=\Omega^{1,1}_{0}(X)\oplus\Omega^{0,0}(X)<\omega>\oplus\Omega^{0,2}(X)\oplus\Omega^{2,0}(X).
\nonumber\end{equation}
Coupled with bundles, the deformation complex of $Spin(7)$ instantons \cite{lewis}
\begin{equation}\Omega^{0}(X,g_{E})\rightarrow\Omega^{1}(X,g_{E})\rightarrow\Omega^{2}_{7}(X,g_{E}) \nonumber \end{equation}
is the same as
\begin{equation}\Omega^{0}(X,g_{E})\rightarrow\Omega^{0,1}(X,EndE)\rightarrow\Omega^{0,2}_{+}(X,EndE)\oplus\Omega^{0}(X,g_{E}). \nonumber \end{equation}
Correspondingly, the $Spin(7)$ instanton equation
\begin{equation}\pi_{7}(F)=0
\nonumber
\end{equation}
is equivalent to $DT_{4}$ equations (\ref{complex ASD equation}), where
\begin{equation}\pi_{7}: \Omega^{2}(X,g_{E})\rightarrow \Omega^{2}_{7}(X,g_{E}) \nonumber \end{equation}
is the projection map.
Thus $Spin(7)$ instanton counting is just the $DT_{4}$ invariant (defined later) when the base manifold is a Calabi-Yau 4-fold.
\end{remark}
${}$ \\
\textbf{The $\mu_{1}$-map}.
Because the virtual dimension of $\mathcal{M}_{c}^{DT_{4}}$ is not zero in general, we need the $\mu$-map to cut down the dimension and define
invariants.

We recall \cite{dk}, if $G=SU(2)$, there exists a universal $SO(3)$ bundle
\begin{equation}\mathcal{P}^{ad}\rightarrow\mathcal{B}^{*}\times X. \nonumber \end{equation}
%\begin{equation}
%\begin{array}{lll}
%       \quad \mathcal{P}^{ad} \\  \quad  \downarrow   \\ \mathcal{B}^{*}\times X.        %  leave space using \textrm{ }
%\end{array} \nonumber \end{equation}
Then we define the $\mu_{1}$-map using the slant product pairing,
\begin{equation}\mu_{1}: H_{*}(X)\otimes \mathbb{Z}[x_{1},x_{2},...]\rightarrow H^{*}(\mathcal{B}^{*}),\nonumber \end{equation}
\begin{equation}\label{u1 map}\mu_{1}(\gamma,P)=P(0,-\frac{1}{4}p_{1}(\mathcal{P}^{ad}),0,...)/\gamma.  \end{equation}
The $\mu_{1}$-map for $U(r)$ bundles can be defined using higher Pontryagin classes of $PU(r)$ bundles with more complicated expression.
\begin{remark}
There exists a universal framed $U(r)$ bundle $\widetilde{\mathcal{P}}\rightarrow \widetilde{\mathcal{B}}^{*}\times X$, where $\widetilde{\mathcal{B}}^{*}$ is the space of framed irreducible connections and we can define a $\widetilde{\mu}_{1}$-map
\begin{equation}\tilde{\mu}_{1}: H_{*}(X)\otimes \mathbb{Z}[x_{1},x_{2},...]\rightarrow H^{*}(\widetilde{\mathcal{B}}^{*}),  \nonumber\end{equation}
\begin{equation}\label{mu1 map framed bdl}\tilde{\mu}_{1}(\gamma,P)=P(c_{1}(\widetilde{\mathcal{P}}),c_{2}(\widetilde{\mathcal{P}}),...)/\gamma.   \end{equation}
For $SU(2)$ bundles, $\mu_{1}(\gamma,P)$ pulls back to be $\tilde{\mu}_{1}(\gamma,P)$ via $\beta: \widetilde{\mathcal{B}}^{*}\rightarrow \mathcal{B}^{*}$ \cite{dk}. However, Pontryagin classes of $\mathcal{P}^{ad}$ can't recover Chern classes of $\widetilde{\mathcal{P}}$ for higher rank bundles in general.
\end{remark}
We use the pairing between $DT_{4}$ virtual cycles and ${\mu}_{1}$-maps to define $DT_{4}$ invariants.
%Thus $DT_{4}$ invariants (defined by pairing $DT_{4}$ virtual cycles in $H_{*}(\mathcal{B}^{*},\mathbb{Z})$, $H_{*}(\widetilde{\mathcal{B}}^{*},\mathbb{Z})$ with corresponding $\mu_{1}$, $\tilde{\mu}_{1}$ maps) remain the same.
\begin{definition}\label{DT4 inv of bundles}Under the assumption in Theorem \ref{main theorem},
the $DT_{4}$ invariant of $(X,\mathcal{O}_{X}(1))$ with respect to Chern character $c$ and an orientation data $o(\mathcal{L})$
is a map
\begin{equation}\label{mu map for bundles}DT_{4}^{\mu_{1}}(X,\mathcal{O}_{X}(1),c,o(\mathcal{L})):
Sym^{*}\big(H_{*}(X,\mathbb{Z})\otimes \mathbb{Z}[x_{1},x_{2},...]\big) \rightarrow \mathbb{Z} \end{equation}
such that
\begin{equation}DT_{4}^{\mu_{1}}(X,\mathcal{O}_{X}(1),c,o(\mathcal{L}))((\gamma_{1},P_{1}),(\gamma_{2},P_{2}),...) \nonumber \end{equation}
\begin{equation}=<{\mu}_{1}(\gamma_{1},P_{1})\cup {\mu}_{1}(\gamma_{2},P_{2})\cup... ,[\mathcal{M}^{DT_{4}}_{c}]^{vir}>, \nonumber \end{equation}
where
%$[\widetilde{\mathcal{M}^{DT_{4}}_{c}}]^{vir}\in H_{*}(\widetilde{\mathcal{B}}^{*})$ denotes the virtual cycle of the framed $DT_{4}$ moduli space (Remark \ref{framed DT4}) defined similarly as Definition \ref{virtual cycle when Mc=Mbdl},
$<,>$ denotes the natural pairing between homology and cohomology classes.
\end{definition}
\begin{remark}
The above $DT_{4}$ invariants can be viewed as partition functions of certain eight dimension quantum field theory \cite{bks} because of the standard super-symmetry localization \cite{witten}.
%2. If two Calabi-Yau 4-folds under Mukai flops \cite{mukai} are deformation equivalent to
%each other \cite{h}, $DT_{4}$ invariants will remain the same under these flops.
\end{remark}

\subsection{The virtual cycle construction of generalized $DT_{4}$ moduli spaces }
In this subsection, we construct virtual cycles of generalized $DT_{4}$ moduli spaces $\overline{\mathcal{M}}_{c}^{DT_{4}}$'s, when they are defined without the gluing assumption \ref{assumption on gluing}.

The first case is when $\mathcal{M}_{c}$ is smooth: the obstruction sheaf $Ob$ such that $Ob|_{\mathcal{F}}=Ext^{2}(\mathcal{F},\mathcal{F})$ is a bundle with quadratic form $Q_{Serre}$, where $Q_{Serre}$ is the Serre duality pairing. By Lemma 5 \cite{eg}, there exists a real sub-bundle $Ob_{+}$ with positive definite quadratic form such that $Ob\cong Ob_{+}\otimes_{\mathbb{R}}\mathbb{C}$ as vector bundles with quadratic form and $w_{1}(Ob_{+})=0 \Leftrightarrow$ the structure group of $(Ob,Q_{Serre})$ can be reduced to $SO(\bullet,\mathbb{C})$. We call $Ob_{+}$ the self-dual obstruction bundle and choose an orientation data $o(\mathcal{L})$ for $\mathcal{M}_{c}$ which gives an orientation on $Ob_{+}$.
\begin{definition}\label{virtual cycle when Mc smooth}
We assume $\overline{\mathcal{M}}_{c}=\mathcal{M}_{c}$ is smooth, by Proposition \ref{gene DT4 if Mc smooth}, $\overline{\mathcal{M}}_{c}^{DT_{4}}$ exists and $\overline{\mathcal{M}}_{c}^{DT_{4}}\cong\mathcal{M}_{c}$. We assume there exists an orientation data $o(\mathcal{L})$.
Then the virtual fundamental class of $\overline{\mathcal{M}}_{c}^{DT_{4}}$ ($DT_{4}$ virtual cycle for short) is
the Poincar\'{e} dual of the Euler class of the self-dual obstruction bundle over $\mathcal{M}_{c}$, i.e.
\begin{equation}[\overline{\mathcal{M}}^{DT_{4}}_{c}]^{vir}\triangleq PD(e(Ob_{+}))\in H_{r}(\mathcal{M}_{c},\mathbb{Z}),\nonumber \end{equation}
where $r=2-\chi(\mathcal{F},\mathcal{F})$ is the real virtual dimension of $\overline{\mathcal{M}}_{c}^{DT_{4}}$.
\end{definition}
\begin{remark}
In this case, the $DT_{4}$ virtual cycle is zero when $r$ is odd, and is algebraic when $r$ is even \cite{eg}.
\end{remark}
When $\mathcal{M}_{c}$ is smooth, the following lemma will be useful for later computations.
\begin{lemma}(Edidin and Graham \cite{eg}) \label{ASD equivalent to max isotropic}
Let $E\rightarrow U$ be a complex vector bundle with a non-degenerate quadratic form. $V$ is a maximal isotropic subbundle of $E$. \\
(1) If $rk(E)=2n$, then the structure group of $E$ reduces to $SO(2n,\mathbb{C})$ and the half Euler class of $E$ (i.e. the Euler class of the corresponding real quadratic bundle) is $\pm c_{n}(V)$
where the sign
depends on the choice of the maximal isotropic family of $V$. \\
(2) If $rk(E)=2n+1$ and the the structure group of $E$ reduces to $SO(2n+1,\mathbb{C})$, then the class is zero.
\end{lemma}
The next case where we have $\overline{\mathcal{M}}_{c}^{DT_{4}}$ without the gluing assumption is the following.
\begin{definition}\label{virtual cycle when ob=v+v*}
We assume $\overline{\mathcal{M}}_{c}=\mathcal{M}_{c}$ and there exists a perfect obstruction theory \cite{bf}
\begin{equation}\phi: \quad \mathcal{V}^{\bullet}\rightarrow \mathbb{L}^{\bullet}_{\mathcal{M}_{c}},  \nonumber \end{equation}
such that
\begin{equation}H^{0}(\mathcal{V}^{\bullet})|_{\{\mathcal{F}\}}\cong Ext^{1}(\mathcal{F},\mathcal{F}), \nonumber \end{equation}
\begin{equation}H^{-1}(\mathcal{V}^{\bullet})|_{\{\mathcal{F}\}}\oplus
H^{-1}(\mathcal{V}^{\bullet})|_{\{\mathcal{F}\}}^{*}\cong Ext^{2}(\mathcal{F},\mathcal{F}),  \nonumber \end{equation}
and $H^{-1}(\mathcal{V}^{\bullet})|_{\{\mathcal{F}\}}$ is a maximal isotropic subspace of $Ext^{2}(\mathcal{F},\mathcal{F})$ with respect
to the Serre duality pairing, then by Proposition \ref{ob=v+v*}, $\overline{\mathcal{M}}_{c}^{DT_{4}}$ exists, $\overline{\mathcal{M}}_{c}^{DT_{4}}\cong\mathcal{M}_{c}$ and $(\mathcal{L}_{\mathbb{C}},Q_{Serre})$ has a natural complex orientation $o(\mathcal{O})$ (Definition \ref{nat cpx ori}).

The virtual fundamental class of $\overline{\mathcal{M}}_{c}^{DT_{4}}$ ($DT_{4}$ virtual cycle for short) with respect to the natural complex orientation $o(\mathcal{O})$ is the virtual fundamental class of the above perfect obstruction theory, i.e.
\begin{equation}[\overline{\mathcal{M}}^{DT_{4}}_{c}]^{vir}\triangleq[\mathcal{M}_{c},\mathcal{V}^{\bullet}]^{vir}\in A_{\frac{r}{2}}(\mathcal{M}_{c}),   \nonumber \end{equation}
where $r=2-\chi(\mathcal{F},\mathcal{F})$ is the real virtual dimension of $\overline{\mathcal{M}}_{c}^{DT_{4}}$.
\end{definition}
${}$ \\
\textbf{The $\mu_{2}$-map}:
We define a $\mu_{2}$-map for the above two cases. We denote the universal sheaf of $\mathcal{M}_{c}$ by
\begin{equation}\mathfrak{F}\rightarrow\mathcal{M}_{c}\times X.    \nonumber \end{equation}
%\begin{equation}
%\begin{array}{lll}
%      \quad \quad \mathfrak{F} \\  \quad \quad \downarrow   \\ \mathcal{M}_{c}\times X.        %  leave space using \textrm{ }
%\end{array} \nonumber \end{equation}
The $\mu_{2}$-map is similarly defined to be
\begin{equation}\mu_{2}: H_{*}(X)\otimes \mathbb{Z}[x_{1},x_{2},...]\rightarrow H^{*}(\mathcal{M}_{c}), \nonumber \end{equation}
\begin{equation}\label{mu map for sheaves} \mu_{2}(\gamma,P)=P(c_{1}(\mathfrak{F}),c_{2}(\mathfrak{F}),...)/\gamma.   \end{equation}

\begin{definition}\label{DT4 inv of sheaves}In Definitions \ref{virtual cycle when Mc smooth}, \ref{virtual cycle when ob=v+v*},
the $DT_{4}$ invariant of $(X,\mathcal{O}_{X}(1))$ with respect to Chern character $c$ and an orientation data $o(\mathcal{L})$
%any smooth space $\mathcal{B}_{2}$ such that $\mathcal{M}_{c}\subseteq \mathcal{B}_{2}$
is a map
\begin{equation}\label{u2 map}DT_{4}^{\mu_{2}}(X,\mathcal{O}_{X}(1),c,o(\mathcal{L})): Sym^{*}\big(H_{*}(X,\mathbb{Z}) \otimes \mathbb{Z}[x_{1},x_{2},...]\big)
\rightarrow \mathbb{Z} \end{equation}
such that
\begin{equation}DT_{4}^{\mu_{2}}(X,\mathcal{O}_{X}(1),c,o(\mathcal{L}))((\gamma_{1},P_{1}),(\gamma_{2},P_{2}),...) \nonumber \end{equation}
\begin{equation}=<\mu_{2}(\gamma_{1},P_{1})\cup \mu_{2}(\gamma_{2},P_{2})\cup... ,[\overline{\mathcal{M}}^{DT_{4}}_{c}]^{vir}>,
\nonumber \end{equation}
where $<,>$ denotes the natural pairing between homology and cohomology classes.
\end{definition}

Actually, the above definition of $DT_{4}$ invariants is consistent with the definition before (\ref{mu map for bundles}).
\begin{proposition}
We assume $\overline{\mathcal{M}}_{c}=\mathcal{M}_{c}^{bdl}\neq\emptyset$ is smooth, $c=(2,0,*,0,0)$ i.e. $E\rightarrow X$ is a $SU(2)$ bundle, then
%(ii) the condition in Definition \ref{virtual cycle when ob=v+v*} is satisfied with the further assumption that the Hermitian metric on each $Ext^{2}(E,E)\cong H^{-1}(\mathcal{V}^{\bullet})|_{\{E\}}\oplus H^{-1}(\mathcal{V}^{\bullet})|_{\{E\}}^{*}$ which is induced from the Hermitian metric on $E$ is the direct sum Hermitian metric induced from a metric on $H^{-1}(\mathcal{V}^{\bullet})|_{\{E\}}$, then
\begin{equation}DT_{4}^{\mu_{1}}(X,\mathcal{O}_{X}(1),c,o(\mathcal{L}))=DT_{4}^{\mu_{2}}(X,\mathcal{O}_{X}(1),c,o(\mathcal{L})). \nonumber \end{equation}
\end{proposition}
\begin{proof}
By the assumption, we have $\overline{\mathcal{M}}_{c}\cong\mathcal{M}^{DT_{4}}_{c}\hookrightarrow \mathcal{B}^{*}$.
%\begin{equation}
%\begin{array}{llll}
%      & &  \widetilde{\mathcal{M}^{DT_{4}}_{c}}  \hookrightarrow & \widetilde{\mathcal{B}}^{*}
%      \\  & \quad  & \quad \downarrow & \downarrow \\
%       &  \overline{\mathcal{M}}_{c}\cong &\mathcal{M}^{DT_{4}}_{c}\hookrightarrow &  \mathcal{B}^{*}.
%\end{array}\nonumber \end{equation}
It is not hard to check $[\mathcal{M}^{DT_{4}}_{c}]^{vir}=i_{*}(PD(e(Ob_{+})))$, where $Ob_{+}$ is the self-dual obstruction bundle over $\overline{\mathcal{M}}_{c}$ and $i: \mathcal{M}^{DT_{4}}_{c}\hookrightarrow\mathcal{B}^{*}$ is the inclusion.
%We know $\widetilde{\mathcal{M}^{DT_{4}}_{c}}$ is smooth by the smoothness of $\mathcal{M}^{DT_{4}}_{c}$. Similarly, we have
%$[\widetilde{\mathcal{M}^{DT_{4}}_{c}}]^{vir}=i_{*}(PD(e(\widetilde{Ob}_{+})))$, where $\widetilde{Ob}_{+}$ is the self-dual obstruction bundle over $\widetilde{\mathcal{M}^{DT_{4}}_{c}}$. One can check $\widetilde{Ob}_{+}=\beta^{*}({Ob}_{+})$, where $\beta$ is induced from the projection $\widetilde{\mathcal{B}}^{*}\rightarrow \mathcal{B}^{*}$.
Meanwhile, $i^{*}(p_{1}(\mathcal{P}^{ad}))=-4c_{2}(\mathcal{P})$, where $\mathcal{P}\rightarrow\overline{\mathcal{M}}_{c}\times X$ is the universal bundle and we abuse the notation $i: \overline{\mathcal{M}}_{c}\times X\hookrightarrow \mathcal{B}^{*}\times X$ for the product of the inclusion map and the identity map.
%For the case when the condition in Definition \ref{virtual cycle when ob=v+v*} is satisfied, denote $V_{E}=H^{-1}(\mathcal{V}^{\bullet})|_{\{E\}}$.
%We assume the Hermitian metric on $Ext^{2}(E,E)\cong V_{E}\oplus V_{E}^{*}$ induced from the metric on the bundle $E$ is the direct sum Hermitian metric induced from a metric on $V_{E}$. By proposition \ref{ob=v+v*}, $Ext^{2}_{+}(E,E)\cong \{a+a^{*} | a\in V_{E}\}\cong V_{E}$ and the Kuranishi map $\kappa_{+}$ factors through $V_{E}$,
%\begin{equation}
%\xymatrix{
%Ext^{1}(E,E) \ar[r]^{\quad \kappa} \ar[dr]_{\kappa_{+}}
%& V_{E} \ar[d]^{a\mapsto a+a^{*}}_{\wr\mid}   \\
%& Ext^{2}_{+}(E,E).    }
%\nonumber \end{equation}
%We then use the natural orientation of $V_{E}$ to give an orientation on $Ext^{2}_{+}(E,E)$. The proof of the equivalence of virtual cycles constructions is standard \cite{lt3}.
\end{proof}
\begin{remark}${}$ \\
1. If $\overline{\mathcal{M}}_{c}=\mathcal{M}_{c}^{bdl}\neq\emptyset$, $c=(2,0,*,0,0)$ and conditions in Definition \ref{virtual cycle when ob=v+v*} are satisfied, we
also have
\begin{equation}DT_{4}^{\mu_{1}}(X,\mathcal{O}_{X}(1),c,o(\mathcal{L}))=DT_{4}^{\mu_{2}}(X,\mathcal{O}_{X}(1),c,o(\mathcal{L})). \nonumber \end{equation}
This is proved by showing the equivalence of our $DT_{4}$ virtual cycles with Borisov-Joyce's virtual cycles (see the appendix) as their virtual cycles are shown to be independent of choices of local charts and splittings \cite{bj}. \\
2. The condition $c=(2,0,*,0,0)$, i.e. $G=SU(2)$ is to ensure Pontryagin classes of $\mathcal{P}^{ad}$ can recover Chern classes of $\mathcal{P}$.
For higher rank bundles,
\begin{equation}DT_{4}^{\mu_{1}}(X,\mathcal{O}_{X}(1),c,o(\mathcal{L}))((\gamma_{1},P_{1}),(\gamma_{2},P_{2}),...)=
DT_{4}^{\mu_{2}}(X,\mathcal{O}_{X}(1),c,o(\mathcal{L}))((\gamma_{1},P_{1}),(\gamma_{2},P_{2}),...) \nonumber \end{equation}
holds only for careful choices of insertions $((\gamma_{1},P_{1}),(\gamma_{2},P_{2}),...)$.
\end{remark}

\subsection{Monodromy group actions and $DT_{4}$ virtual cycles}
In this subsection, we prove a result which restricts the way $DT_{4}$ virtual cycles would sit inside $\widetilde{\mathcal{B}}^{*}$ (the space of framed irreducible connections). It is based on monodromy group actions determined by loops of complex structures on $X$ and the deformation invariance of $DT_{4}$ virtual cycles. The idea is suggested to the authors by Simon Donaldson.

We first consider the case when $\overline{\mathcal{M}}_{c}=\mathcal{M}_{c}^{bdl}$, i.e. the Gieseker moduli space consists of slope stable bundles only. Then we know $\mathcal{M}_{c}^{DT_{4}}$ is compact and its virtual fundamental class exists. For convenience purposes, we will take the $DT_{4}$ virtual cycle as a homology class in $H_{*}(\widetilde{\mathcal{B}}^{*})$
as the rational cohomology of $\widetilde{\mathcal{B}}^{*}$ can be easily calculated.
\begin{lemma}(Page 181 of \cite{dk})\label{coho of framed B general}
We assume $H^{odd}(X,\mathbb{Z})=0$, then $H^{*}(\widetilde{\mathcal{B}}^{*},\mathbb{Q})$ is the polynomial algebra freely generated by
$\tilde{\mu}_{1}(\gamma,x_{l})$ (\ref{mu1 map framed bdl}), where $1\leq l\leq rk(E)$ and $\gamma$ runs through a basis of $H_{2i}(X)$ for $0<i<l$.
\end{lemma}
%\begin{remark}
%In the case when $H^{odd}(X,\mathbb{Z})\neq0$, a similar result is also true, i.e. $H^{*}(\widetilde{\mathcal{B}}^{*},\mathbb{Q})$ is the tensor product of a polynomial algebra on some even-dimensional generators and an exterior algebra on some odd-dimensional ones.
%\end{remark}
%\begin{corollary}\label{coho of framed B}
%We assume $H^{odd}(X,\mathbb{Z})=0$, $H_{2}(X)\cong \mathbb{Z}$ and the structure group of the bundle $E\rightarrow X$ is $SU(r)$ with $2\leq r\leq4$, then we have
%\begin{equation}H^{*}(\widetilde{\mathcal{B}}^{*},\mathbb{Q})\cong\mathbb{Q}[\tilde{\mu}_{1}([\omega],x_{2}),\tilde{\mu}_{1}([\omega],x_{3}),\tilde{\mu}_{1}
%([\omega^{2}],x_{3}),\tilde{\mu}_{1}(\gamma_{1},x_{3}),...,\tilde{\mu}_{1}(\gamma_{b_{4}-1},x_{3}),\nonumber \end{equation}
%\begin{equation}\tilde{\mu}_{1}([\omega],x_{4}),\tilde{\mu}_{1}
%([\omega^{2}],x_{4}),\tilde{\mu}_{1}(\gamma_{1},x_{4}),...,\tilde{\mu}_{1}(\gamma_{b_{4}-1},x_{4}),\tilde{\mu}_{1}(PD([\omega]),x_{4})], \nonumber \end{equation}
%where $\{\gamma_{i}\}_{1\leq i\leq b_{4}-1}$ is a basis of the primitive cohomology $P^{4}(X,\mathbb{Z})$ and $PD([\omega])\in H_{6}(X)$ is the Poincar\'{e} dual of $[\omega]$.
%\end{corollary}
We restrict to the case when $X$ is a smooth $CY_{4}$ complete intersections of $k$ hypersurfaces of degree $d=(d_{1},...,d_{k})$ inside $\mathbb{C}\mathbb{P}^{4+k}$, where $k\geq1$. We denote $U$ to be the space of all such complete intersections and $\pi: V\rightarrow U$ to be the corresponding family. We assume $u\in U$ corresponds to the chosen $X$. It is well-known that $\pi_{1}(U,u)$ acts on $H^{4}_{prim}(X,\mathbb{Z})$ preserving the quadratic form $q$.
\begin{definition}(\cite{ebeling})
The global monodromy group $\Gamma=\Gamma_{4,d}$ is the image of the monodromy representation $\pi_{1}(U,u)\rightarrow O(H^{4}_{prim}(X,\mathbb{Z}),q)$, where $H^{4}_{prim}(X,\mathbb{Z})$ is the primitive cohomology.
\end{definition}
The following result characterizes the behavior of the monodromy group $\Gamma$ inside $O(H^{4}_{prim}(X,\mathbb{Z}),q)$.
\begin{lemma}(Corollary 4.3.1 \cite{ebeling})\label{monodromy gp}
The global monodromy group $\Gamma$ is of finite index in $O(H^{4}_{prim}(X,\mathbb{Z}),q)$.
\end{lemma}
We notice that a ring of polynomials invariant under $\Gamma$-action is the same as the ring invariant under the action of its Zariski closure in $SO(H^{4}_{prim}(X,\mathbb{C}),q)$.
In fact, we have large Zariski closure in our case.
\begin{lemma}(Theorem 2.10, Chapter VI \cite{fm})\label{monodromy Zariski dense}${}$ \\
Let $L$ be a lattice, i.e. a finite rank free $\mathbb{Z}$-module together with a non-degenerate quadratic form $q$. We denote $V\triangleq L\otimes\mathbb{C}$ and suppose $q$ is of signature (r,s) with $r\geq3$, $s\geq2$. Let $D$ be a subgroup of finite index of $SO(L,q)$. Then $D$ is Zariski dense in $SO(V,q)$.
\end{lemma}
The following lemma gives the invariant subring explicitly.
\begin{lemma}(Lemma 2.1, Chapter VI \cite{fm})\label{invariant poly}
Let $V$ be a finite dimensional complex vector space with $dim_{\mathbb{C}}V\geq2$, $q$ is a non-degenerate quadratic form on $V$, then
$Sym^{*}(V^{*})^{SO(V,q)}=\mathbb{C}[q]$.
\end{lemma}
Finally, we can describe the $\Gamma$-invariant part of $H^{*}(\widetilde{\mathcal{B}}^{*},\mathbb{Q})$.
\begin{theorem}
Let $X$  be a smooth complete intersection Calabi-Yau 4-fold inside $\mathbb{P}^{N}$, $E\rightarrow X$ be a complex vector bundle with $SU(r)$ structure group, where $r\geq2$. Then we have
\begin{equation}
H^{*}(\widetilde{\mathcal{B}}^{*},\mathbb{C})^{\Gamma}=\mathbb{C}[\tilde{\mu}_{1}(PD[\omega^{3}],x_{j})_{2\leq j\leq r},\tilde{\mu}_{1}(PD[\omega^{2}],x_{j})_{3\leq j\leq r},
\tilde{\mu}_{1}(PD[q],x_{j})_{3\leq j\leq r},\nonumber \end{equation}
\begin{equation}\tilde{\mu}_{1}(PD[\omega],x_{j})_{4 \leq j\leq r},\tilde{\mu}_{1}([X],x_{j})_{5\leq j\leq r}], \nonumber \end{equation}
where $PD$ denotes the Poincar\'{e} dual operator.
\end{theorem}
\begin{proof}
By the Lefschetz hyperplane theorem, we know $H^{odd}(X,\mathbb{Z})=0$, $H_{2}(X)\cong \mathbb{Z}$. By \cite{klemyau}, the signature of $(H_{prim}^{4}(X,\mathbb{Z}),q)$ is $(4h^{3,1}+49,2h^{3,1})$. We know $h^{3,1}\geq1$ for complete intersection $CY_{4}$'s as it is the dimension of the deformation space. Then we apply Lemma \ref{coho of framed B general}, \ref{monodromy gp}, \ref{monodromy Zariski dense} and \ref{invariant poly} to $X$.
\end{proof}
\begin{corollary}\label{monodromy inv cycles}
The $DT_{4}$ virtual cycle $[\widetilde{\mathcal{M}_{c}^{DT_{4}}}]^{vir}\in H_{*}(\widetilde{\mathcal{B}}^{*})$ (Remark \ref{framed DT4}) can be expressed as a homogenous polynomial in terms of the above generators of $H^{*}(\widetilde{\mathcal{B}}^{*},\mathbb{C})^{\Gamma}$.
\end{corollary}
\begin{proof}
By the deformation invariance of $DT_{4}$ virtual cycles.
\end{proof}
\begin{remark}
For general cases when $\overline{\mathcal{M}}_{c}\neq \mathcal{M}_{c}^{bdl}$, based on Definition \ref{DT4 inv of sheaves}, $DT_{4}$ invariants are defined by pairing $DT_{4}$ virtual cycles $[\overline{\mathcal{M}}_{c}^{DT_{4}}]^{vir}$'s $\in H_{r}(\overline{\mathcal{M}}_{c})$ with the corresponding $\mu_{2}$-maps. These a priori have nothing to do with $H^{*}(\widetilde{\mathcal{B}}^{*})$ and the global monodromy group action is not explicitly related to $DT_{4}$ invariants.

However, we can use Seidel-Thomas twists to embed the Gieseker moduli space into $\mathcal{B}_{E^{'}}^{*}$, the space of irreducible connections on some other topological bundle $E^{'}$. We then use $\mu_{1}$-maps associated with $\mathcal{B}_{E^{'}}^{*}$ to define invariants and the invariants would share a similar property in Corollary \ref{monodromy inv cycles}. We notice that the way Gieseker moduli spaces are embedded into $\mathcal{B}_{E^{'}}^{*}$ is not canonical and invariants defined in this way are a priori not the same as $DT_{4}$ invariants defined before in
Definition \ref{DT4 inv of sheaves}.
\end{remark}

\subsection{Some vanishing results of $DT_{4}$ virtual cycles }

\subsubsection{Vanishing results for certain choices of $c$}
We know that $DT_{4}$ invariants are deformation invariants of Calabi-Yau 4-fold $X$. Unlike the case of Calabi-Yau 3-folds with $SU(3)$ holonomy \cite{js}, classes in $H^{2,2}(X)$ may not remain $(2,2)$-type when we deform the complex structure. When it is deformed to the case when non-algebraic stuff appears, i.e. $c\notin\bigoplus_{i}\textrm{ }H^{i,i}(X)$,
one has to use other nice analytic compactification to define invariants. However, $DT_{4}$ invariants turn out to be rather trivial in this case.
\begin{proposition}\label{vanishing for some c}
Let $X$ be a compact Calabi-Yau 4-fold. We fix cohomology classes $c=\bigoplus_{i=0}^{4} c|_{H^{2i}(X,\mathbb{Q})}$, then \\
(1) If $c|_{H^{4}(X,\mathbb{Q})}$ has no component in $H^{0,4}(X)$ and $c\notin \bigoplus_{i=0}^{4}H^{i,i}(X)$, then $\mathcal{M}^{DT_{4}}_{c}=\emptyset$. \\
(2) If $c\in \bigoplus_{i=0}^{4}H^{i,i}(X)$ and $\exists\textrm{ } \varphi\in H^{1}(X,TX)$ such that $\varphi\lrcorner\textrm{ } \big(c|_{H^{2,2}(X,\mathbb{Q})}\big)\neq0$, then $\mathcal{M}_{c}=\emptyset$ for some complex structure $X_{t}$, where $t$ is small
and $X_{t}$ is the family of Calabi-Yau 4-fold such that $X_{0}=X$ determined by $\varphi$.
\end{proposition}
\begin{proof}
(1) If there is no complex vector bundle $E$ such that $ch(E)=c$, then it is trivial to get $\mathcal{M}^{DT_{4}}_{c}=\emptyset$. Thus we assume
there exists a complex vector bundle $E$ such that $ch(E)=c$, $ch_{2}(E)$ does not have $(0,4)$ component and
$\bigoplus_{i=0}^{4}ch_{i}(E)\notin \bigoplus_{i=0}^{4}H^{i,i}(X)$. If $\mathcal{M}^{DT_{4}}_{c}\neq\emptyset$, we take an element $d_{A}$ inside. Then by Lemma \ref{C2 condition}, the bundle is holomorphic which contradicts with the assumption.

(2) If $\mathcal{M}_{c}=\emptyset$ for $X$, we are done. Thus we assume $\mathcal{M}_{c}\neq\emptyset$ and take a stable sheaf $\mathcal{F}$ with
$ch(\mathcal{F})=c$. By the Tian-Todorov theorem \cite{h0}, $\varphi$ gives a one parameter family of complex structures $X_{t}$ such that $X_{0}=X$. Since
$\varphi\lrcorner\textrm{ } ch_{2}(\mathcal{F})\neq0$, we get $\varphi\circ At(\mathcal{F})\neq0$ as
\begin{equation}tr\big(\varphi\circ At(\mathcal{F})\circ At(\mathcal{F})\big)=2\varphi\lrcorner\textrm{ } ch_{2}(\mathcal{F}), \nonumber \end{equation}
where $At(\mathcal{F})$ is the Atiyah class of $\mathcal{F}$ \cite{bflenner}.
We notice that $\varphi\circ At(\mathcal{F})\in Ext^{2}(\mathcal{F},\mathcal{F})$ is the obstruction class for the deformation of $\mathcal{F}$ along $\varphi$. Hence we see that any sheaf in $\mathcal{M}_{c}$ can't be deformed to the nearby complex structures. Thus, $\mathcal{M}_{c}=\emptyset$
for $X_{t}$, where $t\neq0$ and small.
\end{proof}
\begin{remark}
Under the same assumption as in Proposition \ref{vanishing for some c}, any deformation invariant associated with $\mathcal{M}_{c}$ must vanish, for instance, $DT_{4}$ invariants defined in this paper and Borisov-Joyce's virtual fundamental classes \cite{bj} will vanish.
\end{remark}

\subsubsection{Vanishing results for hyper-K\"ahler 4-folds }
We assume $X$ to be a compact hyper-K\"ahler 4-fold with a holomorphic symplectic two form $\sigma\in H^{0}(X,\Omega_{X}^{2})$.

There is an obvious surjective cosection map of the obstruction sheaf of $\mathcal{M}_{c}$
\begin{equation}\nu: Ob \twoheadrightarrow \mathcal{O}_{\mathcal{M}_{c}},   \nonumber \end{equation}
given by the trace map
\begin{equation}\nu: Ext^{2}(\mathcal{F},\mathcal{F})\rightarrow H^{2}(X,\mathcal{O}_{X})\cong \mathbb{C} .\nonumber \end{equation}
Then we have the surjective cosection map for $DT_{4}$ obstruction space
\begin{equation}\label{nu+}\nu_{+}:Ext^{2}_{+}(\mathcal{F},\mathcal{F})\twoheadrightarrow H^{2}_{+}(X,\mathcal{O}_{X})\cong \mathbb{R}, \end{equation}
which leads to the vanishing of virtual fundamental classes of (generalized) $DT_{4}$ moduli spaces.
\begin{remark}
If we fix the determinant of the torsion-free sheaf, there is a less obvious cosection map \cite{bflenner}
\begin{equation}\nu_{hyper}: Ext^{2}_{0}(\mathcal{F},\mathcal{F})\twoheadrightarrow H^{4,4}(X), \nonumber \end{equation}
defined to be the composition of
\begin{equation} \xymatrix@1{
Ext^{2}_{0}(\mathcal{F},\mathcal{F})\ar[r]^{\cdot\frac{(At(\mathcal{F}))^{2}}{2}\quad} & Ext^{4}(\mathcal{F},\mathcal{F}\otimes\Omega_{X}^{2})\ar[r]^{\quad tr}
& H^{4}(X,\Omega_{X}^{2})\ar[r]^{\wedge \sigma} & H^{4,4}(X) } . \nonumber \end{equation}
Similar to \cite{maulik}, one can show $\nu_{hyper}$ is surjective if $ch_{3}(\mathcal{F})\neq0$. However, it does not
factor through $Ext^{2}_{+}(\mathcal{F},\mathcal{F})$ and in general does not give a surjective cosection map of
the trace-free $DT_{4}$ obstruction space.
\end{remark}
In fact, we will show later in the $DT_{4}/GW$ correspondence that there exists a surjective cosection map for the trace-free $DT_{4}$ obstruction space which turns out to be the same as the cosection map of $GW$ theory for hyper-K\"ahler 4-folds (\ref{hyper cosection}).

\section{$DT_{4}$ invariants for compactly supported sheaves on local $CY_{4}$}
In this section, we study $DT_{4}$ invariants for compactly supported sheaves on local (non-compact) Calabi-Yau 4-folds.
We first consider Calabi-Yau 4-folds of type $K_{Y}$, where $Y$ is a
compact Fano 3-fold. We define their $DT_{4}$ invariants and show that there is a $DT_{4}/DT_{3}$ correspondence.
Secondly, we study $DT_{4}$ invariants on $T^{*}S$, where $S$ is a compact algebraic surface,
and prove their vanishing. When $S=\mathbb{P}^{2}$, we discuss the behaviour of $DT_{4}$ invariants under a class of birational transformations
called Mukai flops \cite{mukai}. Finally, we consider $DT_{4}$ invariants for the total space $X=Tot(L_{1}\oplus L_{2}\rightarrow S)$ of some rank two bundle over a projective surface $S$
and compute them explicitly using G\"{o}ttsche-Carlsson-Okounkov formula \cite{carlssonokounkov}, \cite{gottsche}.

\subsection{The case of $X=K_{Y}$ and $DT_{4}/DT_{3}$ correspondence}
We first describe the stability of compactly supported sheaves on $X$.
We denote $\iota: Y\rightarrow K_{Y}$ to be the zero section and $\pi: K_{Y}\rightarrow Y $ to be the projection.
We pick an ample line bundle $\mathcal{O}_{Y}(1)$ on $Y$ and define the Hilbert polynomial of a compactly supported coherent sheaf $\mathcal{F}$ on $X=K_{Y}$ to be $\chi(\mathcal{F}\otimes \pi^{*}\mathcal{O}_{Y}(k))$ for $k\gg0$. Then we can define Gieseker $\pi^{*}\mathcal{O}_{Y}(1)$-stability on compactly supported sheaves over $X$ following \cite{hl}.
\begin{lemma}\label{cp supp}
Let $Y$ be a compact Fano 3-fold, then any $\pi^{*}\mathcal{O}_{Y}(1)$-stable sheaf with three-dimensional compact support on local Calabi-Yau 4-fold $X=K_{Y}$ is of type $\iota_{*}(\mathcal{F})$,
where $\mathcal{F}$ is $\mathcal{O}_{Y}(1)$-stable on $Y$.
\end{lemma}
\begin{proof}
The proof is similar to the proof of Lemma 7.1 in \cite{hua}. We only need to show that any compactly supported stable sheaf is
scheme theoretically supported on $Y$. We denote $Z$ to be the scheme theoretical support of a compactly supported stable sheaf $\mathcal{E}$.
By the trace map \cite{hl}, we have
\begin{equation}H^{0}(Z,\mathcal{O}_{Z})\hookrightarrow Ext^{0}_{Z}(\mathcal{E},\mathcal{E}). \nonumber \end{equation}
It suffices to show $dim H^{0}(Z,\mathcal{O}_{Z})>1$ to get contradiction as stable sheaf $\mathcal{E}$ is always simple.

By our assumption, dimension of the support of $\mathcal{E}$ is three.
Then $Z$ is an order $n\geq 1$ thickening of $Y$ in the normal direction inside $X$, i.e.
\begin{equation}Z=\mathbf{Spec}\big(\bigoplus_{i=0}^{n}K_{Y}^{-i}\big).\nonumber \end{equation}
There is a spectral sequence such that $E_{\infty}^{0,0}=H^{0}(Z,\mathcal{O}_{Z})$ and $E_{2}^{0,0}=H^{0}(Y,\oplus_{i=0}^{n}K_{Y}^{-i})$
which implies
\begin{equation}H^{0}(Z,\mathcal{O}_{Z})\cong\oplus_{i=0}^{n}H^{0}(Y,K_{Y}^{-i}) .\nonumber \end{equation}
As $H^{0}(Y,K_{Y}^{-1})\neq 0$ for any Fano 3-fold, $dim H^{0}(Z,\mathcal{O}_{Z})\geq 2$, which leads to a contradiction.
Thus $\mathcal{E}$ is of type $\iota_{*}\mathcal{F}$, where $\mathcal{F}$ is a sheaf on $Y$. Now we show $\mathcal{F}$ is stable with respect to $\mathcal{O}_{Y}(1)$.

By the projection formula, for any $k$, we have
\begin{equation}\iota_{*}(\mathcal{F}\otimes_{\mathcal{O}_{Y}}\mathcal{O}_{Y}(k))
=\iota_{*}(\mathcal{F}\otimes_{\mathcal{O}_{Y}}\iota^{*}\pi^{*}\mathcal{O}_{Y}(k))
=\iota_{*}\mathcal{F}\otimes_{\mathcal{O}_{X}}\pi^{*}\mathcal{O}_{Y}(k). \nonumber \end{equation}
Thus
\begin{equation}H^{*}(Y,\mathcal{F}\otimes_{\mathcal{O}_{Y}}\mathcal{O}_{Y}(k))
=H^{*}(X,\iota_{*}\mathcal{F}\otimes_{\mathcal{O}_{X}}\pi^{*}\mathcal{O}_{Y}(k)). \nonumber \end{equation}
Then the $\pi^{*}\mathcal{O}_{Y}(1)$-stability for $\iota_{*}\mathcal{F}$ on $K_{Y}$ is equivalent to the $\mathcal{O}_{Y}(1)$-stability for $\mathcal{F}$ on $Y$.
%Then, we assume that $Z$ is an order $n\geq 1$ thickening of a subvariety $D$ in $Y$. Similarly, we only need to show $H^{1}(D,K_{Y}^{-1}\mid_{D})$ is not zero. This is true because $Y$ is Fano and the ample line bundle $K_{Y}^{-1}$ is still ample when restricted to any subvariety by Nakai-Moishezon-Kleiman criterion. Thus we finish the proof.
\end{proof}
We compare the obstruction theory of $\iota_{*}(\mathcal{F})$ on $K_{Y}$ with the obstruction theory of $\mathcal{F}$ on $Y$, which are controlled by
$L_{\infty}$-algebra structures on $Ext^{*}(\iota_{*}\mathcal{F},\iota_{*}\mathcal{F})$ and $Ext^{*}(\mathcal{F},\mathcal{F})$ respectively.
We first recall
\begin{definition}\label{L inf}
Let $(L=\oplus_{i=0}^{d}L^{i},\{\mu_{k}\})$ be a finite dimensional $L_{\infty}$-algebra over $\mathbb{C}$, and $\bar{L}$
be the graded vector space $L\oplus L[-d-1]$, i.e. $\bar{L}^{i}=L^{i}\oplus(L^{d+1-i})^{*}$.
We define a cyclic pairing and $L_{\infty}$-products $\bar{\mu}_{k}: \wedge^{k}\bar{L}\rightarrow \bar{L}[2-k]$
as follows: \\
$(1)$ we define a bilinear form $Q$ on $\bar{L}$ by the natural pairing between $L$ and $L^{*}$, \\
$(2)$ if the inputs of $\bar{\mu}_{k}$ all belong to $L$, then $\bar{\mu}_{k}=\mu_{k}$,  \\
$(3)$ if more than one input belong to $L^{*}$, then $\bar{\mu}_{k}=0$, \\
$(4)$ if there is exactly one input $a_{i}^{*}\in L^{*}$, then $\bar{\mu}_{k}$ is defined by
\begin{equation}Q(\overline{\mu}_{k}(a_{1},...,a_{i},...,a_{k}),b)=(-1)^{\epsilon}
Q(\mu_{k}(a_{i+1},...,a_{k},b,a_{1},...,a_{i-1}),a_{i}^{*})  \nonumber \end{equation}
for arbitary $b\in L$ and $\epsilon$ depends on $\{a_{i}\}$, $b$ only.  \\
We call the $L_{\infty}$-algebra $(\bar{L},\{\bar{\mu}_{k}\},Q)$ the $(d+1)$-dimensional cyclic completion of $(L,\{\mu_{k}\})$.
\end{definition}
\begin{lemma}(Segal \cite{segal}) \label{segal theorem}
Let $Y$ be a smooth proper scheme of $dim_{\mathbb{C}}=d-1$, $\iota: Y\rightarrow K_{Y}$ be the zero section map. Then for any $S\in D^{b}(Y)$ the
$A_{\infty}$ ($L_{\infty}$)-algebra $Ext^{*}_{K_{Y}}(\iota_{*}S,\iota_{*}S)$ is the d-dim cyclic completion of $Ext^{*}_{Y}(S,S)$.
\end{lemma}
As an application of the above lemma, we have,
\begin{lemma}Let $\mathcal{F}$ be a torsion-free slope-stable sheaf on a compact Fano 3-fold $Y$.
We denote $\iota: Y\rightarrow K_{Y}=X$ to be the zero section map. Then we have canonical isomorphisms
\begin{equation}Ext^{1}_{X}(\iota_{*}\mathcal{F},\iota_{*}\mathcal{F})\cong Ext^{1}_{Y}(\mathcal{F},\mathcal{F}), \nonumber \end{equation}
\begin{equation}\label{cyclic cpt on coh}Ext^{2}_{X}(\iota_{*}\mathcal{F},\iota_{*}\mathcal{F})\cong
Ext^{2}_{Y}(\mathcal{F},\mathcal{F})\oplus Ext^{2}_{Y}(\mathcal{F},\mathcal{F})^{*}. \end{equation}
And there exists a local Kuranishi map
\begin{equation}\kappa: Ext^{1}_{X}(\iota_{*}\mathcal{F},\iota_{*}\mathcal{F})\rightarrow
Ext^{2}_{X}(\iota_{*}\mathcal{F},\iota_{*}\mathcal{F}) \nonumber \end{equation}
for deformations of $\iota_{*}\mathcal{F}$ on $X$ which can be identified with a local Kuranishi map
\begin{equation}Ext^{1}_{Y}(\mathcal{F},\mathcal{F})\rightarrow Ext^{2}_{Y}(\mathcal{F},\mathcal{F}) \nonumber \end{equation}
for deformations of $\mathcal{F}$ on $Y$.
Furthermore, under the above identification, $Ext^{2}_{Y}(\mathcal{F},\mathcal{F})$ is a maximal isotropic subspace of $Ext^{2}_{X}(\iota_{*}\mathcal{F},\iota_{*}\mathcal{F})$ with
respect to the Serre duality pairing.
\end{lemma}
\begin{proof}
By Lemma \ref{segal theorem}, we have
\begin{equation}Ext^{1}_{X}(\iota_{*}\mathcal{F},\iota_{*}\mathcal{F})\cong Ext^{1}_{Y}(\mathcal{F},\mathcal{F})\oplus
Ext^{3}_{Y}(\mathcal{F},\mathcal{F})^{*}=Ext^{1}_{Y}(\mathcal{F},\mathcal{F}), \nonumber \end{equation}
\begin{equation}\label{cyclic cpt on coh}Ext^{2}_{X}(\iota_{*}\mathcal{F},\iota_{*}\mathcal{F})\cong
Ext^{2}_{Y}(\mathcal{F},\mathcal{F})\oplus Ext^{2}_{Y}(\mathcal{F},\mathcal{F})^{*}. \end{equation}
At $\iota_{*}\mathcal{F}$, there exists a Kuranishi map
\begin{equation}\kappa: Ext^{1}_{X}(\iota_{*}\mathcal{F},\iota_{*}\mathcal{F})\rightarrow
Ext^{2}_{X}(\iota_{*}\mathcal{F},\iota_{*}\mathcal{F}), \quad \kappa(x)=\sum_{k}\bar{\mu}_{k}(x^{\otimes k}),  \nonumber \end{equation}
which is described by the $L_{\infty}$-products in Definition \ref{L inf}. We can then identify $\kappa$ with the Kuranishi map
\begin{equation}\kappa: Ext^{1}_{X}(\mathcal{F},\mathcal{F})\rightarrow
Ext^{2}_{X}(\mathcal{F},\mathcal{F}), \quad \kappa(x)=\sum_{k}\mu_{k}(x^{\otimes k})  \nonumber \end{equation}
for deformations of $\mathcal{F}$ on $Y$.

By (\ref{cyclic cpt on coh}), we know the quadratic pairing
\begin{equation}Ext^{2}_{X}(\iota_{*}\mathcal{F},\iota_{*}\mathcal{F})\otimes Ext^{2}_{X}(\iota_{*}\mathcal{F},\iota_{*}\mathcal{F})\rightarrow
Ext^{4}_{X}(\iota_{*}\mathcal{F},\iota_{*}\mathcal{F}) \nonumber \end{equation}
is just the quadratic pairing
\begin{equation}(Ext^{2}_{Y}(\mathcal{F},\mathcal{F})\oplus Ext^{2}_{Y}(\mathcal{F},\mathcal{F})^{*})\otimes (Ext^{2}_{Y}(\mathcal{F},\mathcal{F})\oplus Ext^{2}_{Y}(\mathcal{F},\mathcal{F})^{*})\rightarrow
Ext^{3}_{Y}(\mathcal{F},\mathcal{F}\otimes K_{Y}).   \nonumber \end{equation}
When we restrict to $Ext^{2}_{Y}(\mathcal{F},\mathcal{F})$, the pairing will produce an element in $Ext^{4}_{Y}(\mathcal{F},\mathcal{F})=0$.
\end{proof}
From the above lemma and Lemma \ref{cp supp}, for a polarized compact Fano 3-fold $(Y,\mathcal{O}_{Y}(1))$, the moduli space of $\pi^{*}\mathcal{O}_{Y}(1)$ slope-stable compactly supported sheaves on $K_{Y}$ with compactly supported Chern character \cite{js}, $c=(0,c|_{H_{c}^{2}(X)}\neq 0,c|_{H_{c}^{4}(X)},c|_{H_{c}^{6}(X)},c|_{H_{c}^{8}(X)})$ can be identified with a moduli space of torsion-free $\mathcal{O}_{Y}(1)$-stable sheaves on $Y$ with certain Chern character $c'\in H^{even}(Y)$ which is uniquely determined by $c$. Meanwhile the condition in Definition \ref{virtual cycle when ob=v+v*} is satisfied due to Thomas \cite{th}.

%Meanwhile, we have natural isomorphism $Ext^{2}_{Y}(\mathcal{F},\mathcal{F})\cong Ext^{2}_{+}(\iota_{*}\mathcal{F},\iota_{*}\mathcal{F})$ which gives natural orientation on $Ext^{2}_{+}(\iota_{*}\mathcal{F},\iota_{*}\mathcal{F})$. Then the determinant line bundle $\mathcal{L}$ attains a natural complex orientation $o(\mathcal{L})$.
By Definition \ref{virtual cycle when ob=v+v*}, the generalized $DT_{4}$ moduli space exists and we can identify its virtual fundamental class with the virtual fundamental class of a moduli space of stable sheaves on $Y$. Furthermore, if we use the same $\mu$-map (\ref{mu map for sheaves}) to define invariants, we can also identify them.
\begin{theorem}\label{compact supp DT4}($DT_{4}/DT_{3}$ correspondence) \\
Let $\pi: X=K_{Y}\rightarrow Y$ be the projection map and $(Y,\mathcal{O}_{Y}(1))$ be a polarized compact Fano 3-fold. If $c=(0,c|_{H_{c}^{2}(X)}\neq 0,c|_{H_{c}^{4}(X)},c|_{H_{c}^{6}(X)},c|_{H_{c}^{8}(X)})$ and the Gieseker moduli space of compactly supported sheaves $\overline{\mathcal{M}}_{c}(X,\pi^{*}\mathcal{O}_{Y}(1))$ consists of slope-stable sheaves, then sheaves in $\mathcal{M}_{c}(X,\pi^{*}\mathcal{O}_{Y}(1))$ are of type $\iota_{*}(\mathcal{F})$, where $\iota: Y\rightarrow K_{Y}$ is the zero section and $c'=ch(\mathcal{F})\in H^{even}(Y)$ is uniquely determined by $c$. Furthermore, the generalized $DT_{4}$ moduli space exists
\begin{equation}\overline{\mathcal{M}}_{c}^{DT_{4}}(X,\pi^{*}\mathcal{O}_{Y}(1))=\mathcal{M}_{c}(X,\pi^{*}\mathcal{O}_{Y}(1))\cong \mathcal{M}_{c'}(Y,\mathcal{O}_{Y}(1)) \nonumber \end{equation}
and its virtual fundamental class (Definition \ref{virtual cycle when ob=v+v*}) satisfies
\begin{equation}[\overline{\mathcal{M}}_{c}^{DT_{4}}(X,\pi^{*}\mathcal{O}_{Y}(1))]^{vir}=[\mathcal{M}_{c'}(Y,\mathcal{O}_{Y}(1))]^{vir}, \nonumber \end{equation}
where $[\mathcal{M}_{c'}(Y,\mathcal{O}_{Y}(1))]^{vir}$ is the $DT_{3}$ virtual cycle defined by Thomas \cite{th}.

Since $H_{*}(X)\cong H_{*}(Y)$ and $H^{*}(\overline{\mathcal{M}}_{c}^{DT_{4}}(X,\pi^{*}\mathcal{O}_{Y}(1)))\cong H^{*}(\mathcal{M}_{c'}(Y,\mathcal{O}_{Y}(1)))$, we can use the same $\mu_{2}$-map (\ref{mu map for sheaves}) to define invariants, then
\begin{equation}DT_{4}^{\mu_{2}}(X,\pi^{*}\mathcal{O}_{Y}(1),c,o(\mathcal{O}))=DT_{3}(Y,\mathcal{O}_{Y}(1),c'), \nonumber \end{equation}
where $o(\mathcal{O})$ is the natural complex orientation and $DT_{3}(Y,\mathcal{O}_{Y}(1),c')$ is defined by pairing $[\mathcal{M}_{c'}(Y,\mathcal{O}_{Y}(1))]^{vir}$ with the $\mu_{2}$-map (\ref{u2 map}).
\end{theorem}
\begin{remark}
If we have a compact complex smooth 4-fold $X$ (not necessarily Calabi-Yau) containing a
Fano 3-fold $Y$ such that $\mathcal{N}_{Y/X}=K_{Y}$ and $\mathcal{N}_{Y/X}^{*}$ is ample, e.g. $X=P(K_{Y}\oplus\mathcal{O})$.
Then we can define $DT_{4}$ invariants for stable sheaves supported in $Y$ because $X$ contains $K_{Y}$ as its open subset by the renowned theorem of Grauert.
\end{remark}

\subsection{The case of $X=T^{*}S$ and relations with Mukai flops}
We take $X=T^{*}S$ which is a hyper-K\"ahler 4-fold when $S=\mathbb{P}^{2}$.
In this subsection, we consider counting torsion sheaves scheme theoretically supported on $S$.
Let $\mathcal{F}$ be a torsion-free sheaf on a projective surface $(S,\mathcal{O}_{S}(1))$,
$\pi: T^{*}S\rightarrow S$ be the projection and $\iota: S\rightarrow T^{*}S$ be the inclusion of the zero section.

We relate the obstruction theory of sheaf $\iota_{*}(\mathcal{F})$ on $X$ to the obstruction theory of $\mathcal{F}$ on $S$.
By the projection formula \cite{hart},
\begin{equation}\iota_{*}(\mathcal{F})=\iota_{*}(\iota^{*}\pi^{*}\mathcal{F}\otimes_{\mathcal{O}_{S}}\mathcal{O}_{S})
=\pi^{*}\mathcal{F}\otimes_{\mathcal{O}_{X}}\iota_{*}\mathcal{O}_{S},
\nonumber \end{equation}
where $\mathcal{F}$ is a complex of locally free sheaves on $S$.
Then we have a local to global spectral sequence $E_{2}^{p,q}=Ext^{q}_{S}(\mathcal{F},\wedge^{p}\Omega^{1}_{S}\otimes \mathcal{F})\Rightarrow Ext^{*}_{X}(\iota_{*}\mathcal{F},\iota_{*}\mathcal{F})$.
%\begin{eqnarray*}Ext^{*}_{X}(\iota_{*}\mathcal{F},\iota_{*}\mathcal{F})&\Leftarrow&
%H^{*}(X,\mathcal{E}xt^{*}_{X}(\iota_{*}\mathcal{F},\iota_{*}\mathcal{F})) \\
%&\cong& H^{*}(X,\mathcal{E}xt^{*}_{X}(\pi^{*}\mathcal{F}\otimes_{\mathcal{O}_{X}}\iota_{*}\mathcal{O}_{S},
%\pi^{*}\mathcal{F}\otimes_{\mathcal{O}_{X}}\iota_{*}\mathcal{O}_{S})) \\
%&\cong& H^{*}(X,\mathcal{E}xt^{*}_{X}(\iota_{*}\mathcal{O}_{S},\iota_{*}\mathcal{O}_{S})\otimes_{\mathcal{O}_{X}}End (\pi^{*}\mathcal{F})) \\
%&\cong& H^{*}(X,\mathcal{E}xt^{*}_{X}(\mathcal{O}_{S},\mathcal{O}_{S})\otimes_{\mathcal{O}_{X}} End (\pi^{*}\mathcal{F})) \\
%&\cong& H^{*}(X,\iota_{*}(\wedge^{*}\mathcal{N}_{S/X})\otimes_{\mathcal{O}_{X}} \pi^{*}End \mathcal{F}) \\
%&\cong& H^{*}(X,\iota_{*}(\wedge^{*}\mathcal{N}_{S/X}\otimes_{\mathcal{O}_{S}} \iota^{*}\pi^{*}End \mathcal{F})) \\
%&\cong& H^{*}(S,\wedge^{*}\Omega^{1}_{S}\otimes End \mathcal{F}) \\
%&\cong& Ext^{*}_{S}(\mathcal{F},\wedge^{*}\Omega^{1}_{S}\otimes \mathcal{F}).
%\end{eqnarray*}
We give a criterion for it to degenerate at $E_{2}$ page.
\begin{lemma}\label{TS lemma1}Let $\mathcal{F}$ be a torsion-free sheaf on $S$. \\
$(1)$ If $Ext_{S}^{2}(\mathcal{F},\mathcal{F})=0$, the above spectral sequence degenerates at $E_{2}$ page.  \\
$(2)$ If the degree of $K_{S}$ is negative with respect to the chosen polarization $\mathcal{O}_{S}(1)$ and $\mathcal{F}$ is slope-stable, then we have $Ext_{S}^{2}(\mathcal{F},\mathcal{F})=0$.
\end{lemma}
\begin{proof}
$(1)$ We denote $E_{2}^{p,q}=Ext_{S}^{q}(\mathcal{F},\wedge^{p}\Omega^{1}_{S}\otimes \mathcal{F})$. We have
\begin{equation}E_{2}^{p-2,q+1}\rightarrow E_{2}^{p,q}\rightarrow E_{2}^{p+2,q-1}, \nonumber \end{equation}
whose cohomology is $E_{3}^{p,q}$. Then $0\rightarrow E_{2}^{1,q}\rightarrow 0$
yields $E_{3}^{1,q}\cong E_{2}^{1,q}$. Meanwhile, we have
\begin{equation}E_{2}^{0,q+1}\rightarrow E_{2}^{2,q}\rightarrow 0 ,\quad  0\rightarrow E_{2}^{0,q}\rightarrow E_{2}^{2,q-1}.
\nonumber \end{equation}
Under the assumption that $Ext_{S}^{2}(\mathcal{F},\mathcal{F})=0$, we get
\begin{equation}E_{2}^{2,0}=E_{2}^{0,2}=0. \nonumber \end{equation}
Thus the above spectral sequence degenerates at $E_{2}$. \\
${}$ \\
$(2)$ By Serre duality, we have
\begin{equation} Ext_{S}^{2}(\mathcal{F},\mathcal{F})\cong Hom_{\mathcal{O}_{S}}(\mathcal{F},\mathcal{F}\otimes K_{S}).  \nonumber \end{equation}
By assumption,
\begin{equation}\mu(\mathcal{F})=\frac{deg(\mathcal{F})}{rk(\mathcal{F})}>\frac{deg(\mathcal{F}\otimes K_{S})}
{rk(\mathcal{F}\otimes K_{S})}= \mu(\mathcal{F}\otimes K_{S}).\nonumber \end{equation}
If the above homomorphism is not zero, we choose such a nonzero morphism
\begin{equation}f:\mathcal{F}\rightarrow \mathcal{F}\otimes K_{S},  \nonumber \end{equation}
then
\begin{equation}0\neq\mathcal{F}/ker(f)\hookrightarrow \mathcal{F}\otimes K_{S}.  \nonumber \end{equation}
By the stability of $\mathcal{F}$, we have
\begin{equation}\mu(\mathcal{F}/ker(f))\geq\mu(\mathcal{F}).  \nonumber \end{equation}
Thus
\begin{equation}\mu(\mathcal{F}\otimes K_{S})<\mu(\mathcal{F})\leq\mu(\mathcal{F}/ker(f)), \nonumber \end{equation}
which contradicts with the slope semi-stability of $\mathcal{F}\otimes K_{S}$.
\end{proof}
We have to make sure that sheaves scheme theoretically supported on $S$ can not move outside. This can be done by finding conditions such that
\begin{equation}Ext^{0}_{S}(\mathcal{F},\mathcal{F}\otimes \Omega_{S}^{1})=0. \nonumber \end{equation}
If $S=\mathbb{P}^{2}$ and $\mathcal{F}$ is torsion-free slope stable, then $\Omega_{S}^{1}$ is slope stable and the
tensor product $\mathcal{F}\otimes \Omega_{S}^{1}$ is slope semi-stable (see Theorem 3.1.4 \cite{hl}).
The condition $Ext^{0}_{S}(\mathcal{F},\mathcal{F}\otimes \Omega_{S}^{1})=0$ is satisfied by following the argument in the above proof.
When $\mathcal{F}=I$ is an ideal sheaf of points, we have
\begin{lemma}\label{TS lemma2}
Let $\mathcal{F}=I$ be an ideal sheaf of points on $S$. If $h^{0,1}(S)=0$, then
$Ext^{0}_{S}(\mathcal{F},\mathcal{F}\otimes \Omega_{S}^{1})=0$.
\end{lemma}
\begin{proof}
$0\rightarrow I\rightarrow \mathcal{O}_{S}\rightarrow \mathcal{O}_{Z}\rightarrow 0$ induces
\begin{eqnarray*}0\rightarrow Ext^{0}_{S}(I,I\otimes \Omega_{S}^{1})\rightarrow Ext^{0}_{S}(I,\Omega_{S}^{1})&\cong&
Ext^{2}(\Omega_{S}^{1},I\otimes K_{S}) \\
&=& H^{2}(S,I\otimes K_{S}\otimes T_{S}),
\end{eqnarray*}
\begin{equation}0=H^{1}(S,\mathcal{O}_{Z}\otimes K_{S}\otimes T_{S})\rightarrow H^{2}(S,I\otimes K_{S}\otimes T_{S})\rightarrow
H^{2}(S,K_{S}\otimes T_{S})\cong H^{0}(S,\Omega_{S}^{1}) .\nonumber \end{equation}
Then $h^{1,0}=0\Rightarrow Ext^{0}_{S}(I,I\otimes \Omega_{S}^{1})=0$.
\end{proof}
\begin{proposition}
Under the following assumptions
\begin{equation}\label{vanishing assumption}Ext^{0}_{S}(\mathcal{F},\mathcal{F}\otimes \Omega_{S}^{1})=0, \textrm{ }
Ext^{2}_{S}(\mathcal{F},\mathcal{F})=0,  \end{equation}
which are satisfied when (i) $S$ is del-Pezzo, $\mathcal{F}$ is an ideal sheaf of points on $S$ or (ii) when $S=\mathbb{P}^{2}$, $\mathcal{F}$ is slope-stable torsion-free on $S$, we have canonical isomorphisms
\begin{equation}Ext^{1}_{X}(\iota_{*}\mathcal{F},\iota_{*}\mathcal{F})\cong Ext^{1}_{S}(\mathcal{F},\mathcal{F}),\nonumber \end{equation}
\begin{equation}Ext^{2}_{X}(\iota_{*}\mathcal{F},\iota_{*}\mathcal{F})\cong Ext^{1}_{S}(\mathcal{F},\mathcal{F}\otimes \Omega_{S}^{1}).
\nonumber \end{equation}
\end{proposition}
\begin{proof}
By Lemma \ref{TS lemma1}, Lemma \ref{TS lemma2} and the degenerate spectral sequence.
\end{proof}
Let $\mathcal{M}_{c}(S,\mathcal{O}_{S}(1))$ be the Gieseker moduli space for $(S,\mathcal{O}_{S}(1))$ with Chern character $c\in H^{even}(S)$.
Under assumption (\ref{vanishing assumption}), we denote
\begin{equation}\mathcal{M}_{c}^{S_{cpn}}\triangleq\{\iota_{*}\mathcal{F} \textrm{ } | \mathcal{F}\in \mathcal{M}_{c}(S,\mathcal{O}_{S}(1))\}\cong \mathcal{M}_{c}(S,\mathcal{O}_{S}(1)) \nonumber \end{equation}
to be the component(s) of a moduli space of sheaves on $X$ which can be identified with the Gieseker moduli space $\mathcal{M}_{c}(S,\mathcal{O}_{S}(1))$.
We notice that $\mathcal{M}_{c}(S,\mathcal{O}_{S}(1))$ is smooth by assumption (\ref{vanishing assumption})
and use the philosophy of defining $DT_{4}$ virtual cycles (Definition \ref{virtual cycle when Mc smooth})
to define the virtual cycle of $\mathcal{M}_{c}^{S_{cpn}}$, i.e. we define the $DT_{4}$ virtual cycle of $\mathcal{M}_{c}^{S_{cpn}}$
%denoted by $[\mathcal{M}_{c}^{S_{cpn}}]^{vir}$
to be the Poincar\'{e} dual of the Euler class of a self-dual obstruction bundle.
\begin{proposition}
Under assumption (\ref{vanishing assumption}), $\mathcal{M}_{c}^{S_{cpn}}\cong \mathcal{M}_{c}(S,\mathcal{O}_{S}(1))$ and the virtual dimension $v.d_{\mathbb{R}}(\mathcal{M}_{c}^{S_{cpn}})\triangleq 2ext^{1}(\iota_{*}\mathcal{F},\iota_{*}\mathcal{F})-ext^{2}(\iota_{*}\mathcal{F},\iota_{*}\mathcal{F})<0$,
where $\mathcal{F}\in\mathcal{M}_{c}(S,\mathcal{O}_{S}(1))$.
\end{proposition}
\begin{proof}
By the Hirzebruch-Riemann-Roch theorem and the assumption (\ref{vanishing assumption}), we have
\begin{equation}dim_{\mathbb{C}} Ext^{1}_{S}(\mathcal{F},\mathcal{F}\otimes \Omega_{S}^{1})=2dim_{\mathbb{C}} Ext^{1}_{S}(\mathcal{F},\mathcal{F})+
r^{2}e(S)-2 ,\nonumber \end{equation}
where $r=rk(\mathcal{F})\geq1$, $e(S)$ is the Euler characteristic of $S$.
Thus,
\begin{equation}v.d_{\mathbb{R}}(\mathcal{M}_{c}^{S_{cpn}})\triangleq 2ext^{1}(\iota_{*}\mathcal{F},\iota_{*}\mathcal{F})-ext^{2}(\iota_{*}\mathcal{F},\iota_{*}\mathcal{F})=2-r^{2}e(S). \nonumber \end{equation}
Finally, as the trace map is surjective (Lemma 10.1.3 \cite{hl}), the assumption (\ref{vanishing assumption}) implies that $h^{0,i}(S)=0$ for $i=1,2$, and $v.d_{\mathbb{R}}(\mathcal{M}_{c}^{S_{cpn}})=2-r^{2}(2+h^{1,1})<0$.
\end{proof}
\begin{remark}
%Since $2-r^{2}e(S)$ is a topological invariant, the conclusion still holds even without the condition (\ref{vanishing assumption}).
The negative virtual dimension makes the $DT_{4}$ virtual cycle of $\mathcal{M}_{c}^{S_{cpn}}$ vanish.
\end{remark}
${}$ \\
\textbf{The reduced counting}.
More carefully, we note that the self-dual obstruction bundle has a trivial subbundle, which could make the virtual cycle vanish. We consider the trace-free part of the obstruction bundle and define the reduced $DT_{4}$ virtual cycle.
\begin{definition}\label{red vir cycle}
Let $X=T^{*}S$ and $(S,\mathcal{O}_{S}(1))$ be a projective surface. We assume sheaves in $\mathcal{M}_{c}(S,\mathcal{O}_{S}(1))$ satisfy (\ref{vanishing assumption}) and the self-dual trace-free obstruction bundle of $\mathcal{M}_{c}^{S_{cpn}}$ is orientable.

The reduced $DT_{4}$ virtual cycle of $\mathcal{M}_{c}^{S_{cpn}}$ is
the Poincar\'{e} dual of the Euler class of a self-dual trace-free obstruction bundle
\begin{equation}[\mathcal{M}_{c}^{S_{cpn}}]^{vir}_{red}\triangleq PD\big(e(Ob_{0,+})\big)\in
H_{r_{red}}(\mathcal{M}_{c}\big(S,\mathcal{O}_{S}(1)\big),\mathbb{Z}), \nonumber \end{equation}
where $Ob_{0,+}$ is a self-dual trace-free obstruction bundle, $r_{red}=(1-r^{2})(2+h^{1,1}(S))$ and $\mathcal{M}_{c}(S,\mathcal{O}_{S}(1))$ is a Gieseker moduli space for $(S,\mathcal{O}_{S}(1))$.
\end{definition}
%\begin{remark}
%The above Euler class involves a choice of an orientation on each connected component of $\mathcal{M}_{c}^{S_{cpn}}$. We will see for most interesting cases, a natural orientation exists.
%\end{remark}
We first show reduced virtual cycles for high rank sheaves in $\mathcal{M}_{c}(S,\mathcal{O}_{S}(1))$ vanish.
\begin{proposition}
$[\mathcal{M}_{c}^{S_{cpn}}]^{vir}_{red}=0$, if $c|_{H^{0}(S)}\geq2$.
\end{proposition}
\begin{proof}
The reduced virtual dimension $r_{red}=(1-r^{2})(2+h^{1,1}(S))< 0$ if $r\geq2$.
\end{proof}
Under assumption (\ref{vanishing assumption}) and $r=1$, we have $r_{red}=0$.
The corresponding reduced $DT_{4}$ virtual cycle is zero dimensional.
For ideal sheaves of curves on $S$ (line bundles on $S$),
\begin{equation}Ext^{1}_{S}(\mathcal{F},\mathcal{F})\cong H^{1}(S,\mathcal{O})=0.  \nonumber \end{equation}
%which shows that both the tangent space and the reduced obstruction space are zero.
Then the moduli space is just one point and the reduced $DT_{4}$ invariant is 1 in this case.
\begin{proposition}
$[\mathcal{M}_{c}^{S_{cpn}}]^{vir}_{red}=1$, where $c=(1,c|_{H^{2}(S)},0)$.
\end{proposition}
For ideal sheaves of points on $S$, we have
\begin{proposition}
Let $S$ be a projective surface with $h^{0,i}(S)=0$, $i=1,2$. We take $I$ to be an ideal sheaf of points on $S$,
then we have a canonical isomorphism
\begin{equation}Ext^{1}_{S}(I,I\otimes \Omega_{S}^{1})_{0}\cong Ext^{1}_{S}(\mathcal{O}_{Z},\mathcal{O}_{Z}\otimes \Omega_{S}^{1}).
\nonumber\end{equation}
%Furthermore, under this identification, $Ext^{1}_{S}(\mathcal{O}_{Z},\mathcal{O}_{Z})$ is
%a maximal isotropic subspace with respect to the Serre duality pairing.
\end{proposition}
\begin{proof}
We denote an ideal sheaf of $n$-points on $S$ by $I$. Taking cohomology of the short exact sequence of sheaves
\begin{equation}0\rightarrow I\otimes \Omega_{S}^{1}\rightarrow\Omega_{S}^{1}\rightarrow \mathcal{O}_{Z}\otimes \Omega_{S}^{1}\rightarrow 0,
\nonumber \end{equation}
where $\mathcal{O}_{Z}$ is the structure sheaf of $n$-points, we have
\begin{equation}\label{cotangent of S equ 0} 0\rightarrow H^{0}(S,\mathcal{O}_{Z}\otimes \Omega_{S}^{1})\rightarrow H^{1}(S,I\otimes \Omega_{S}^{1})\rightarrow H^{1}(S,\Omega_{S}^{1})\rightarrow 0, \end{equation}
and
\begin{equation}\label{cotangent of S equ 1} H^{2}(S,I\otimes\Omega_{S}^{1})\cong H^{2}(S,\Omega_{S}^{1})\cong H^{0}(S,\Omega_{S}^{1})=0.\end{equation}
Applying $Hom_{\mathcal{O}_{S}}(\mathcal{O}_{Z},\cdot)$ to
\begin{equation}0\rightarrow I\otimes \Omega_{S}^{1}\rightarrow\Omega_{S}^{1}\rightarrow \mathcal{O}_{Z}\otimes \Omega_{S}^{1}\rightarrow 0,
\nonumber \end{equation}
we get
\begin{equation}\label{cotangent of S equ 2}Ext^{0}_{S}(\mathcal{O}_{Z},I\otimes \Omega_{S}^{1})=0,
Ext^{0}_{S}(\mathcal{O}_{Z},\mathcal{O}_{Z}\otimes \Omega_{S}^{1})\cong Ext^{1}_{S}(\mathcal{O}_{Z},I\otimes \Omega_{S}^{1}), \end{equation}
\begin{equation}Ext^{1}_{S}(\mathcal{O}_{Z},\mathcal{O}_{Z}\otimes \Omega_{S}^{1})\cong Ext^{2}_{S}(\mathcal{O}_{Z},I\otimes \Omega_{S}^{1}).
\nonumber \end{equation}
Applying $Hom_{\mathcal{O}_{S}}(\cdot,I\otimes \Omega_{S}^{1})$ to
\begin{equation}0\rightarrow I\rightarrow\mathcal{O}_{S}\rightarrow \mathcal{O}_{Z} \rightarrow 0 ,
\nonumber \end{equation}
we have
\begin{equation}\rightarrow Ext^{i}_{S}(\mathcal{O}_{Z},I\otimes \Omega_{S}^{1})\rightarrow
Ext^{i}_{S}(\mathcal{O}_{S},I\otimes \Omega_{S}^{1})\rightarrow  Ext^{i}_{S}(I,I\otimes \Omega_{S}^{1})\rightarrow.\nonumber \end{equation}
By the condition $h^{0,1}(S)=0$, we have $Ext^{0}_{S}(I,I\otimes \Omega_{S}^{1})=0$ by Lemma \ref{TS lemma2}.
Using (\ref{cotangent of S equ 1}), (\ref{cotangent of S equ 2}), we can get
\begin{equation}0\rightarrow Ext^{0}_{S}(\mathcal{O}_{Z},\mathcal{O}_{Z}\otimes \Omega_{S}^{1})\rightarrow
H^{1}(S,I\otimes\Omega_{S}^{1})\rightarrow \nonumber \end{equation}
\begin{equation}\rightarrow Ext^{1}_{S}(I,I\otimes \Omega_{S}^{1})\rightarrow Ext^{1}_{S}(\mathcal{O}_{Z},\mathcal{O}_{Z}\otimes \Omega_{S}^{1})\rightarrow 0 .\nonumber \end{equation}
By (\ref{cotangent of S equ 0}), we get
\begin{equation}0\rightarrow H^{1}(S,\Omega_{S}^{1})\rightarrow Ext^{1}_{S}(I,I\otimes \Omega_{S}^{1})\rightarrow Ext^{1}_{S}(\mathcal{O}_{Z},\mathcal{O}_{Z}\otimes \Omega_{S}^{1})\rightarrow0.\nonumber\end{equation}
where the first injective map is the inclusion of the trace factor.
%Considering the Serre duality pairing,
%\begin{equation}Ext^{1}_{S}(I,I\otimes \Omega_{S}^{1})_{0}\otimes Ext^{1}_{S}(I,I\otimes \Omega_{S}^{1})_{0}\rightarrow
%Ext^{2}_{S}(I,I\otimes \Omega_{S}^{2})\rightarrow H^{2,2}(S)\nonumber \end{equation}
%can be identified with
%\begin{equation}Ext^{1}_{S}(\mathcal{O}_{Z},\mathcal{O}_{Z}\otimes \Omega_{S}^{1})\otimes Ext^{1}_{S}(\mathcal{O}_{Z},\mathcal{O}_{Z}\otimes \Omega_{S}^{1})\rightarrow Ext^{2}_{S}(\mathcal{O}_{Z},\mathcal{O}_{Z}\otimes \Omega_{S}^{2})\rightarrow H^{2,2}(S), \nonumber \end{equation}
%where the last map is taking trace. Furthermore it can be identified with
%\begin{equation}(Ext^{1}_{S}(\mathcal{O}_{Z},\mathcal{O}_{Z})\oplus Ext^{1}_{S}(\mathcal{O}_{Z},\mathcal{O}_{Z}\otimes \Omega_{S}^{2}))\otimes
%(Ext^{1}_{S}(\mathcal{O}_{Z},\mathcal{O}_{Z})\oplus Ext^{1}_{S}(\mathcal{O}_{Z},\mathcal{O}_{Z}\otimes \Omega_{S}^{2}))
%\rightarrow \mathbb{C}, \nonumber \end{equation}
%as
%\begin{equation}Ext^{1}_{S}(\mathcal{O}_{Z},\mathcal{O}_{Z})\otimes Ext^{1}_{S}(\mathcal{O}_{Z},\mathcal{O}_{Z})\rightarrow
%Ext^{2}_{S}(\mathcal{O}_{Z},\mathcal{O}_{Z})\rightarrow H^{2}(S,\mathcal{O}_{S})=0, \nonumber \end{equation}
%$Ext^{1}_{S}(\mathcal{O}_{Z},\mathcal{O}_{Z})$ is a maximal isotropic subspace with respect to the Serre duality pairing.
\end{proof}
%Thus, after taking away the trivial factor $H^{1}(S,\Omega_{S}^{1})$, the maximal isotropic sub-bundle of the reduced obstruction bundle exists and can be identified with the tangent bundle of Hilbert scheme of points on $S$. Note that this gives a natural orientation on the self-dual trace-free obstruction bundle.
%By Lemma \ref{ASD equivalent to max isotropic}, the reduced $DT_{4}$ virtual cycle can be identified with the Euler characteristic of Hilbert scheme of $n$-points on $S$.
%\begin{theorem}\label{DT4 of cotangent bundle of S}
%We take $X=T^{*}S$ and $c=(1,0,-n)$, where $S$ is a compact algebraic surface with $q(S)=0$ and $n\geq 1$. We assume sheaves in $\mathcal{M}_{c}(S,\mathcal{O}_{S}(1))$ satisfy (\ref{vanishing assumption}) which is true when $S$ is del-Pezzo. Then the self-dual trace-free obstruction bundle (Definition \ref{red vir cycle}) has a natural complex orientation and
%\begin{equation}[\mathcal{M}_{c}^{S_{cpn}}]^{vir}_{red}=e(Hilb^{n}(S)). \nonumber \end{equation}
%Furthermore, they fit into the following generating function
%\begin{equation}\sum_{n\geq0}[\mathcal{M}_{(1,0,-n)}^{S_{cpn}}]^{vir}_{red}q^{n}=\prod_{k\geq1}\frac{1}{(1-q^{k})^{e(S)}}_{.} \nonumber \end{equation}
%\end{theorem}
%\begin{proof}
%By the above discussion and \cite{cheah}.
%\end{proof}
\begin{remark}
We expect the reduced $DT_{4}$ virtual cycles will still be zero from a calculation for ideal sheaves of one point.
\end{remark}
The vanishing of (reduced) $DT_{4}$ virtual cycles should not upset readers as it
might be a hint for the following question on the invariance of $DT_{4}$ invariants under Mukai flops.

If we take a Calabi-Yau 4-fold $X$ which contains an embedded $\mathbb{P}^{2}$ with normal bundle to be $T^{*}\mathbb{P}^{2}$,
there is a birational transformation
\begin{equation}\phi: X\rightarrow X^{+}, \nonumber \end{equation}
called Mukai flop given by the composition of the blow-up of $X$ along the $\mathbb{P}^{2}$ and the blow-down of the exceptional divisor in the other ruling \cite{mukai}. The resulting 4-fold $X^{+}$ is still Calabi-Yau and contains an embedded $\mathbb{P}^{2}$.
Outside these $\mathbb{P}^{2}$'s, $X^{+}\backslash\mathbb{P}^{2}$ is isomorphic to $X\backslash\mathbb{P}^{2}$. Near these $\mathbb{P}^{2}$'s, the (local) $DT_{4}$ invariants for $T^{*}\mathbb{P}^{2}$ are zero by the above calculations.
%Meanwhile, as the normal bundle of $\mathbb{P}^{2}$ is its cotangent bundle, $\mathbb{P}^{2}$ can be blow-down. So heuristically speaking, $DT_{4}$ invariants will not receive contributions from a neighbourhood of $\mathbb{P}^{2}$ as the (local) $DT_{4}$ invariants for $T^{*}\mathbb{P}^{2}$ are zero by earlier calculations.
We therefore ask the following question.
\begin{question}
Are $DT_{4}$ invariants are invariant under Mukai flops ?  \\
More precisely,
given a Mukai flop $\phi: X^{+}\rightarrow X$ between two projective Calabi-Yau 4-folds, cohomology classes $c=(1,*,*,*,*)\in H^{even}(X)$, an ample line bundle $\mathcal{O}_{X}(1)$ and an orientation data $\mathcal{O}(\mathcal{L})$ with respect to $(X,\mathcal{O}_{X}(1),c)$, does the equality
\begin{equation}DT_{4}(X,\mathcal{O}_{X}(1),c,\mathcal{O}(\mathcal{L}))=DT_{4}(X^{+},\mathcal{O}_{X^{+}}(1),\phi^{*}c,\mathcal{O}(\mathcal{L}^{+}))  \nonumber \end{equation}
hold for an isomorphism
$\phi^{*}: H^{*}(X,\mathbb{Z})\cong H^{*}(X^{+},\mathbb{Z})$ of cohomology rings (see e.g. Hu and Zhang \cite{huzhang}),
any ample line bundle $\mathcal{O}_{X^{+}}(1)$ on $X^{+}$, and certain orientation data $\mathcal{O}(\mathcal{L}^{+})$ with respect to $(X^{+},\mathcal{O}_{X^{+}}(1),\phi^{*}c)$ ?
\end{question}
%\begin{question}
%Given a Mukai flop $\phi: X^{+}\rightarrow X$ between two projective Calabi-Yau 4-folds, cohomology classes $c=(1,*,*,*,*)\in H^{even}(X)$, an ample line bundle $\mathcal{O}_{X}(1)$ and an orientation data $\mathcal{O}(\mathcal{L})$ with respect to $(X,\mathcal{O}_{X}(1),c)$. Then
%\begin{equation}DT_{4}(X,\mathcal{O}_{X}(1),c,\mathcal{O}(\mathcal{L}))=DT_{4}(X^{+},\mathcal{O}_{X^{+}}(1),\phi^{*}c,\mathcal{O}(\mathcal{L}^{+})), \nonumber \end{equation}
%where $\phi^{*}: H^{*}(X,\mathbb{Z})\cong H^{*}(X^{+},\mathbb{Z})$ is the isomorphism of cohomology rings (see e.g. Hu and Zhang \cite{huzhang}),
%$\mathcal{O}_{X^{+}}(1)$ is any ample line bundle on $X^{+}$, and $\mathcal{O}(\mathcal{L}^{+})$ is certain orientation data with respect to $(X^{+},\mathcal{O}_{X^{+}}(1),\phi^{*}c)$.
%\end{question}
\begin{remark} ${}$ \\
1. By the work of Huybrechts \cite{h}, if $X$ is a projective hyper-K\"{a}hler 4-fold, its Mukai flop $X^{+}$ is deformation equivalent to $X$. The invariance can be deduced from the deformation invariance of $DT_{4}$ invariants. \\
2. For general Chern character $c$, we should have a similar result for a suitable polarization $\mathcal{O}_{X^{+}}(1)$. This would rely on a good understanding of wall-crossing formulas for $DT_{4}$ invariants. \\
3. By the works of Kawamata \cite{kawamata} and Namikawa \cite{namikawa}, there is an equivalence of categories $D^{b}(X^{+})\cong D^{b}(X)$ given by Fourier-Mukai transformations. We expect this equivalence and wall-crossing formulas for $DT_{4}$ invariants will answer this question.
\end{remark}

\subsection{The case of $X=Tot(L_{1}\oplus L_{2}\rightarrow S)$ and G\"{o}ttsche-Carlsson-Okounkov formula}
Parallel to the previous subsection, we consider counting torsion sheaves on the total space
$X=Tot(L_{1}\oplus L_{2}\rightarrow S)$ of some rank two bundle (such that $K_{S}\cong L_{1}\otimes L_{2}$) over a projective surface $S$,
which are scheme theoretically supported on the zero section $S$.

We take $\mathcal{F}$ to be a torsion-free sheaf on $S$ and
denote $\iota: S\rightarrow X$ to be the inclusion of the zero section.
The comparison of deformation-obstruction theories for $\iota_{*}\mathcal{F}$ on $X$ and
$\mathcal{F}$ on $S$ can be similarly obtained as before.
\begin{proposition}\label{prop on local S}
Let $\mathcal{F}$ be a torsion-free slope stable sheaf on a projective surface $(S,\mathcal{O}_{S}(1))$
and $\iota_{*}\mathcal{F}$ be its push-forward to
$X=Tot(L_{1}\oplus L_{2}\rightarrow S)$ with $K_{S}\cong L_{1}\otimes L_{2}$. \\
If $L_{i}\cdot\mathcal{O}_{S}(1)<0$ ($i=1,2$) and $K_{S}\cdot\mathcal{O}_{S}(1)<0$, then there exist canonical isomorphisms
\begin{equation}Ext^{1}_{X}(\iota_{*}\mathcal{F},\iota_{*}\mathcal{F})\cong Ext^{1}_{S}(\mathcal{F},\mathcal{F}), \nonumber \end{equation}
\begin{equation}Ext^{2}_{X}(\iota_{*}\mathcal{F},\iota_{*}\mathcal{F})\cong
Ext^{1}_{S}(\mathcal{F},\mathcal{F}\otimes L_{1})\oplus
Ext^{1}_{S}(\mathcal{F},\mathcal{F}\otimes L_{1})^{*},  \nonumber \end{equation}
under which $Ext^{1}_{S}(\mathcal{F},\mathcal{F}\otimes L_{1})$ is a
maximal isotropic subspace of $Ext^{2}_{X}(\iota_{*}\mathcal{F},\iota_{*}\mathcal{F})$ with respect to the Serre duality pairing.
\end{proposition}
\begin{proof}
As $\mathcal{F}$ is torsion-free slope stable, the tensor $\mathcal{F}\otimes L$ with a line bundle is also slope stable.
We have $Ext^{0}_{S}(\mathcal{F},\mathcal{F}\otimes L)=0$ if $L\cdot\mathcal{O}_{S}(1)<0$
(see e.g. Proposition 1.2.7 of \cite{hl}).
One can then easily show the spectral sequence
\begin{equation}E_{2}=Ext^{*}_{S}(\mathcal{F},\mathcal{F}\otimes\wedge^{*}
(L_{1}\oplus L_{2}))\Rightarrow E_{\infty}=
Ext^{*}_{X}(\iota_{*}\mathcal{F},\iota_{*}\mathcal{F})\nonumber \end{equation}
degenerates at $E_{2}$ page.
\end{proof}
We similarly denote
\begin{equation}\mathcal{M}_{c}^{S_{cpn}}=\{\iota_{*}\mathcal{F}\textrm{ } |\textrm{ } \mathcal{F}\in \mathcal{M}_{c}(S,\mathcal{O}_{S}(1))  \}
\cong \mathcal{M}_{c}(S,\mathcal{O}_{S}(1)) \nonumber \end{equation}
to be the component(s) of a moduli of sheaves on $X$ which can be identified with the Gieseker moduli space
$\mathcal{M}_{c}(S,\mathcal{O}_{S}(1))$ for $(S,\mathcal{O}_{S}(1))$ with Chern character $c\in H^{even}(S)$
(we assume every Gieseker semi-stable sheaf is slope stable).
As $\mathcal{M}_{c}(S,\mathcal{O}_{S}(1))$ is smooth
\footnote{With the set-up of Proposition \ref{prop on local S}, we have
$Ext^{2}(\mathcal{F},\mathcal{F})\cong Ext^{0}(\mathcal{F},\mathcal{F}\otimes K_{S})^{*}=0$.},
the obstruction sheaf $Ob$ of $\mathcal{M}_{c}^{S_{cpn}}$ is a vector bundle endowed with
a non-degenerate quadratic form (Serre duality pairing).

By Proposition \ref{prop on local S}, the obstruction bundle has an maximal isotropic subbundle whose fiber over $\mathcal{F}$
can be identified with $Ext^{1}_{S}(\mathcal{F},\mathcal{F}\otimes L_{1})$. With respect to the natural complex orientation, we identify
a self-dual obstruction bundle $Ob_{+}$ with the maximal isotropic subbundle.
\begin{definition}\label{def of DT4 vc for Tot of rank2}
Let $(S,\mathcal{O}_{S}(1))$ be a projetive surface and $L_{i}$ ($i=1,2$) be two holomorphic line bundles on $S$ such that
$K_{S}\cong L_{1}\otimes L_{2}$, $L_{i}\cdot\mathcal{O}_{S}(1)<0$ ($i=1,2$) and $K_{S}\cdot\mathcal{O}_{S}(1)<0$.
Let $\mathcal{M}_{c}(S,\mathcal{O}_{S}(1))$ be a Gieseker moduli space with Chern character $c\in H^{even}(S)$ such that every Giseker semi-stable
sheaf is slope stable.

The $DT_{4}$ virtual cycle $[\mathcal{M}_{c}^{S_{cpn}}]^{vir}$ of $\mathcal{M}_{c}^{S_{cpn}}$ (with respect to the natural complex orientation)
is the Poincar\'{e} dual of the Euler class of the maximal isotropic subbundle $Ob_{+}$.
\end{definition}
\begin{remark}
The real virtual dimension of $\mathcal{M}_{c}^{S_{cpn}}$ is
\begin{equation}v.d_{\mathbb{R}}(\mathcal{M}_{c}^{S_{cpn}})=2ext^{1}(\iota_{*}\mathcal{F},\iota_{*}\mathcal{F})-
ext^{2}(\iota_{*}\mathcal{F},\iota_{*}\mathcal{F})
=2(1-\frac{r^{2}}{2}L_{1}\cdot L_{2}), \nonumber \end{equation}
where $r=c|_{H^{0}(S)}$ is the rank of $\mathcal{F}$.
\end{remark}
Note that we always have $v.d_{\mathbb{R}}(\mathcal{M}_{c}^{S_{cpn}})\leq0$ if $h^{1,0}(S)=0$.
\begin{lemma}
With the same set-up as in Definition \ref{def of DT4 vc for Tot of rank2}, we have
\begin{equation}L_{1}\cdot L_{2}=2(h^{1}(S,L_{1})+\chi(\mathcal{O}_{S})). \nonumber \end{equation}
Furthermore, $L_{1}\cdot L_{2}\geq 2$ if $h^{1,0}(S)=0$.
\end{lemma}
\begin{proof}
We take $\mathcal{F}=L$ to be a line bundle, then $Ext_{S}^{*}(\mathcal{F},\mathcal{F}\otimes L_{1})\cong H^{*}(S,L_{1})$.
Meanwhile, by $Ext_{S}^{0}(\mathcal{F},\mathcal{F}\otimes L_{1})=0$ and
$Ext_{S}^{2}(\mathcal{F},\mathcal{F}\otimes L_{1})\cong Ext_{S}^{0}(\mathcal{F},\mathcal{F}\otimes L_{2})^{*}=0$ (Proposition 1.2.7 \cite{hl}),
we have $H^{0}(L_{i})=H^{2}(L_{i})=0$ for $i=1,2$. From the Hirzebruch-Riemann-Roch theorem, we obtain
\begin{equation}-h^{1}(S,L_{1})=\chi(\mathcal{O}_{S})+\frac{1}{2}L_{1}\cdot L_{1}-\frac{1}{2}L_{1}\cdot K_{S}. \nonumber \end{equation}
Using the relation $K_{S}\cong L_{1}\otimes L_{2}$, we finally get $L_{1}\cdot L_{2}=2(h^{1}(S,L_{1})+\chi(\mathcal{O}_{S}))$.
\end{proof}
When the virtual dimension is negative, we have vanishing of virtual cycles.
\begin{corollary}
If $c|_{H^{0}(S)}\geq2$ or $h^{1}(S,L_{1})\geq1$, then $[\mathcal{M}_{c}^{S_{cpn}}]^{vir}=0$ provided that $h^{1,0}(S)=0$.
\end{corollary}
%We now restrict to rank one case ($r=1$).
From another perspective, we note that the self-dual obstruction bundle $Ob_{+}$ has a surjective morphism
\begin{equation}tr: Ob_{+}\twoheadrightarrow (\mathcal{O}_{\mathcal{M}_{c}^{S_{cpn}}})^{\oplus h^{1}(S,L_{1})},  \nonumber \end{equation}
\begin{equation}tr: Ob_{+}|_{\mathcal{F}}=Ext_{S}^{1}(\mathcal{F},\mathcal{F}\otimes L_{1})\rightarrow H^{1}(S,L_{1}),   \nonumber \end{equation}
which is given by the trace map (see Lemma 10.1.3 \cite{hl}). If $h^{1}(S,L_{1})\geq1$, $Ob_{+}$ has an trivial subbundle
which makes the $DT_{4}$ virtual cycle vanish.
\begin{definition}
With the same set-up as in Definition \ref{def of DT4 vc for Tot of rank2} and further assuming $h^{1,0}(S)=0$,
the reduced self-dual obstruction bundle is the kernel $Ker(tr)$ of the above trace map. The reduced $DT_{4}$ virtual cycle
$[\mathcal{M}_{c}^{S_{cpn}}]_{red}^{vir}$ is its Euler class.
\begin{remark}
When $h^{1}(S,L_{1})=0$ (e.g. $S=\mathbb{P}^{2}$, $L_{1}=\mathcal{O}_{\mathbb{P}^{2}}(-1)$),
$[\mathcal{M}_{c}^{S_{cpn}}]_{red}^{vir}=[\mathcal{M}_{c}^{S_{cpn}}]^{vir}$.
\end{remark}
\end{definition}
 \begin{proposition}\label{rk 2 bundle over S}
The reduced $DT_{4}$ virtual cycle $[\mathcal{M}_{c}^{S_{cpn}}]_{red}^{vir}$ satisfies \\
(i) $[\mathcal{M}_{c}^{S_{cpn}}]_{red}^{vir}=0$, if $c|_{H^{0}(S)}\geq2$;  \\
(ii) $[\mathcal{M}_{c}^{S_{cpn}}]_{red}^{vir}\in H_{0}(\mathcal{M}_{c}^{S_{cpn}})$, if $c|_{H^{0}(S)}=1$.
Furthermore, with respect to the natural complex orientation, we have
\begin{equation}\sum_{n\geq0}[\mathcal{M}_{(1,0,-n)}^{S_{cpn}}]^{vir}q^{n}=\prod_{k\geq1}(\frac{1}{1-q^{k}})^{e(TS\otimes L_{1})}. \nonumber \end{equation}
\end{proposition}
\begin{proof}
The degree of $[\mathcal{M}_{c}^{S_{cpn}}]_{red}^{vir}$ is
$2(1-r^{2})(1+h^{1}(S,L_{1}))$, where $r=c|_{H^{0}(\mathbb{P}^{2})}$.

If $c=(1,0,-n)$, $\mathcal{M}_{c}^{S_{cpn}}\cong Hilb^{n}(S)$. It is easy to check the reduced self-dual obstruction bundle $Ker(tr)$
is the twisted tangent bundle $T(Hilb^{n}(S),L_{1})$ of $Hilb^{n}(S)$ whose
Euler class can be summarized by the following G\"{o}ttsche-Carlsson-Okounkov formula (see \cite{carlssonokounkov})
\begin{equation} \sum_{n\geq0}e(T(Hilb^{n}(S),L))q^{n}=
\prod_{k\geq1}\big(\frac{1}{1-q^{k}}\big)^{e(TS\otimes L)}. \nonumber \end{equation}
\end{proof}

\section{Computational examples}
We compute $DT_{4}$ invariants for examples when $\mathcal{M}_{c}$'s are smooth in this section.
By Definition \ref{virtual cycle when Mc smooth}, the virtual fundamental class of $\overline{\mathcal{M}}_{c}^{DT_{4}}$ is the
Poincar\'{e} dual of the Euler class of the self-dual obstruction bundle which has its origin in the theory of characteristic classes \cite{eg}.

We take a complex vector bundle with a non-degenerate quadratic form, $(V,q)$ on a projective manifold $X$. We assume the structure group of the bundle can be reduced to $SO(n,\mathbb{C})$, $n=rk(V)$.
By the homotopy equivalence $SO(n,\mathbb{C})\sim SO(n,\mathbb{R})$,
%\begin{equation} SO(n,\mathbb{C})\sim SO(n,\mathbb{R}),\nonumber\end{equation}
we have $H^{*}(BSO(n,\mathbb{C}))\cong H^{*}(BSO(n,\mathbb{R}))$.
%\begin{equation} H^{*}(BSO(n,\mathbb{C});\mathbb{Q})\cong H^{*}(BSO(n,\mathbb{R});\mathbb{Q}).\nonumber\end{equation}
Meanwhile,
\begin{equation}H^{*}(BSO(2r,\mathbb{C});\mathbb{Q})\cong \mathbb{Q}[p_{1},...p_{r-1},e],\nonumber\end{equation}
\begin{equation}H^{*}(BSO(2r+1,\mathbb{C});\mathbb{Q})\cong \mathbb{Q}[p_{1},...p_{r}],\nonumber\end{equation}
where $p_{i}=(-1)^{i}c_{2i}(\omega^{+}\otimes \mathbb{C})$, $\omega^{+}$ is the universal $SO(n)$ bundle and
$e=e(\omega^{+})$ is called the half Euler class of $(V,q)$ \cite{switzer}.

When $\mathcal{M}_{c}$ is smooth and $(V,q)=(Ob,Q_{Serre})$, where $Ob$ is the obstruction bundle with $Ob|_{\mathcal{F}}\cong Ext^{2}(\mathcal{F},\mathcal{F})$, $Q_{Serre}$ is the
Serre duality pairing, then $\overline{\mathcal{M}}_{c}^{DT_{4}}$ exists and $\overline{\mathcal{M}}_{c}^{DT_{4}}=\mathcal{M}_{c}$. Furthermore, $[\overline{\mathcal{M}}_{c}^{DT_{4}}]^{vir}=PD(e)$ with an appropriate choice of orientations.

%By Lemma \ref{ASD equivalent to max isotropic}, if there exists a maximal isotropic subbundle of $Ob$, its top Chern class will be the half Euler class.

\subsection{Li-Qin's examples}
We have examples when $\overline{\mathcal{M}}_{c}=\mathcal{M}_{c}^{bdl}$, where moduli spaces consist of rank two bundles coming from non-trivial extensions of line bundles. The construction is due to W. P. Li and Z.Qin \cite{lq}.

Let $X$ be a generic smooth hyperplane section in $\mathbb{P}^{1}\times\mathbb{P}^{4}$ of bi-degree $(2,5)$. Take
\begin{equation}cl=[1+(-1,1)|_{X}]\cdot[1+(\epsilon_{1}+1,\epsilon_{2}-1)|_{X}],\nonumber \end{equation}
\begin{equation}k=(1+\epsilon_{1})\left(\begin{array}{l}6-\epsilon_{2} \\ \quad 4\end{array}\right), \quad \epsilon_{1},\epsilon_{2}=0,1,
\quad L_{r}=\mathcal{O}_{\mathbb{P}^{1}\times\mathbb{P}^{4}}(1,r)|_{X}.  \nonumber\end{equation}
We define $\overline{\mathcal{M}}_{c}(L_{r})$ to be the moduli space of Gieseker $L_{r}$-semi-stable rank two torsion-free sheaves
with Chern character $c$ (which can be easily read from the total Chern class $cl$).
\begin{lemma}(Li-Qin, Theorem 5.7 \cite{lq}) \label{lq example} \\
The moduli space of rank two bundles on $X$ with the given Chern class stated above satisfies the following properties, \\
(i) The moduli space is isomorphic to $\mathbb{P}^{k}$ and consists of all the rank two bundles in the nonsplitting extensions
\begin{equation}0\rightarrow \mathcal{O}_{X}(-1,1)\rightarrow E\rightarrow \mathcal{O}_{X}(\epsilon_{1}+1,\epsilon_{2}-1)\rightarrow 0,
\nonumber\end{equation}
when
\begin{equation}\frac{15(2-\epsilon_{2})}{6+5\epsilon_{1}+2\epsilon_{2}}<r<\frac{15(2-\epsilon_{2})}{\epsilon_{1}(1+2\epsilon_{2})}.
\nonumber\end{equation}
(ii) $\overline{\mathcal{M}}_{c}(L_{r})$  is empty when
\begin{equation} 0<r<\frac{15(2-\epsilon_{2})}{6+5\epsilon_{1}+2\epsilon_{2}}.\nonumber\end{equation}
\end{lemma}
By the Hirzebruch-Riemann-Roch theorem,
\begin{equation}\epsilon_{1}=0,\textrm{ }\epsilon_{2}=1 \Rightarrow \chi(E,E)=-6, \textrm{ }\textrm{ }k=4 ,\nonumber\end{equation}
\begin{equation}\textrm{ }\epsilon_{1}=1,\textrm{ }\epsilon_{2}=1 \Rightarrow \chi(E,E)=-16, \textrm{ } k=9, \nonumber\end{equation}
\begin{equation}\textrm{ }\textrm{ }\epsilon_{1}=0,\textrm{ }\epsilon_{2}=0 \Rightarrow \chi(E,E)=-26,\textrm{ } k=14, \nonumber\end{equation}
\begin{equation}\textrm{ }\textrm{ }\epsilon_{1}=1,\textrm{ }\epsilon_{2}=0 \Rightarrow \chi(E,E)=-56,\textrm{ } k=29.\nonumber\end{equation}
By Lemma 5.2 \cite{lq} and simple computations, $k=dim Ext^{1}(E,E)$, $Ext^{2}(E,E)=0$ in all above four cases.
Thus $\overline{\mathcal{M}}_{c}^{DT_{4}}$ exists and $\overline{\mathcal{M}}_{c}^{DT_{4}}=\overline{\mathcal{M}}_{c}(L_{r})$ is compact smooth whose
virtual fundamental class is the usual fundamental class of $\overline{\mathcal{M}}_{c}(L_{r})$.
%\textbf{The wall crossing phenomenon}.
%By Lemma \ref{lq example}, when the parameter $r$ is small,
%the moduli space is empty and $DT_{4}$ invariants are zero. When $r$ crosses the critical value
%$\frac{15(2-\epsilon_{2})}{\epsilon_{1}(1+2\epsilon_{2})}$, the virtual cycle will be nontrivial and produces nonzero invariants.
%Hence, wall-crossing phenomenon exists in $DT_{4}$ theory.  \\
\begin{theorem}\label{liqin eg}
Let $X$ be a generic smooth hyperplane section in $\mathbb{P}^{1}\times\mathbb{P}^{4}$ of $(2,5)$ type.
Let
\begin{equation}cl=[1+(-1,1)|_{X}]\cdot[1+(\epsilon_{1}+1,\epsilon_{2}-1)|_{X}],
\nonumber
\end{equation}
\begin{equation}k=(1+\epsilon_{1})\left(\begin{array}{l}6-\epsilon_{2} \\ \quad 4\end{array}\right), \quad \epsilon_{1},\epsilon_{2}=0,1, \quad
L_{r}=\mathcal{O}_{\mathbb{P}^{1}\times\mathbb{P}^{4}}(1,r)|_{X}. \nonumber\end{equation}
Denote $\overline{\mathcal{M}}_{c}(L_{r})$ to be the moduli space of Gieseker $L_{r}$-semi-stable rank two
torsion-free sheaves with Chern character $c$ (which can be easily read from the total Chern class $cl$). \\
$(1)$ If
\begin{equation}\frac{15(2-\epsilon_{2})}{6+5\epsilon_{1}+2\epsilon_{2}}<r<\frac{15(2-\epsilon_{2})}{\epsilon_{1}(1+2\epsilon_{2})},
\nonumber\end{equation}
then $\overline{\mathcal{M}}_{c}^{DT_{4}}$ exists and $\overline{\mathcal{M}}_{c}^{DT_{4}}\cong\overline{\mathcal{M}}_{c}(L_{r})\cong\mathbb{P}^{k}$, $[\overline{\mathcal{M}}_{c}^{DT_{4}}]^{vir}=[\mathbb{P}^{k}]$. \\
${}$ \\
$(2)$ If
\begin{equation} 0<r<\frac{15(2-\epsilon_{2})}{6+5\epsilon_{1}+2\epsilon_{2}},\nonumber\end{equation}
then $\overline{\mathcal{M}}_{c}^{DT_{4}}=\emptyset $ and $[\overline{\mathcal{M}}_{c}^{DT_{4}}]^{vir}=0$. \\
\end{theorem}
\begin{remark}
Wall-crossing phenomenon exists in $DT_{4}$ theory.
\end{remark}
Using $\mu$-maps to define corresponding $DT_{4}$ invariants, we need the universal bundle of the moduli space which comes from the universal extension \cite{lange}
\begin{equation}0\rightarrow \pi_{1}^{*}L_{2}\otimes \pi_{2}^{*}\mathcal{O}_{X}(1)\rightarrow\mathcal{E}\rightarrow\pi_{1}^{*}L_{1}\rightarrow 0,
\nonumber\end{equation}
\begin{equation}L_{1}=\mathcal{O}_{X}(\epsilon_{1}+1,\epsilon_{2}-1), L_{2}=\mathcal{O}_{X}(-1,1), \nonumber\end{equation}
where $\pi_{1}: X\times\overline{\mathcal{M}}_{c}(L_{r})\rightarrow X$,
$\pi_{2}: X\times\overline{\mathcal{M}}_{c}(L_{r})\rightarrow \overline{\mathcal{M}}_{c}(L_{r})$ are projection maps.
The Chern class of the universal bundle $\mathcal{E}$ is given by $\big(1+\pi_{1}^{*}c_{1}(L_{1})\big)\big(1+\pi_{1}^{*}c_{1}(L_{2})+\pi_{2}^{*}c_{1}(\mathcal{O}_{X}(1))\big)$.
We can then compute all $DT_{4}$ invariants in these examples.

\subsection{$DT_{4}/GW$ correspondence in some special cases }
It is well-known that we have equivalence between
Donaldson-Thomas ideal sheaves invariants and Gromov-Witten invariants on Calabi-Yau 3-folds \cite{briDTPT} \cite{moop} \cite{pandpixton} \cite{toda}. One may expect, there will be a similar gauge-string duality for Calabi-Yau 4-folds. We will study such a correspondence in some special cases in this section.
We fix $c=1-PD(\beta)-nPD(1)$, where $\beta\in H_{2}(X,\mathbb{Z})$ and $n\in \mathbb{Z}$.
%$n=\chi (\mathcal{O}_{C})$, $\mathcal{O}_{C}$ is the structure sheaf of a curve $C$.
We first compute the real virtual dimension of $\overline{\mathcal{M}}_{c}^{DT_{4}}$, which is defined to be $2ext^{1}(I_{C},I_{C})-ext^{2}(I_{C},I_{C})$, $I_{C}\in \mathcal{M}_{c}$.
\begin{lemma}\label{v.d of ideal sheaves of curves}The real virtual dimension of $\overline{\mathcal{M}}_{c}^{DT_{4}}$
with $c=1-PD(\beta)-nPD(1)$ on a compact Calabi-Yau 4-fold $X$ satisfies
\begin{equation}\textrm{ } v.d_{\mathbb{R}}(\overline{\mathcal{M}}_{c}^{DT_{4}})=2n,  \quad \quad \textrm{ } \textrm{ }\textmd{if} \quad Hol(X)=SU(4), \nonumber \end{equation}
\begin{equation}v.d_{\mathbb{R}}(\overline{\mathcal{M}}_{c}^{DT_{4}})=2n-1, \quad \textmd{if} \quad  Hol(X)=Sp(2). \nonumber \end{equation}
\end{lemma}
\begin{proof}
By the Hirzebruch-Riemann-Roch theorem,
\begin{equation}\chi(I_{C},I_{C})
=2\int_{X}ch_{4}(I_{C})+ 2h^{0}(X,\mathcal{O}_{X})+h^{2}(X,\mathcal{O}_{X}).
\nonumber \end{equation}
For $Hol(X)=Sp(2)$, we have
\begin{equation}ext^{1}(I_{C},I_{C})-\frac{1}{2}ext^{2}(I_{C},I_{C})=n-\frac{1}{2}.  \nonumber \end{equation}
%\begin{equation}v.d_{\mathbb{R}}(\overline{\mathcal{M}}_{c}^{DT_{4}})=2n-1 .\nonumber \end{equation}
Similar calculations apply to the case when $Hol(X)=SU(4)$.
\end{proof}
\begin{remark}
%1. The generalized $DT_{4}$ moduli space of ideal sheaves of subschemes is not always defined as
%it depends on the gluing assumption \ref{assumption on gluing}.
%However, we still define its virtual dimension as stated above.
%At least in the case when $\mathcal{M}_{c}$ is smooth, $\overline{\mathcal{M}}_{c}^{DT_{4}}$ exists and %$\overline{\mathcal{M}}_{c}^{DT_{4}}=\mathcal{M}_{c}$. \\
If $Hol(X)=SU(4)$, $v.d_{\mathbb{R}}(\overline{\mathcal{M}}_{c}^{DT_{4}})=2n=2(1-g)$ for ideal sheaves of smooth connected genus $g$ curves. $DT_{4}$ invariants for such ideal sheaves with $g\geq 2$ vanish which coincides with the situation of Gromov-Witten invariants.
\end{remark}
We come to study $DT_{4}/GW$ correspondence when $\mathcal{M}_{c}$ is smooth.

\subsubsection{The case of $Hol(X)=SU(4)$ }
We start with one dimensional closed subschemes of $X$.
\begin{lemma}\label{DT GW 1} Let $X$ be a compact Calabi-Yau 4-fold with $Hol(X)=SU(4)$.
Let $C\hookrightarrow X$ be a closed subscheme with $dim_{\mathbb{C}}C\leq 1$ and $H^{1}(X,\mathcal{O}_{C})=0$,
where $\mathcal{O}_{C}$ is the structure sheaf of $C$.
Then we have canonical isomorphisms
\begin{equation}Ext^{i}(I_{C},I_{C})\cong Ext^{i}(\mathcal{O}_{C},\mathcal{O}_{C}), \quad i=1,2, \nonumber \end{equation}
where $I_{C}$ is the ideal sheaf of $C$ in $X$.
\end{lemma}
\begin{proof}
Taking cohomology of the following short exact sequence
\begin{equation}0\rightarrow I_{C}\rightarrow \mathcal{O}_{X}\rightarrow \mathcal{O}_{C}\rightarrow 0,  \nonumber \end{equation}
we get
\begin{equation}\rightarrow H^{i}(X,\mathcal{O}_{X})\rightarrow H^{i}(X,\mathcal{O}_{C})
\rightarrow  H^{i+1}(X,I_{C})\rightarrow.   \nonumber \end{equation}
%Because the first two nonzero terms are isomorphic, we can replace the first two nonzero terms by 0, the remaining sequence is still exact  H^{0}(X,\mathcal{O}_{C})\cong H^{1}(X,I_{C}),.
We have $H^{i}(X,\mathcal{O}_{X})=0$ for $i=1,2,3$ by $Hol(X)=SU(4)$.
$H^{i}(X,\mathcal{O}_{C})=0$ for $i=2,3,4$, because $dim_{\mathbb{C}}\mathcal{O}_{C}\leq1$. Thus
\begin{equation}\label{DT GW equation 0} H^{1}(X,\mathcal{O}_{C})\cong H^{2}(X,I_{C}), H^{3}(X,I_{C})=0,
H^{4}(X,I_{C})\cong H^{4}(X,\mathcal{O}_{X}). \end{equation}
Applying $Hom(\mathcal{O}_{C}, \cdot)$ to $0\rightarrow I_{C}\rightarrow \mathcal{O}_{X}\rightarrow \mathcal{O}_{C}\rightarrow 0$, we have
\begin{equation}\rightarrow Ext^{i}(\mathcal{O}_{C},\mathcal{O}_{X})\rightarrow
Ext^{i}(\mathcal{O}_{C},\mathcal{O}_{C})\rightarrow Ext^{i+1}(\mathcal{O}_{C},I_{C})\rightarrow.  \nonumber \end{equation}
By Serre duality, $Ext^{i}(\mathcal{O}_{C},\mathcal{O}_{X})\cong Ext^{4-i}(\mathcal{O}_{X},\mathcal{O}_{C})=0$ if $i=0,1,2$. Hence
\begin{equation}\label{DT GW equation 1} Ext^{0}(\mathcal{O}_{C},I_{C})=0, Ext^{k}(\mathcal{O}_{C},\mathcal{O}_{C})
\cong Ext^{k+1}(\mathcal{O}_{C},I_{C}), k=0,1. \end{equation}
Meanwhile $Ext^{3}(\mathcal{O}_{C},\mathcal{O}_{X})\cong H^{1}(X,\mathcal{O}_{C})=0$ by our assumption, thus
\begin{equation}\label{DT GW equation 2} Ext^{2}(\mathcal{O}_{C},\mathcal{O}_{C})\cong Ext^{3}(\mathcal{O}_{C},I_{C}). \end{equation}
Using
\begin{equation}H^{0}(X,\mathcal{O}_{C})\cong Ext^{4}(\mathcal{O}_{C},\mathcal{O}_{X})\twoheadrightarrow
Ext^{4}(\mathcal{O}_{C},\mathcal{O}_{C})\cong Ext^{0}(\mathcal{O}_{C},\mathcal{O}_{C}) \nonumber \end{equation}
and the above long exact sequence, we have
\begin{equation}\label{DT GW equation 3} Ext^{1}(\mathcal{O}_{C},\mathcal{O}_{C})\cong
Ext^{3}(\mathcal{O}_{C},\mathcal{O}_{C})\cong Ext^{4}(\mathcal{O}_{C},I_{C}). \end{equation}
Applying the functor $Hom(\cdot,I_{C})$ to $0\rightarrow I_{C}\rightarrow \mathcal{O}_{X}\rightarrow \mathcal{O}_{C}\rightarrow 0$, we get
\begin{equation}\rightarrow Ext^{i}(\mathcal{O}_{X},I_{C})\rightarrow Ext^{i}(I_{C},I_{C})\rightarrow
Ext^{i+1}(\mathcal{O}_{C},I_{C})\rightarrow .\nonumber \end{equation}
By (\ref{DT GW equation 0}), we know $H^{2}(X,I_{C})\cong H^{1}(X,\mathcal{O}_{C})=0$ and $H^{3}(X,I_{C})=0$. Hence
\begin{equation}\label{DT GW equation 4}Ext^{2}(I_{C},I_{C})\cong Ext^{3}(\mathcal{O}_{C},I_{C}).  \end{equation}
The remaining sequence of the long exact sequence is
\begin{equation}\label{DT GW equation 5} Ext^{3}(I_{C},I_{C})\cong Ext^{4}(\mathcal{O}_{C},I_{C}) , \end{equation}
because $H^{4}(X,\mathcal{O}_{X})\cong Ext^{4}(\mathcal{O}_{X},I_{C})\twoheadrightarrow Ext^{4}(I_{C},I_{C})$ and they have the same dimensions.

By (\ref{DT GW equation 2}),(\ref{DT GW equation 3}),(\ref{DT GW equation 4}),(\ref{DT GW equation 5}),
we are done.
\end{proof}
If we further assume $C$ to be a connected smooth imbedded curve inside $X$, we have
\begin{lemma}\label{DT GW 2}  If $C$ is a connected genus zero smooth imbedded curve inside $X$, then we have canonical isomorphisms
\begin{equation}Ext^{1}(I_{C},I_{C})\cong H^{0}(C,\mathcal{N}_{C/X}), \nonumber \end{equation}
\begin{equation}Ext^{2}(I_{C},I_{C})\cong H^{1}(C,\mathcal{N}_{C/X})\oplus H^{1}(C,\mathcal{N}_{C/X})^{*}. \nonumber \end{equation}
%where $\mathcal{N}_{C/X}$ is the normal bundle of $C$.
Furthermore, under this identification, $H^{1}(C,\mathcal{N}_{C/X})$ is a maximal isotropic subspace of
$Ext^{2}(I_{C},I_{C})$ with respect to the Serre duality pairing.
\end{lemma}
\begin{proof}
We have the local to global spectral sequence which degenerates at $E_{2}$ terms
\begin{eqnarray*}Ext^{*}_{X}(\mathcal{O}_{C},\mathcal{O}_{C})&\cong&
H^{*}(X,\mathcal{E}xt^{*}_{X}(\mathcal{O}_{C},\mathcal{O}_{C})) \\
&\cong& H^{*}(X,\iota_{*}(\wedge^{*}\mathcal{N}_{C/X})) \\
&\cong& H^{*}(C,\wedge^{*}\mathcal{N}_{C/X}),
\end{eqnarray*}
where $\iota: C\hookrightarrow X$ is the imbedding map. Then we have
\begin{equation}Ext^{1}(\mathcal{O}_{C},\mathcal{O}_{C})\cong H^{0}(C,\mathcal{N}_{C/X})\oplus H^{1}(C,\mathcal{O}_{C}), \nonumber \end{equation}
\begin{equation}Ext^{2}(\mathcal{O}_{C},\mathcal{O}_{C})\cong H^{0}(C,\wedge^{2}\mathcal{N}_{C/X})\oplus H^{1}(C,\mathcal{N}_{C/X}).
\nonumber \end{equation}
By $H^{1}(X,\mathcal{O}_{C})=0$ and the perfect pairing $\wedge^{2}\mathcal{N}_{C/X}\otimes\mathcal{N}_{C/X}\rightarrow \wedge^{3}\mathcal{N}_{C/X}\cong \Omega_{C}$, we get the hoped canonical isomorphisms.
%\begin{equation}Ext^{1}(I_{C},I_{C})\cong H^{0}(C,\mathcal{N}_{C/X}), \nonumber \end{equation}
%\begin{equation}Ext^{2}(I_{C},I_{C})\cong H^{1}(C,\mathcal{N}_{C/X})^{*}\oplus H^{1}(C,\mathcal{N}_{C/X}). \nonumber \end{equation}
As for the pairing, we have
\begin{equation}Ext^{2}(I_{C},I_{C})\otimes Ext^{2}(I_{C},I_{C})\rightarrow Ext^{4}(I_{C},I_{C})\rightarrow H^{4}(X,\mathcal{O}_{X}),
\nonumber \end{equation}
where the last trace map is a canonical isomorphism. The above pairing can be identified with
\begin{equation}(H^{0}(C,\wedge^{2}\mathcal{N}_{C/X})\oplus H^{1}(C,\mathcal{N}_{C/X}))\otimes
(H^{0}(C,\wedge^{2}\mathcal{N}_{C/X})\oplus H^{1}(C,\mathcal{N}_{C/X}))\rightarrow \nonumber\end{equation}
\begin{equation}\rightarrow H^{1}(C,\wedge^{3}\mathcal{N}_{C/X})\cong\mathbb{C}.\nonumber\end{equation}
Meanwhile
\begin{equation} H^{1}(C,\mathcal{N}_{C/X})\otimes H^{1}(C,\mathcal{N}_{C/X})\rightarrow 0, \nonumber\end{equation}
which shows $H^{1}(C,\mathcal{N}_{C/X})$ is a maximal isotropic subspace under the above identification.
\end{proof}
We have canonically identified the deformation and obstruction spaces of $DT_{4}$ theory and $GW$ theory in the above case.
If we have some further assumptions on moduli spaces, we have the following $DT_{4}/GW$ correspondence by Definition \ref{virtual cycle when Mc smooth} and Lemma \ref{ASD equivalent to max isotropic}.
\begin{theorem}\label{DT=GW}
Let $X$ be a compact Calabi-Yau 4-fold with $Hol(X)=SU(4)$. If $\mathcal{M}_{c}$ with given Chern character
$c=(1,0,0,-PD(\beta),-1)$ is smooth and consists of ideal sheaves of smooth connected genus zero imbedded curves only, then $(\mathcal{L}_{\mathbb{C}},Q_{Serre})$ on $\mathcal{M}_{c}$ has a natural complex orientation $o(\mathcal{O})$.
Assume the $GW$ moduli space $\overline{\mathcal{M}}_{0,0}(X,\beta)\cong \mathcal{M}_{c}$, then $\overline{\mathcal{M}}_{c}^{DT_{4}}$ exists, $\overline{\mathcal{M}}_{c}^{DT_{4}}\cong\overline{\mathcal{M}}_{0,0}(X,\beta)$ and  $[\overline{\mathcal{M}}_{c}^{DT_{4}}]^{vir}=[\overline{\mathcal{M}}_{0,0}(X,\beta)]^{vir}$.
\end{theorem}
By Theorem 4.3, Chapter V of \cite{kollar}, the Hilbert scheme of lines in a generic sextic $CY_{4}$,
$X\subseteq \mathbb{P}^{5}$ satisfies the above conditions, and
the moduli space $\overline{\mathcal{M}}_{c}^{DT_{4}}$ is a general type algebraic curve whose virtual cycle is its usual fundamental class.

Another example is given by those $X$ containing an embedded $\mathbb{P}^{3}$, i.e. $\exists$ $i: \mathbb{C}\mathbb{P}^{3}\hookrightarrow X$.
We denote $\beta=i_{*}[H^{2}]$, where $[H^{2}]\in H_{2}(\mathbb{C}\mathbb{P}^{3})$ is the generator. As the normal bundle $\mathcal{N}_{\mathbb{P}^{3}/X}$ is negative, ideal sheaves of curves in class $\beta$ are of the form $I_{C}$, where $C\hookrightarrow \mathbb{C}\mathbb{P}^{3}\hookrightarrow X$ is of degree $1$.
Meanwhile, $\mathcal{M}_{c}$ with $c=(1,0,0,-PD(\beta),-1)$, $\beta=i_{*}[H^{2}]\in H_{2}(X)$ is smooth. Hence conditions in Theorem \ref{DT=GW} are satisfied. We determine $DT_{4}$ the virtual cycle in this case.
\begin{proposition}
Let $c=(1,0,0,-PD(\beta),-1)\in H^{even}(X,\mathbb{Z})$.
If there exists an embedding $i: \mathbb{P}^{3}\hookrightarrow X$ and $\beta=i_{*}[H^{2}]$,
then $\overline{\mathcal{M}}_{c}^{DT_{4}}$ exists and $\overline{\mathcal{M}}_{c}^{DT_{4}}\cong \overline{\mathcal{M}}_{0,0}(K_{\mathbb{P}^{3}},[H^{2}])$. Furthermore, $[\overline{\mathcal{M}}_{c}^{DT_{4}}]^{vir}=[\overline{\mathcal{M}}_{0,0}(K_{\mathbb{P}^{3}},[H^{2}])]^{vir}$.
\end{proposition}
\begin{proof}
By the positivity of $TC\cong T\mathbb{P}^{1}$, we have $H^{1}(C,\mathcal{N}_{C/X})\cong H^{1}(C,\iota^{*}TX)$, where
$\iota: C\hookrightarrow X$.
Meanwhile,
\begin{equation}0\rightarrow T\mathbb{P}^{3}\rightarrow TX\mid_{\mathbb{P}^{3}}\rightarrow \mathcal{N}_{\mathbb{P}^{3}/X}\rightarrow 0
\nonumber \end{equation}
induces $H^{1}(C,\iota^{*}TX)\cong H^{1}(C,\iota^{*}K_{\mathbb{P}^{3}})$
which is the obstruction space of Gromov-Witten theory of $K_{\mathbb{P}^{3}}$ as
$H^{1}(C,\iota^{*}TK_{\mathbb{P}^{3}})\cong H^{1}(C,\iota^{*}K_{\mathbb{P}^{3}})$.
\end{proof}

\subsubsection{The case of $Hol(X)=Sp(2)$ }
When $Hol(X)=Sp(2)$, we similarly have
\begin{lemma}Let $X$ be a compact irreducible hyper-K\"ahler 4-fold.
Under the assumption that $C\hookrightarrow X$ is a closed subscheme with $dim_{\mathbb{C}}C\leq 1$ and $H^{1}(X,\mathcal{O}_{C})=0$, we have
an isomorphism
\begin{equation}Ext^{1}(I_{C},I_{C})\cong Ext^{1}(\mathcal{O}_{C},\mathcal{O}_{C}), \nonumber \end{equation}
and a short exact sequence
\begin{equation}0\rightarrow H^{2}(X,\mathcal{O}_{X})\rightarrow Ext^{2}(I_{C},I_{C})\rightarrow Ext^{2}(\mathcal{O}_{C},\mathcal{O}_{C})\rightarrow 0 .\nonumber \end{equation}
\end{lemma}
%\begin{proof}By a similar argument as the case of $Hol(X)=SU(4)$.
%\end{proof}
\begin{lemma}If $C$ is a connected genus zero smooth imbedded curve inside $X$, then we have
\begin{equation}Ext^{1}(I_{C},I_{C})\cong H^{0}(C,\mathcal{N}_{C/X}), \nonumber \end{equation}
\begin{equation}0\rightarrow H^{2}(X,\mathcal{O}_{X})\rightarrow Ext^{2}(I_{C},I_{C})\rightarrow H^{1}(C,\mathcal{N}_{C/X})\oplus H^{1}(C,\mathcal{N}_{C/X})^{*} \rightarrow 0. \nonumber \end{equation}
\end{lemma}
The $GW$ obstruction space $H^{1}(C,\mathcal{N}_{C/X})\cong H^{1}(C,\iota^{*}TX)$ if $C\cong \mathbb{P}^{1}$.
Meanwhile,
\begin{equation}0\rightarrow \mathcal{N}_{C/X}^{*}\rightarrow \iota^{*}\Omega_{X}\rightarrow \Omega_{C}\rightarrow 0
\nonumber \end{equation}
induces
\begin{equation}0\cong H^{0}(C,\Omega_{C})\rightarrow H^{1}(C,\mathcal{N}_{C/X}^{*})\rightarrow
H^{1}(C,\iota^{*}\Omega_{X})\cong H^{1}(C,\iota^{*}TX)\rightarrow H^{1}(C,\Omega_{C})\rightarrow 0,
\nonumber \end{equation}
where the isomorphism is given by the holomorphic symplectic two-form $\sigma\in H^{0}(X,\Omega^{2}_{X})$.
The above sequence establishes the surjective cosection of the obstruction sheaf of $GW$ theory for hyper-K\"ahler manifolds \cite{kiem}, \cite{kiem0}.

As vector spaces,
\begin{equation}\label{hyper cosection}Ext^{2}(I_{C},I_{C})\cong H^{1}(C,\mathcal{N}_{C/X}^{*})\oplus H^{1}(C,\Omega_{C})
\oplus H^{1}(C,\mathcal{N}_{C/X}^{*})^{*}\oplus H^{1}(C,\Omega_{C})^{*}\oplus H^{2}(X,\mathcal{O}_{X}). \end{equation}
%Here the dimension of the trivial factors in the $DT_{4}$ obstruction space is bigger than one.
Taking away trivial factors and restrict to the maximal isotropic subspace, we define the hyper-reduced $DT_{4}$ obstruction space.
\begin{definition}The hyper-reduced $DT_{4}$ obstruction space is
\begin{equation}Ext^{2}_{\emph{hyper-red}}(I_{C},I_{C})\triangleq H^{1}(C,\mathcal{N}_{C/X}^{*}) \nonumber \end{equation}
\end{definition}
\begin{remark}
The hyper-reduced $DT_{4}$ obstruction space coincides with the reduced $GW$ obstruction space \cite{kiem}.
\end{remark}
\begin{definition}
The hyper-reduced virtual fundamental class of $\overline{\mathcal{M}}_{c}^{DT_{4}}$, denoted by $[\overline{\mathcal{M}}_{c}^{DT_{4}}]^{vir}_{hyper-red}$ is the Poincar\'{e} dual of the Euler class of the
hyper-reduced $DT_{4}$ obstruction bundle.
\end{definition}
\begin{theorem}\label{DT=GW2}
Let $X$ be a compact irreducible hyper-K\"ahler 4-fold. If $\mathcal{M}_{c}$ with given Chern character $c=(1,0,0,-PD(\beta),-1)$
is smooth and consists of ideal sheaves of smooth connected genus zero imbedded curves only, then $(\mathcal{L}_{\mathbb{C}},Q_{Serre})$ on $\mathcal{M}_{c}$ has a natural complex orientation $o(\mathcal{O})$. Assume the $GW$ moduli space
$\overline{\mathcal{M}}_{0,0}(X,\beta)\cong \mathcal{M}_{c}$, then $\overline{\mathcal{M}}_{c}^{DT_{4}}$ exists, $\overline{\mathcal{M}}_{c}^{DT_{4}}\cong\overline{\mathcal{M}}_{0,0}(X,\beta)$
and $[\overline{\mathcal{M}}_{c}^{DT_{4}}]^{vir}=0$.

Furthermore, $[\overline{\mathcal{M}}_{c}^{DT_{4}}]^{vir}_{hyper-red}=[\overline{\mathcal{M}}_{0,0}(X,\beta)]^{vir}_{red}$, where
$[\overline{\mathcal{M}}_{0,0}(X,\beta)]^{vir}_{red}$ is the reduced virtual fundamental class of the $GW$ moduli space defined by Kiem and Li \cite{kiem}, \cite{kiem0}.
\end{theorem}
A simple application is the following result due to Mukai \cite{mukai}.
\begin{corollary}$\mathbb{P}^{3}$ can't be embedded into any compact irreducible hyper-K\"ahler 4-fold. \end{corollary}
\begin{proof}We take away the trivial factor $H^{2}(X,\mathcal{O}_{X})$ and consider the maximal isotropic subspace
\begin{eqnarray*}Ext^{2}_{red}(I_{C},I_{C})&\triangleq&H^{1}(C,\mathcal{N}_{C/X}) \\
&=& H^{1}(C,\iota^{*}TX) \\
&=& H^{1}(C,\iota^{*}K_{\mathbb{P}^{3}}).
\end{eqnarray*}
The Euler class of the $GW$ obstruction bundle with fiber $H^{1}(C,\iota^{*}K_{\mathbb{P}^{3}})$ is not trivial by localization calculation \cite{klempand}.
By the hyper-K\"ahler condition, $H^{1}(C,\iota^{*}TX)$ has a surjective map to $H^{1}(C,\Omega_{C})$ which leads to the vanishing of the virtual cycle.
%Hence $\mathbb{P}^{3}$ can not sit inside any compact hyper-K\"ahler 4-fold.
\end{proof}

\subsection{Moduli spaces of ideal sheaves of one point}
In this subsection, we consider the Hilbert scheme of one point on $X$, i.e. $\mathcal{M}_{c}=X$. Here we prefer
%(actually equivalent to ideal sheaves of one point if $Hol(X)=SU(4)$)
considering the moduli space of structure sheaf of one point.
\begin{proposition}\label{moduli of one point}
If $Hol(X)=SU(4)$ and $c=(1,0,0,0,-1)$, then $\overline{\mathcal{M}}_{c}^{DT_{4}}$ exists and $\overline{\mathcal{M}}_{c}^{DT_{4}}\cong X$, $[\overline{\mathcal{M}}_{c}^{DT_{4}}]^{vir}=\pm PD\big(c_{3}(X)\big)$.
\end{proposition}
\begin{proof}
As there is an isomorphism $SU(4)\cong Spin(6)$ and a homotopy equivalence $SO(6;\mathbb{R})\sim SO(6;\mathbb{C})$, we have a map
$\pi: BSU(4)\rightarrow BSO(6;\mathbb{C})$ (induced from the $2:1$ covering $Spin(6)\rightarrow SO(6)$)
such that the corresponding universal bundle $E_{SU(4)}$, $E_{SO(6;\mathbb{C})}$ satisfies
\begin{equation}\label{equation 6}\pi^{*}E_{SO(6;\mathbb{C})}\cong\wedge^{2}E_{SU(4)}.  \end{equation}
Since $\overline{\mathcal{M}}_{c}\cong X$ is smooth, the obstruction bundle $Ob$ can be identified with $\wedge^{2}TX$. The quadratic bundle
$(\wedge^{2}TX,Q_{Serre})$ corresponds to some map $f:X\rightarrow BSO(6;\mathbb{C})$, which admits a lift $\widetilde{f}$ fitting into
commutative diagram
\begin{equation}\xymatrix{
  BSU(4) \ar[rr]^{\pi}
                &  &    BSO(6;\mathbb{C})     \\
                & X   \ar[ur]_{f} \ar[ul]^{\widetilde{f}} }. \nonumber\end{equation}
The half Euler class of $(Ob,Q_{Serre})$ satisfies $e(Ob,Q_{Serre})=f^{*}e(E_{SO(6;\mathbb{R})})$.
Meanwhile, by the property of half Euler class (see \cite{eg}, \cite{switzer}), $e(E_{SO(6;\mathbb{R})})^{2}=-c_{6}(E_{SO(6;\mathbb{C})})$.
Combining with (\ref{equation 6}), $c(TX)=\widetilde{f}^{*}(c(E_{SU(4)}))$ and the commutative diagram, we can obtain $e(Ob,Q_{Serre})=\pm c_{3}(X)$.
\end{proof}
%\begin{proof}
%By the standard Koszul resolution, we have
%\begin{equation}\cdot\cdot\cdot\rightarrow \mathcal{O}\otimes {\wedge}^{2}T_{p}^{*}\rightarrow \mathcal{O}\otimes
%T_{p}^{*}\rightarrow \mathcal{O}\rightarrow \mathcal{O}_{p}\rightarrow 0.  \nonumber\end{equation}
%Then
%\begin{eqnarray*}Ext^{i}(\mathcal{O}_{p},\mathcal{O}_{p})&\cong&Ext^{i}(\mathcal{O}\otimes {\wedge}^{\bullet}T_{p}^{*},\mathcal{O}_{p} ) \\
%&\cong& H^{i}(\mathcal{O}\otimes {\wedge}^{\bullet}T_{p}\otimes \mathcal{O}_{p}) \\
%&\cong& {\wedge}^{i}T_{p}\otimes \mathcal{O}_{p}. \nonumber\end{eqnarray*}
%The last equality is because differentials in Koszul complex are zero at point $p\in X$. The above isomorphism is canonical and
%we can identify the obstruction bundle as ${\wedge}^{2}T\mathcal{M}_{c}\cong{\wedge}^{2}TX$.
%
%As $SU(4)=Spin(6)$, we take $V$ to be a bundle of fundamental representations of $Spin(6)$ on $X$ such that
%\begin{equation} V\otimes_{\mathbb{R}}\mathbb{C}\cong\wedge^{2}T^{*}X. \nonumber \end{equation}
%We take a complex bundle $U$ such that $V$ is its underlying real bundle, then the spinor bundle $S^{+}(V)=\wedge^{even}U\otimes K^{\frac{1}{2}}$,
%where $K=\wedge^{3}U^{*}$ and $c_{3}(S^{+}(V))=-c_{3}(U)$.

%If we identify $T^{*}X\cong S^{+}(V)$ (corresponds to choose an orientation), we get $e(V)=c_{3}(U)=c_{3}(X)$.
%Since $Ob_{+}\triangleq\wedge^{2}_{+}TX\cong V^{*}$, we have $e(Ob_{+})=e(V^{*})=-c_{3}(X)$.
%\end{proof}
\begin{remark}If $Hol(X)=Sp(2)$, $v.d_{\mathbb{R}}(\overline{\mathcal{M}}_{c}^{DT_{4}})=1$ by Lemma {\ref{v.d of ideal sheaves of curves}}. Fixing determinants of ideal sheaves, the (reduced) real virtual dimension of $\overline{\mathcal{M}}_{c}^{DT_{4}}$ is $2$ and
the (reduced) virtual cycle also vanishes as odd Chern classes of hyper-K\"{a}hler manifolds are zero.
\end{remark}

\section{Equivariant $DT_{4}$ invariants on toric $CY_{4}$ via localization}
In this section, we restrict to the moduli space of ideal sheaves of curves $I_{n}(X,\beta)$ (the whole section also works for moduli spaces of ideal sheaves of surfaces in $X$), where $X$ is a toric Calabi-Yau 4-fold. By definition, $X$ admits a $(\mathbb{C}^{*})^{4}$-action which can be naturally lifted to the moduli space. If we restrict to the three dimensional sub-torus $T\subseteq (\mathbb{C}^{*})^{4}$
which preserves the holomorphic top form of $X$, the action will also preserve the Serre duality pairing.

By the philosophy of virtual localization due to Graber and Pandharipande \cite{gp}, we will define the corresponding equivariant $DT_{4}$ invariants. Roughly speaking, we should have
\begin{equation}\int_{[\overline{\mathcal{M}}_{n,\beta}^{DT_{4}}(X)]^{vir}}\prod_{i=1}^{r}\gamma_{i}\thickapprox\sum_{[\mathcal{I}]
\in I_{n}(X,\beta)^{T}}\int_{[S(\mathcal{I})]^{vir}}\prod_{i=1}^{r}\gamma_{i}|_{\mathcal{I}}\cdot
\frac{\sqrt{e_{T}(Ext^{2}(\mathcal{I},\mathcal{I}))}}{e_{T}(Ext^{1}(\mathcal{I},\mathcal{I}))},
\nonumber \end{equation}
where $\overline{\mathcal{M}}_{n,\beta}^{DT_{4}}(X)$ denotes the generally undefined generalized $DT_{4}$ moduli space whose reduced structure
is the same as the reduced structure of $I_{n}(X,\beta)$ and $\gamma_{i}$ are certain insertion fields we only
define on the right hand side.

By a similar argument as Lemma 6, 8 in \cite{mnop}, one can show that for $\mathcal{I}\in I_{n}(X,\beta)^{T}$ which is a $T$-fixed point, $T$-representations
\begin{equation}Ext^{1}(\mathcal{I},\mathcal{I}), \quad Ext^{2}(\mathcal{I},\mathcal{I})  \nonumber \end{equation}
contain no trivial sub-representations.
Hence when we are reduced to local contributions, we can get ride of the non-reduced structure and consider
everything on $I_{n}(X,\beta)$ instead of on the generalized $DT_{4}$ moduli space. \\

For $\mathcal{I}\in I_{n}(X,\beta)^{T}$, we form the following complex vector bundle over $BT$ whose fiber is $V_{\mathcal{I}}\triangleq Ext^{2}(\mathcal{I},\mathcal{I})$,
\begin{equation}
\begin{array}{lll}
      & \quad ET\times_{T}V_{\mathcal{I}}
      \\  &  \quad \quad \quad \downarrow \\   &   ET\times_{T}\{\mathcal{I}\}=BT.
      %  leave space using \textrm{ }
\end{array}\nonumber\end{equation}
The Serre duality pairing naturally induces a non-degenerate pairing $Q_{Serre}$ on $ET\times_{T}V_{\mathcal{I}}$ as $T$ preserves the holomorphic top form. Thus, $(ET\times_{T}V_{\mathcal{I}},Q_{Serre})$ becomes a quadratic bundle (vector bundle with a non-degenerate quadratic form).

By the theory of characteristic classes of quadratic bundles \cite{eg}, there exists a half Euler class $e(ET\times_{T}V_{\mathcal{I}},Q_{Serre})$ (i.e. Euler class of a real form) if the structure group of $(ET\times_{T}V_{\mathcal{I}},Q_{Serre})$ can be reduced to $SO(N,\mathbb{C})$, $N=ext^{2}(\mathcal{I},\mathcal{I})$. In fact, we have
\begin{lemma}
For any $\mathcal{I}\in I_{n}(X,\beta)^{T}$, the structure group of $(ET\times_{T}V_{\mathcal{I}},Q_{Serre})$ can be reduced to $SO(N,\mathbb{C})$.
\begin{proof}
As there is a homotopy equivalence $SO(N,\mathbb{R})\sim SO(N,\mathbb{C})$, we are reduced to prove any real form of the quadratic bundle
$(ET\times_{T}V_{\mathcal{I}},Q_{Serre})\rightarrow BT$ is orientable, which is obvious as $BT$ is simply connected.
\end{proof}
\end{lemma}
\begin{definition}
For $\mathcal{I}\in I_{n}(X,\beta)^{T}$,
\begin{equation}e_{T}(Ext^{2}_{iso}(\mathcal{I},\mathcal{I}))\triangleq e(ET\times_{T}V_{\mathcal{I}},Q_{Serre})\in H^{*}_{T}(pt) \nonumber \end{equation}
is the half Euler class of the above quadratic bundle \cite{eg}, where $Q_{Serre}$ denotes the induced Serre duality pairing.
\end{definition}

\begin{remark}\label{remark on localization}${}$ \\
1. If $dim_{\mathbb{C}}V_{\mathcal{I}}$ is odd, the half Euler class is zero.
If $dim_{\mathbb{C}}V_{\mathcal{I}}$ is even, the half Euler class is unique up to a sign which is determined by an $SO(N,\mathbb{C})$ reduction of the structure group of the quadratic bundle
\begin{equation}(ET\times_{T}V_{\mathcal{I}},Q_{Serre})\rightarrow  ET\times_{T}\{\mathcal{I}\}, \nonumber \end{equation}
where $N=dim_{\mathbb{C}}V_{\mathcal{I}}$.
${}$ \\
2. If $dim_{\mathbb{C}}V_{\mathcal{I}}$ is even, by Proposition 2 of \cite{eg},
\begin{equation}\big(e_{T}(Ext^{2}_{iso}(\mathcal{I},\mathcal{I}))\big)^{2}=(-1)^{\frac{dim_{\mathbb{C}}V_{\mathcal{I}}}{2}}
e_{T}(Ext^{2}(\mathcal{I},\mathcal{I}))\in H^{*}(BT). \nonumber \end{equation}
Meanwhile,
$H^{*}(BT)$ is a polynomial ring, thus $e_{T}(Ext^{2}_{iso}(\mathcal{I},\mathcal{I}))$ is uniquely determined by $e_{T}(Ext^{2}(\mathcal{I},\mathcal{I}))$ up to a sign corresponding to the choice of an orientation. Hence we are essentially reduced to calculate
$e_{T}(Ext^{i}(\mathcal{I},\mathcal{I}))$ for $i=1,2$.
\end{remark}
%Now let us define equivariant $DT_{4}$ invariants for ideal sheaves of curves on toric $CY_{4}$.
As we do not have Seidel-Thomas twist for toric $CY_{4}$, we do not know how to give a compatible orientation for different components of $ET\times_{T}I_{n}(X,\beta)^{T}$. We just arbitrarily choose an orientation for each component at the moment.
\begin{definition}\label{toric ori data}
The toric orientation data is a choice of $SO(N,\mathbb{C})$ reduction of the structure group of the quadratic bundle
\begin{equation}(ET\times_{T}V_{\mathcal{I}},Q_{Serre})\rightarrow  ET\times_{T}\{\mathcal{I}\}, \nonumber \end{equation}
for each $\mathcal{I}\in I_{n}(X,\beta)^{T}$.
\end{definition}
\begin{definition}
Given $[\mathcal{I}]\in I_{n}(X,\beta)^{T}$, $P\in \mathbb{Z}[x_{1},x_{2},...]$ and $\gamma\in H^{*}_{T}(X,\mathbb{Z})$, we define
\begin{equation}
\pi_{*}([\mathcal{I}],P,\gamma)\triangleq\pi_{*}\Big(P\big(c_{i}^{T}(\mathfrak{I}|_{[\mathcal{I}]\times X})\big)\cup \gamma\Big)\in H^{*}_{T}(pt),  \nonumber \end{equation}
where $\pi_{*}: H^{*}_{T}(X)\rightarrow H^{*}_{T}(pt)$ is the equivariant push-forward and $\mathfrak{I}\rightarrow I_{n}(X,\beta)\times X$ is the universal ideal sheaf.
\end{definition}
\begin{definition}\label{def of equi DT4}
Given a toric Calabi-Yau 4-fold $X$, $\beta\in H_{2}(X,\mathbb{Z})$, $n,r\in \mathbb{Z}_{+}$, polynomials
$P_{i}\in \mathbb{Z}[x_{1},x_{2},...]$, insertion fields $\gamma_{i}\in H^{*}_{T}(X,\mathbb{Z})$ for $i=1,2,...,r$ and a toric orientation data,
the equivariant $DT_{4}$ invariant for ideal sheaves of curves associated with the above data is
\begin{equation}Z_{DT_{4}}\big(X,n\big|(P_{1},\gamma_{1}),...,(P_{r},\gamma_{r})\big)_{\beta}\triangleq\sum_{[\mathcal{I}]\in I_{n}(X,\beta)^{T}}
\bigg(\prod_{i=1}^{r}\big(\pi_{*}([\mathcal{I}],P_{i},\gamma_{i})\big)\bigg)\cdot
\frac{e_{T}(Ext^{2}_{iso}(\mathcal{I},\mathcal{I}))}{e_{T}(Ext^{1}(\mathcal{I},\mathcal{I}))},    \nonumber \end{equation}
where the sign in $e_{T}(Ext^{2}_{iso}(\mathcal{I},\mathcal{I}))$ is compatible with the chosen toric orientation data.
\end{definition}
\begin{remark} ${}$ \\
$1$. The above definition can also be applied to define equivariant $DT_{4}$ invariants for ideal sheaves of surfaces (two dimensional closed subschemes) in toric $CY_{4}$ as $T$-fixed locus of moduli spaces are also isolated. \\
$2$. The insertions in the above definition is just one plausible choice. If we want to match it with the corresponding $GW$ invariants, we may need some adjustment as toric 3-folds cases \cite{moop}. \\
$3$. By Remark \ref{remark on localization},
\begin{equation}\Bigg(\frac{e_{T}(Ext^{2}_{iso}(\mathcal{I},\mathcal{I}))}{e_{T}(Ext^{1}(\mathcal{I},\mathcal{I}))}\Bigg)^{2}
=\pm\frac{e_{T}(Ext^{2}(\mathcal{I},\mathcal{I}))}{e_{T}(Ext^{3}(\mathcal{I},\mathcal{I}))e_{T}(Ext^{1}(\mathcal{I},\mathcal{I}))}.
\nonumber \end{equation}
The RHS has a generalization to any dimensional smooth toric varieties. \\
%We do not know whether there is any counting meaning for the right hand side in general as
%we do not have the corresponding virtual theory for sheaves counting at this moment. \\
$4$. If the base manifold is an algebraic surface, moduli spaces of stable sheaves are smooth with vanishing obstruction spaces at
least for $K3$ and del-Pezzo surfaces. Euler characteristics of moduli spaces are virtual invariants.
For Calabi-Yau threefolds, if $\mathcal{M}_{c}$ is smooth, the
$DT_{3}$ invariant is the Euler characteristic of the moduli space up to a sign.

In the above two cases, Euler characteristics of moduli spaces in some sense represent sheaves virtual countings which should be partition functions of certain topological quantum field theories.

But for $CY_{4}$, Euler characteristics can not reflect the corresponding $DT_{4}$ invariants in general.
We should consider the Euler class of the self-dual obstruction bundle and the quadratic form takes its role.
Maybe this is one of the reasons we did not get a closed formula for the generating function of Euler characteristics
of Hilbert schemes of points on a complex 4-fold \cite{cheah}. We are wondering whether there are closed formulas for generating functions of
$DT_{4}$ invariants for ideal sheaves of points.
\end{remark}
We calculate the simplest example of equivariant $DT_{4}$ invariants for $X=\mathbb{C}^{4}$: take the sub-torus $T=\{(t_{1},t_{2},t_{3},t_{4})\in(\mathbb{C}^{*})^{4} \mid t_{1}t_{2}t_{3}t_{4}=1\}$ which acts on $\mathbb{C}^{4}$ by
\begin{equation}(t_{1},t_{2},t_{3},t_{4})\cdot(x_{1},x_{2},x_{3},x_{4})=(t_{1}x_{1},t_{2}x_{2},t_{3}x_{3},t_{4}x_{4}).
\nonumber \end{equation}
Let $\rho_{i}$ be the character of $(\mathbb{C}^{*})^{4}$ defined by $\rho_{i}(t_{1},t_{2},t_{3},t_{4})=t_{i}$ and $\lambda_{i}$ be its weight. Then
\begin{equation}H^{*}_{T}(pt,\mathbb{C})\cong \mathbb{C}[\lambda_{1},\lambda_{2},\lambda_{3},\lambda_{4}]/\sum_{i=1}^{4}\lambda_{i}.\nonumber \end{equation}
We consider the moduli space of ideal sheaves of one point on $\mathbb{C}^{4}$, $I_{1}(\mathbb{C}^{4},0)\cong \mathbb{C}^{4}$ on which $T$ acts with only one fixed point $I_{0}$.  It is easy to check
\begin{equation}e_{T}(Ext^{1}(I_{0},I_{0}))=\lambda_{1}\lambda_{2}\lambda_{3}\lambda_{4}, \nonumber \end{equation}
\begin{equation}e_{T}(Ext^{2}(I_{0},I_{0}))=\prod_{1\leq i<j\leq 4}(\lambda_{i}+\lambda_{j}). \nonumber \end{equation}
By Remark \ref{remark on localization}, we are left to find a square root of $-e_{T}(Ext^{2}(I_{0},I_{0}))$. It is obvious that
\begin{equation}-e_{T}(Ext^{2}(I_{0},I_{0}))=\big((\lambda_{1}+\lambda_{2})(\lambda_{1}+\lambda_{3})(\lambda_{2}+\lambda_{3})\big)^{2},
\nonumber \end{equation}
which gives
\begin{proposition}\label{equi for C4}
Let $X=\mathbb{C}^{4}$, for some choice of toric orientation data (Definition \ref{toric ori data}), we have
\begin{equation}Z_{DT_{4}}\big(X,1 \textrm{ } | (1,1)\big)_{0}=\frac{\sigma_{1}\sigma_{2}-\sigma_{3}}{\sigma_{1}\sigma_{3}}, \nonumber \end{equation}
where $\sigma_{i}$ is the $i$-th elementary symmetric polynomial of variables $\lambda_{1},\lambda_{2},\lambda_{3}$.
\end{proposition}
\begin{remark}
We remark that $c_{3}^{T}(\mathbb{C}^{4})=\sigma_{3}-\sigma_{1}\sigma_{2}$ and $c_{4}^{T}(\mathbb{C}^{4})=\sigma_{1}\sigma_{3}$. Thus
\begin{equation}Z_{DT_{4}}\big(X,1 \textrm{ } | (1,1)\big)_{0}=\pm\frac{c_{3}^{T}(\mathbb{C}^{4})}{c_{4}^{T}(\mathbb{C}^{4})},
\nonumber \end{equation}
which relates to the fact that the $DT_{4}$ virtual cycle for the moduli space of ideal sheaves of one point on $X$ is $\pm PD(c_{3}(X))$ (Proposition \ref{moduli of one point}).
\end{remark}

\section{Noncommutative $DT_{4}$ invariants}
In the non-commutative world, sheaves on manifolds are replaced by representations of algebras. When algebras are $CY_{4}$ coming from path algebras of quivers with relations, we have a corresponding theory counting their representations. Our constructions are motivated by the previous works of Szendr\"{o}i \cite{sz} and Mozgovoy-Reineke \cite{mr} on the non-commutative Donaldson-Thomas theory for $CY_{3}$ algebras.

\subsection{Basic facts on quivers with relations}

\subsubsection{Moduli spaces of representations of quivers with relations }
Let $Q=(Q_{0},Q_{1},h,t)$ be a quiver, where $Q_{0}$ is a finite set of vertices indexed by $0,1,...,n$, $Q_{1}$ is a finite set of oriented edges,
and $h,t: Q_{1}\rightarrow Q_{0}$ are maps giving the head and tail of each edge. The path algebra $\mathbb{C}Q$ is an associative algebra
with basis consisting of all paths of finite length. A quiver with relations $(Q,I)$ is a quiver $Q$ with a two-sided ideal $I$ in $\mathbb{C}Q$.
We denote $A=\mathbb{C}Q/I$.
\begin{definition}(\cite{js},\cite{king})
Let $Q=(Q_{0},Q_{1},h,t)$ be a quiver. A representation of $Q$ is a pair $(V=\bigoplus_{i=0}^{n}V_{i},\phi=\bigoplus_{a\in Q_{1}}\phi_{a})$
which consists of a finite dimensional $\mathbb{C}$-vector space $V_{i}$ for each vertex $i\in Q_{0}$ and a $\mathbb{C}$-linear map
$\phi_{a}: V_{t(a)}\rightarrow V_{h(a)}$ for each edge $a\in Q_{1}$. If $(Q,I)$ is a quiver with relations, a representation of $(Q,I)$ is a representation of $Q$, denoted by $(V,\phi)$ such that $\phi\mid_{I}=0$. The dimension vector of $(V,\phi)$ is denoted by $\underline{dim}(V)=\textbf{d}=(d_{0},d_{1},...,d_{n})$, where $d_{i}=dim_{\mathbb{C}}V_{i}$. We say $(Q,I)$
has bounded relations if $I$ is finitely generated.
%given by finitely many polynomials of edges.
\end{definition}
\begin{remark}
Representations of $(Q,I)$ are in one-one correspondence with finite-dimensional $A=\mathbb{C}Q/I$-modules.
\end{remark}
%The dimension vector of a quiver representation $(V,\phi)$ is denoted by $\textbf{d}=(d_{0},d_{1},...,d_{n})$, where $d_{i}=dim_{\mathbb{C}}V_{i}$.
We fix a dimension vector $\textbf{d}$ and consider the vector space
\begin{equation}R_{\textbf{d}}(Q)\triangleq \bigoplus_{a\in Q_{1}}Hom_{\mathbb{C}}(V_{t(a)},V_{h(a)}), \nonumber \end{equation}
on which the algebraic reductive group $G_{\textbf{d}}=\prod_{i\in Q_{0}}GL(V_{i})$ acts via
\begin{equation}(g_{i})\circ (\phi_{a})=(g_{h(a)}\phi_{a}g_{t(a)}^{-1}). \nonumber \end{equation}
To define a proper quotient, we introduce a stability condition due to \cite{king} (see also \cite{er}).
\begin{definition}(King \cite{king})\label{stab condition}
We fix a linear form $\Theta\in (\mathbb{Z}^{n+1})^{*}$ (called a stability condition). The slope of a $\mathbb{C}Q$-module $M$ with non-zero dimension vector $\underline{dim}M$ is
\begin{equation}\mu(M)\triangleq\mu(\underline{dim}M)\triangleq\frac{\Theta(\underline{dim}M)}{\sum_{i=0}^{n}dim_{\mathbb{C}} M_{i}}.
\nonumber \end{equation}
$M$ is $\mu$-stable (resp. $\mu$-semi-stable), if
\begin{equation}\mu(U)<\mu(M)\textrm{ }  (\textrm{resp}. \textrm{ } \mu(U)\leq\mu(M)) \nonumber \end{equation}
for any non-zero proper sub-representation $U$ of $M$.
\end{definition}
We denote $R^{s}_{\textbf{d}}(Q)\subseteq R_{\textbf{d}}(Q)$ to be the subspace of $\mu$-stable representations of $Q$. By \cite{er}, \cite{king}, there exists a smooth complex algebraic variety
\begin{equation}\mathcal{M}_{\textbf{d}}^{s}(Q)\triangleq R^{s}_{\textbf{d}}(Q)/G_{\textbf{d}}\nonumber \end{equation}
parametrizing isomorphism classes of $\mu$-stable representations of $Q$ \big(not $(Q,I)$\big) with dimension vector $\textbf{d}$.
We denote $\mathcal{M}_{\textbf{d}}^{ss}(Q)$ to be the coarse moduli space of $\mu$-semi-stable representations of $Q$. It
contains $\mathcal{M}_{\textbf{d}}^{s}(Q)$ as its open subvariety.

As in the case of moduli spaces of sheaves \cite{hl},
we want to find out a condition so that $\mathcal{M}_{\textbf{d}}^{s}(Q)=\mathcal{M}_{\textbf{d}}^{ss}(Q)$.
\begin{definition}(\cite{er})\label{def of coprime}
A dimension vector $\textbf{d}\in \mathbb{N}^{n+1}$ is coprime for $\Theta\in (\mathbb{Z}^{n+1})^{*} $ if $\mu(\textbf{e})\neq\mu(\textbf{d})$ for
all $0<\textbf{e}<\textbf{d}$, where $\textbf{e}<\textbf{d}$ means that $e_{i}\leq d_{i}$ for all $0\leq i\leq n$ and $e_{i}< d_{i}$ for some $i$.
\end{definition}
\begin{lemma}(\cite{king}, \cite{er}) \label{property of M(Q)} ${}$ \\
(1) If $\textbf{d}$ is coprime for $\Theta$, then $\mathcal{M}_{\textbf{d}}^{ss}(Q)=\mathcal{M}_{\textbf{d}}^{s}(Q)$ is a connected
smooth algebraic variety. \\
(2) If the quiver $Q$ does not have oriented cycles, then
%$\mathcal{M}_{\textbf{d}}^{ssimple}(Q)$ is a single point and
$\mathcal{M}_{\textbf{d}}^{ss}(Q)$ is a projective variety.
\end{lemma}
\begin{proof}
(Review of the proof of \cite{er}) When $\Theta=0$, $\mathcal{M}_{\textbf{d}}^{ss}(Q)$ will be denoted by $\mathcal{M}_{\textbf{d}}^{ssimple}(Q)$ which parametrizes semisimple representations of $Q$. It is affine and contains a special element
\begin{equation}\vec{0}\triangleq\bigoplus_{i=0}^{n}S_{i}\otimes V_{i}\in \mathcal{M}_{\textbf{d}}^{ssimple}(Q). \nonumber \end{equation}
We consider a quotient map $\pi: R_{\textbf{d}}(Q)\rightarrow \mathcal{M}_{\textbf{d}}^{ssimple}(Q)$,
sending a representation to the isomorphism class of its semi-simplification
(i.e. the direct sum of its Jordan-H\"{o}lder components).
Then the fiber $\pi^{-1}(\vec{0})$ consists of nilpotent representations,
i.e. representations with vanishing traces along all non-trivial oriented cycles.

By \cite{er}, $\pi$ is descended to $p: \mathcal{M}_{\textbf{d}}^{ss}(Q)\rightarrow \mathcal{M}_{\textbf{d}}^{ssimple}(Q)$ and $p$
%\begin{equation}p: \mathcal{M}_{\textbf{d}}^{ss}(Q)\rightarrow \mathcal{M}_{\textbf{d}}^{ssimple}(Q) \nonumber \end{equation}
is a projective morphism. If $Q$ does not have oriented cycles, $\mathcal{M}_{\textbf{d}}^{ssimple}(Q)$ is a single point and
$\mathcal{M}_{\textbf{d}}^{ss}(Q)$ is projective.
\end{proof}
Our aim is to study moduli spaces of representations of $(Q,I)$.
\begin{definition}
\begin{equation}\mathcal{M}_{\textbf{d}}^{s}(Q,I)\triangleq R^{s}_{\textbf{d}}(Q,I)/G_{\textbf{d}},  \nonumber \end{equation}
where $R^{s}_{\textbf{d}}(Q,I)\subseteq R^{s}_{\textbf{d}}(Q)$ consists of $\mu$-stable representations of $(Q,I)$.
\end{definition}
The obstruction theory of $\mathcal{M}_{\textbf{d}}^{s}(Q,I)$ is governed by $L_{\infty}$-algebra:
\begin{equation}\label{Ld} L^{*}_{\textbf{d}}\triangleq Ext^{*}_{A}(\bigoplus_{i=0}^{n}S_{i}\otimes V_{i},\bigoplus_{i=0}^{n}S_{i}\otimes V_{i}), \end{equation}
where $S_{i}$ is the simple $A=\mathbb{C}Q/I$ module corresponding to node $i\in Q_{0}$,
$V_{i}$ is a complex vector space of dimension $d_{i}$ (see e.g. \cite{hua}, \cite{segal}).
\begin{definition}
The obstruction bundle of $\mathcal{M}_{\textbf{d}}^{s}(Q,I)$ is
\begin{equation}\label{ob bundle}Ob\triangleq (L_{\textbf{d}}^{1})^{s}\times_{G_{\textbf{d}}}L_{\textbf{d}}^{2}. \end{equation}
\end{definition}
\begin{remark}$Ob$ is over $\mathcal{M}_{\textbf{d}}^{s}(Q)$ with fiber $L_{\textbf{d}}^{2}$ as $L_{\textbf{d}}^{1}$ is canonically identified with $R_{\textbf{d}}(Q)$.
%and $L_{\textbf{d}}^{0}$ is canonically identified with $\bigoplus_{i\in Q_{0}}End(V_{i})$, which is the Lie algebra of $G_{\textbf{d}}=\prod_{i\in Q_{0}}GL(V_{i})$.
\end{remark}
The Maurer-Cartan equation of $L_{\textbf{d}}$ converges provided $(Q,I)$
has bounded relations as $L_{\textbf{d}}^{2}$ computes relations in the quiver. In this case, we have
\begin{proposition}(\cite{hua}, \cite{segal}) \label{description of moduli}
We assume $\textbf{d}$ is coprime for $\Theta$ and $(Q,I)$
has bounded relations, then
$\mathcal{M}_{\textbf{d}}^{s}(Q,I)\cong MC(L_{\textbf{d}})//_{\Theta}G_{\textbf{d}}$.
\end{proposition}
Another way to phrase Proposition \ref{description of moduli} is that there exists a section of the obstruction bundle $Ob$
over $\mathcal{M}_{\textbf{d}}^{s}(Q)$ whose zero loci is $\mathcal{M}_{\textbf{d}}^{s}(Q,I)$, i.e.
\begin{equation}\label{picture zero loci of section}
\begin{array}{lll}
      & &  \textrm{  }\textrm{  } Ob
      \\  & \quad  & \quad \downarrow \\ \mathcal{M}_{\textbf{d}}^{s}(Q,I) & \hookrightarrow & \mathcal{M}_{\textbf{d}}^{s}(Q).
\end{array}\end{equation}
This will be used to construct virtual cycles of $\mathcal{M}_{\textbf{d}}^{s}(Q,I)$ when $(Q,I)$ satisfies the $CY_{4}$ condition.

\subsubsection{Framed quiver moduli spaces}
Moduli spaces of semi-stable representations are in general not schemes unless under the
coprime assumption on the stability (Lemma \ref{property of M(Q)}). This will cause difficulties for applying
the virtual theory developed by Li-Tian \cite{lt1} and Behrend-Fantechi \cite{bf}.
To overcome this, we introduce smooth models of quiver moduli spaces in the sense of Engel and Reineke \cite{er}.
%which adds an extra data called the "framing" in the quiver.
\begin{definition}(\cite{er}) \label{framed quiver}
Given a datum $(Q,\textbf{d},\Theta)$ as before and an extra non-zero dimension vector
$\textbf{e}=(e_{0},e_{1},...,e_{n})\in \mathbb{Z}^{n+1}$, we associate to them
a new datum $(\hat{Q},\hat{\textbf{d}},\hat{\Theta})$ as follows:

$\bullet$ the vertices of $\hat{Q}$ are those of $Q$, together with one additional vertex $\infty$,

$\bullet$ the edges of $\hat{Q}$ are those of $Q$, together with $e_{i}$ edges from $\infty$ to $i$, for each vertex $i$ of $Q$,

$\bullet$ we define $\hat{d_{i}}=d_{i}$ for all $i=0,1,...,n$ and $\hat{d}_{\infty}=1$,

$\bullet$ we define $\hat{\Theta}_{i}=\Theta_{i}$ for all $i=0,1,...,n$ and
$\hat{\Theta}_{\infty}=\frac{\Theta(\textbf{d})}{\sum_{i=0}^{n}d_{i}}+\epsilon$ for some sufficiently small positive $\epsilon\in \mathbb{Q}$.
\end{definition}
We consider the vector space,
\begin{equation}R_{\textbf{d},\textbf{e}}(Q)\triangleq \bigoplus_{a\in Q_{1}}Hom_{\mathbb{C}}(V_{t(a)},V_{h(a)})\times
\bigoplus_{i=0}^{n} Hom_{\mathbb{C}}(W_{i},V_{i}), \nonumber \end{equation}
where $dim_{\mathbb{C}}V_{i}=d_{i}$ and $dim_{\mathbb{C}}W_{i}=e_{i}$. $G_{\textbf{d}}=\prod_{i\in Q_{0}}GL(V_{i})$ acts on $R_{\textbf{d},\textbf{e}}(Q)$ and the GIT quotient exists
\begin{equation}\mathcal{M}_{\textbf{d}}^{ss}(\hat{Q})=R_{\textbf{d},\textbf{e}}(Q)//_{\Theta}G_{\textbf{d}}.\nonumber \end{equation}
\begin{remark}${}$ \\
1. By Lemma 3.2 of \cite{er}, $\mathcal{M}_{\textbf{d}}^{ss}(\hat{Q})=\mathcal{M}_{\textbf{d}}^{s}(\hat{Q})$.
%\begin{equation}\label{framed moduli}\mathcal{M}_{\textbf{d},\textbf{e}}^{\Theta}(Q)\triangleq\mathcal{M}_{\textbf{d}}^{s}(\hat{Q}),\end{equation}
We denote $\mathcal{M}_{\textbf{d},\textbf{e}}^{\Theta}(Q)\triangleq\mathcal{M}_{\textbf{d}}^{s}(\hat{Q})$ and call it the smooth model for $\mathcal{M}_{\textbf{d}}^{ss}(Q)$. \\
2. By Proposition 3.8 of \cite{er}, $\mathcal{M}_{\textbf{d},\textbf{e}}^{\Theta}(Q)$ is
a projective bundle over $\mathcal{M}_{\textbf{d}}^{s}(Q)$ if $\textbf{d}$ is coprime for $\Theta$.
\end{remark}
\begin{definition}
Given a quiver with relations $(Q,I)$, dimension vectors $\textbf{d},\textbf{e}\neq\vec{0}\in\mathbb{Z}_{\geq0}^{n+1}$ and a stability condition
$\Theta\in (\mathbb{Z}^{n+1})^{*}$, the moduli space of framed representations of $(Q,I)$ is
\begin{equation}\mathcal{M}_{\textbf{d},\textbf{e}}^{\Theta}(Q,I)\triangleq R_{\textbf{d},\textbf{e}}(Q,I)//_{\Theta}G_{\textbf{d}},\nonumber \end{equation}
where $ R_{\textbf{d},\textbf{e}}(Q,I)\triangleq Z(I)\times
\bigoplus_{i=0}^{n} Hom_{\mathbb{C}}(W_{i},V_{i}) $ and $Z(I)\subseteq\bigoplus_{a\in Q_{1}}Hom_{\mathbb{C}}(V_{t(a)},V_{h(a)})$
consists of representations of $(Q,I)$.
\end{definition}
Similar to (\ref{ob bundle}), we can define the framed obstruction bundle over $\mathcal{M}_{\textbf{d},\textbf{e}}^{\Theta}(Q)$ with fiber $L_{\textbf{d}}^{2}$.
\begin{definition}
The framed obstruction bundle is
\begin{equation}Ob_{fr}\triangleq R_{\textbf{d},\textbf{e}}^{s}(Q)\times_{G_{\textbf{d}}}L_{\textbf{d}}^{2}. \nonumber \end{equation}
\end{definition}
\begin{remark}
As in (\ref{picture zero loci of section}), $Ob_{fr}$ has a section whose zero loci is
$\mathcal{M}_{\textbf{d},\textbf{e}}^{\Theta}(Q,I)$, i.e.
\begin{equation}\label{framed picture zero loci of section}
\begin{array}{lll}
      & &  \textrm{  }\textrm{  } Ob_{fr}
      \\  & \quad  & \quad \downarrow \\ \mathcal{M}_{\textbf{d},\textbf{e}}^{\Theta}(Q,I) & \hookrightarrow &
\mathcal{M}_{\textbf{d},\textbf{e}}^{\Theta}(Q).
\end{array}\end{equation}
\end{remark}

\subsection{The definition of $NCDT_{4}$ invariants}
%We start with a quiver with relations $(Q,I)$ and fix a dimension vector $\textbf{d}$ and a stability condition $\Theta$.
%When $\textbf{d}$ is coprime for $\Theta$, $\mathcal{M}_{\textbf{d}}^{s}(Q)$ is a connected smooth complex variety by Lemma \ref{property of M(Q)}
We first introduce the concept of $CY_{4}$ algebras.
\begin{definition}(\cite{ginzburg}, \cite{sz}) \label{cy4 alg}
An associative $\mathbb{C}$-algebra $A$ is 4-Calabi-Yau ($CY_{4}$ for short), if for all $M,N\in A$-Mod (the category of
finitely generated $A$-modules) with at least one of them finite dimensional, there exist perfect bi-functorial pairings
\begin{equation}Ext^{k}_{A}(M,N)\times Ext^{4-k}_{A}(N,M)\rightarrow \mathbb{C}  \nonumber \end{equation}
between finite-dimensional $\mathbb{C}$-vector spaces.
\end{definition}
We take a quiver with bounded relations $(Q,I)$ such that $A=\mathbb{C}Q/I$ is $CY_{4}$, then $Ob$ admits a non-degenerate quadratic form $q$. By the characteristic class theory of quadratic bundles \cite{eg}, \cite{switzer},
the half Euler class $e(Ob,q)\in H^{rk(Ob)/2}(\mathcal{M}_{\textbf{d}}^{s}(Q),\mathbb{Z})$ exists if the structure group of $(Ob,q)$ can be reduced to $SO(\bullet,\mathbb{C})$.
\begin{lemma}\label{existence}
The structure group of $(Ob,q)$ can be reduced to $SO(\bullet,\mathbb{C})$ if $H^{1}(\mathcal{M}_{\textbf{d}}^{s}(Q),\mathbb{Z}_{2})=0$.
\end{lemma}
\begin{proof}
$H^{1}(\mathcal{M}_{\textbf{d}}^{s}(Q),\mathbb{Z}_{2})=0$ guarantees that any vector bundle over $\mathcal{M}_{\textbf{d}}^{s}(Q)$ is orientable.
\end{proof}
\begin{remark}
By the work of King and Walter \cite{king2}, the above condition on cohomology is satisfied when $Q$ does not have oriented cycles .
\end{remark}
Note that when $A=\mathbb{C}Q/I$ is $CY_{4}$, $(Q,I)$ has oriented cycles. Thus
$\mathcal{M}_{\textbf{d}}^{ss}(Q,I)$ is non-compact for a general dimension vector $\textbf{d}$. To take care the non-compactness, we take
the Borel-Moore homology $H_{r}^{BM}(\mathcal{M}_{\textbf{d}}^{s}(Q))\cong H^{top-r}(\mathcal{M}_{\textbf{d}}^{s}(Q))$ \cite{fulton}.
By the philosophy of defining virtual cycles in $DT_{4}$ theory before, we define
\begin{definition}\label{def of virtual cycle}
We take a quiver with bounded relations $(Q,I)$ such that $A=\mathbb{C}Q/I$ is a $CY_{4}$ algebra.
We fix a dimension vector $\textbf{d}$ and a stability condition $\Theta$ such that $\textbf{d}$ is coprime for $\Theta$.
Assume the structure group of $(Ob,q)$ can be reduced to $SO(N,\mathbb{C})$ ($N=rk(Ob))$, and choose a reduction. Then
the virtual fundamental class of $\mathcal{M}_{\textbf{d}}^{s}(Q,I)$ is
\begin{equation}\label{NC virtual cycle}[\mathcal{M}_{\textbf{d}}^{s}(Q,I)]^{vir}\triangleq PD(e(Ob,q))\in
H_{v.d}^{BM}(\mathcal{M}_{\textbf{d}}^{s}(Q)),
\end{equation}
where $v.d=2ext^{1}_{A}(W,W)-2ext^{0}_{A}(W,W)-ext^{2}_{A}(W,W)+2$ and $W=\bigoplus_{i=0}^{n}S_{i}\otimes V_{i}$. \\
We further call $[\mathcal{M}_{\textbf{d}}^{s}(Q,I)]^{vir}$ the non-commutative $DT_{4}$ ($NCDT_{4}$ for short) virtual cycle
associated to $(Q,I,\textbf{d},\Theta)$.
\end{definition}
\begin{remark}
%1. The above virtual cycle is a Borel-Moore homology class \cite{fulton}
%in the ambient manifold $\mathcal{M}_{\textbf{d}}^{s}(Q)$. The Borel-Moore homology $H_{r}^{BM}(\mathcal{M}_{\textbf{d}}^{s}(Q))\cong H^{top-r}(\mathcal{M}_{\textbf{d}}^{s}(Q))$ is used to take care the non-compactness of $\mathcal{M}_{\textbf{d}}^{s}(Q)$. \\
%2. As $\mathcal{M}_{\textbf{d}}^{s}(Q)$ is connected, the virtual cycle involves a choice of sign. \\
Because of the non-compactness of $\mathcal{M}_{\textbf{d}}^{s}(Q)$ for general $\textbf{d}$, $NCDT_{4}$ invariants may be zero in general.
However, if we choose $\textbf{d}$ to be of special form (killing oriented cycles in quivers),
$NCDT_{4}$ invariants are shown to be nontrivial (see Proposition \ref{NCDT4/3}).
%As discussed before, for sheaves counting problem on non-compact Calabi-Yau 4-folds, we could only count compactly supported sheaves or ideal sheaves using localization formula to take care the contribution from the infinity. Counting compactly supported sheaves
\end{remark}
%To remedy the non-compactness of $\mathcal{M}_{\textbf{d}}^{s}(Q)$ for general $\textbf{d}$ and
For the case of framed quiver moduli spaces, we similarly define
\begin{definition}\label{def of fr virtual cycle}
We take a quiver with bounded relations $(Q,I)$ such that $A=\mathbb{C}Q/I$ is a $CY_{4}$ algebra.
We fix dimension vectors $\textbf{d},\textbf{e}\neq\vec{0}$, and a stability condition $\Theta$.
Assume the structure group of $(Ob_{fr},q)$ can be reduced to $SO(N,\mathbb{C})$ ($N=rk(Ob_{fr}))$, and choose a reduction. Then
the virtual fundamental class of $\mathcal{M}_{\textbf{d},\textbf{e}}^{\Theta}(Q,I)$ is
\begin{equation}\label{framed virtual cycle}[\mathcal{M}_{\textbf{d},\textbf{e}}^{\Theta}(Q,I)]^{vir}\triangleq PD\big(e(Ob_{fr},q)\big)
\in H_{v.d+2\textbf{d}\cdot\textbf{e}-2}^{BM}(\mathcal{M}_{\textbf{d},\textbf{e}}^{\Theta}(Q)), \end{equation}
where $v.d=2ext^{1}_{A}(W,W)-2ext^{0}_{A}(W,W)-ext^{2}_{A}(W,W)+2$ for $W=\bigoplus_{i=0}^{n}S_{i}\otimes V_{i}$ and $\textbf{d}\cdot\textbf{e}$ is the
inner product between vectors $\textbf{d}$ and $\textbf{e}$. \\
We further call $[\mathcal{M}_{\textbf{d},\textbf{e}}^{\Theta}(Q,I)]^{vir}$ the non-commutative $DT_{4}$ ($NCDT_{4}$ for short) virtual cycle
associated to $(Q,I,\textbf{d},\textbf{e},\Theta)$.
\end{definition}
%\begin{remark} When $\textbf{e}=\vec{0}$ and $\textbf{d}$ is coprime for $\Theta$, Definition \ref{def of fr virtual cycle} recovers Definition \ref{def of virtual cycle}. \end{remark}

\subsection{Computational examples of $NCDT_{4}$ invariants}

We take a quiver $Q$
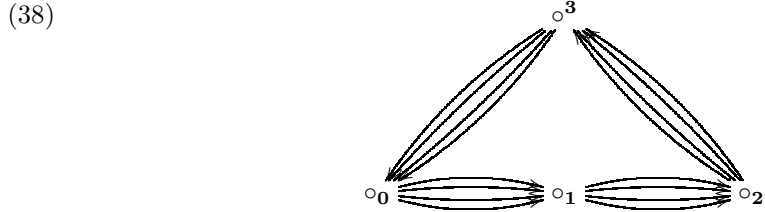
\begin{equation}\label{quiver of KP3}\xymatrix{
%c \ar[rr] ^-{ }
%\ar[dr] _(.2){ } |!{[d];[rr]}\hole
&&
\circ^{\textbf{3}}   \ar@/_0.2pc/[ddll] \ar@/_/[ddll] \ar@/^0.2pc/[ddll] \ar@/^/[ddll] &  \\ & &\\
\circ_{\textbf{0}}   \ar@/_/[rr]\ar@/_0.2pc/[rr] \ar@/^/[rr] \ar@/^0.2pc/[rr]
%\ar[urr] _>>>>{ }
&& \circ_{\textbf{1}}   \ar@/_0.2pc/[rr] \ar@/_/[rr] \ar@/^/[rr]\ar@/^0.2pc/[rr] && \circ_{\textbf{2}} \ar@/_0.2pc/[uull] \ar@/_/[uull]
\ar@/^0.2pc/[uull] \ar@/^/[uull]}
\end{equation}
with relations $I=<y_{j}x_{i}-y_{i}x_{j},z_{j}y_{i}-z_{i}y_{j},w_{j}z_{i}-w_{i}z_{j},x_{j}w_{i}-x_{i}w_{j}>$, where $x_{i},y_{i},z_{i},w_{i}$
denote paths between vertices $0\rightarrow1$, $1\rightarrow2$, $2\rightarrow3$, $3\rightarrow0$ in the above figure respectively. The path algebra is $CY_{4}$ and associates with $K_{\mathbb{P}^{3}}$ by the existence of the full exceptional collection on $\mathbb{P}^{3}$ \cite{bri},
\cite{ginzburg}.

Fixing the dimension vector $\textbf{d}=(d_{0},d_{1},d_{2},d_{3})$ and a stability condition $\Theta$, we get
\begin{equation}
\begin{array}{lll}
      & Ext^{2}_{A}(W,W)\rightarrow  & \textrm{ } Ob
      \\  &   & \quad \downarrow \\ & \quad \mathcal{M}_{\textbf{d}}^{s}(Q,I) \hookrightarrow & \mathcal{M}_{\textbf{d}}^{s}(Q)
      =(Ext^{1}_{A}(W,W))^{s}/G_{\textbf{d}},
\end{array} \nonumber\end{equation}
where $W=\oplus_{i=0}^{3}S_{i}\otimes V_{i}$, $S_{i}$ is the simple $A$-module corresponding to vertex $i$, $V_{i}$ is a complex vector space of
dimension $d_{i}$. Explicitly, we have
%$Ext^{0}_{A}(W,W)=\bigoplus_{i=0}^{3} End(V_{i})$,
\begin{equation}Ext^{1}_{A}(W,W)\cong(V_{0}^{*}\otimes V_{1})^{\oplus4}\oplus (V_{1}^{*}\otimes V_{2})^{\oplus4}\oplus
(V_{2}^{*}\otimes V_{3})^{\oplus4}\oplus (V_{3}^{*}\otimes V_{0})^{\oplus4},  \nonumber \end{equation}
\begin{equation}Ext^{2}_{A}(W,W)\cong[(V_{0}^{*}\otimes V_{2})\oplus (V_{2}^{*}\otimes V_{0}) ]^{\oplus6}\oplus [(V_{1}^{*}\otimes V_{3})\oplus
(V_{3}^{*}\otimes V_{1}) ]^{\oplus6}. \nonumber \end{equation}
%$v.d_{\mathbb{R}}(\mathcal{M}_{\textbf{d}}^{s}(Q,I) )=8(d_{0}d_{1}+d_{1}d_{2}+d_{2}d_{3}+d_{3}d_{0})-12(d_{0}d_{2}+d_{1}d_{3})-2\sum_{i=0}^{3}d_{i}^{2}$.
The obstruction space has a natural quadratic form $q$ with $(V_{0}^{*}\otimes V_{2})^{\oplus6}\oplus(V_{1}^{*}\otimes V_{3})^{\oplus6}$ as a maximal isotropic subspace. By Lemma \ref{ASD equivalent to max isotropic}, the $NCDT_{4}$ virtual cycle associated to $(Q,I,\textbf{d},\Theta)$ exists
(see Definition \ref{def of virtual cycle}).  \\
%and $(Ext^{2}_{A}(W,W),q)$ has a natural complex orientation. By Definition \ref{def of virtual cycle}, the $NCDT_{4}$ virtual cycle of $(Q,I,\textbf{d},\Theta)$ exists.

For convenience of latter discussions, we introduce the $NCDT_{3}$ virtual cycle of the quiver $(Q',I')$,
where $(Q',I')$ is the quiver obtained from $(Q,I)$ by eliminating edges connecting the vertex $3$ and the vertex $0$, i.e. \\
\begin{equation}\xymatrix{
%c \ar[rr] ^-{ }
%\ar[dr] _(.2){ } |!{[d];[rr]}\hole
\circ_{\textbf{0}}   \ar@/_/[rr]\ar@/_0.2pc/[rr] \ar@/^/[rr] \ar@/^0.2pc/[rr]
%\ar[urr] _>>>>{ }
&& \circ_{\textbf{1}}   \ar@/_0.2pc/[rr] \ar@/_/[rr] \ar@/^/[rr]\ar@/^0.2pc/[rr] && \circ_{\textbf{2}}  \ar@/_/[rr]\ar@/_0.2pc/[rr] \ar@/^/[rr] \ar@/^0.2pc/[rr] && \circ_{\textbf{3}} }.
\nonumber \end{equation}
\begin{definition}\label{NCDT3 virtual cycle}
Given the above quiver with relations $(Q',I')$, dimension vector $\textbf{d}$ and a stability condition $\Theta$ such that
$\textbf{d}$ and $\Theta$ are coprime, the $NCDT_{3}$ virtual cycle associated with $(Q',I',\textbf{d},\Theta)$ is
\begin{equation}
[\mathcal{M}_{\textbf{d}}^{s}(Q',I')]^{vir}\triangleq PD\big(e(Ob)\big)\in H_{v.d}(\mathcal{M}_{\textbf{d}}^{s}(Q'),\mathbb{Z}).
\nonumber \end{equation}
where $Ob=Ext^{2}_{A'}(W,W)\times_{G_{\textbf{d}}}(Ext^{1}_{A'}(W,W))^{s}$ and $A'=\mathbb{C}Q'/I'$ is the path algebra of $(Q',I')$.
$v.d\triangleq 2ext^{1}_{A'}(W,W)-2ext^{0}_{A'}(W,W)-ext^{2}_{A'}(W,W)+2$.
\end{definition}
\begin{remark}
In the above definition, $\mathcal{M}_{\textbf{d}}^{s}(Q')$ is compact since the quiver does not have oriented cycles (see Lemma \ref{property of M(Q)}).
\end{remark}
We compute $NCDT_{4}$ virtual cycles for certain dimension vectors and stability conditions.
\begin{proposition}\label{NCDT4/3}$(NCDT_{4}/NCDT_{3})$ ${}$ \\
%$(1)$ $\mathcal{M}_{(d_{0}\geq 2,0,0,0)}^{s}(Q)=\emptyset$, $\mathcal{M}_{(1,0,0,0)}^{s}(Q)=\{pt\}$ and
%$[\mathcal{M}_{(1,0,0,0)}^{s}(Q,I)]^{vir}=1$.
%If $1\leq d\leq n$, the framed quiver moduli space
%\begin{equation}\mathcal{M}^{\Theta}_{(d,0,0,0),(n,0,0,0)}(Q)\cong Gr(d,n)\nonumber \end{equation}
%and $[\mathcal{M}^{\Theta}_{(d,0,0,0),(n,0,0,0)}(Q,I)]^{vir}=[Gr(d,n)]$.
%Otherwise, $\mathcal{M}^{\Theta}_{(d,0,0,0),(n,0,0,0)}(Q)=\emptyset$.   \\
%$(2)$ We assume $d_{0},d_{1}\geq1$, $d_{2},d_{3}=0$ and take a stability condition $\Theta$ which is coprime to %$\textbf{d}=(d_{0},d_{1},d_{2},d_{3})$,
%then $\mathcal{M}_{\textbf{d}}^{s}(Q)$ is a smooth complex projective variety with obstruction bundle
%$Ob=0$ and
%\begin{equation}[\mathcal{M}_{(d_{0},d_{1},0,0)}^{s}(Q,I)]^{vir}=[\mathcal{M}_{(d_{0},d_{1},0,0)}^{s}(Q)]\in
%H_{top}(\mathcal{M}_{\textbf{d}}^{s}(Q),\mathbb{Z}).\nonumber \end{equation}
%In particular, for $\Theta=(-k,k,*,*)$ with $k\geq1$, $\mathcal{M}_{(1,1,0,0)}^{s}(Q)\cong \mathbb{CP}^{3}$.
%For other choices of $\Theta$, $\mathcal{M}_{(1,1,0,0)}^{s}(Q)=\emptyset$.  \\
%$(3)$ $(NCDT_{4}/NCDT_{3})$
We take the quiver with relations $(Q,I)$ in figure (\ref{quiver of KP3}), a stability condition $\Theta$
which is coprime to $\textbf{d}=(d_{0},d_{1},d_{2},d_{3})$, where $d_{0},d_{1},d_{2}\geq1$, $d_{3}=0$,
then $\mathcal{M}_{\textbf{d}}^{s}(Q)\cong \mathcal{M}_{\textbf{d}}^{s}(Q')$ as complex projective varieties. With respect to
the natural complex orientation, we further have
\begin{equation}[\mathcal{M}_{(d_{0},d_{1},d_{2},0)}^{s}(Q,I)]^{vir}=[\mathcal{M}_{(d_{0},d_{1},d_{2},0)}^{s}(Q',I')]^{vir},\nonumber \end{equation}
where $(Q',I')$ is the quiver obtained from $Q$ by eliminating edges connecting the vertex $3$ and the vertex $0$.
$[\mathcal{M}_{(d_{0},d_{1},d_{2},0)}^{s}(Q',I')]^{vir}$ is the $NCDT_{3}$ virtual cycle of $(Q',I')$ (Definition \ref{NCDT3 virtual cycle}).
\end{proposition}
\begin{proof}
%For $(1)$, it is similar to Example in page 4 of \cite{r1}. $R_{(d,0,0,0),(n,0,0,0)}(Q)$ parametrizes precisely all surjective linear maps from
%$\mathbb{C}^{n}$ to $\mathbb{C}^{d}$, thus the quotient is obviously the Grassmannian.

%For $(2)$, $Ext^{2}(W,W)=0$. Meanwhile,
%\begin{equation}\mathcal{M}_{(1,1,0,0)}^{ss}(Q)=\mathbb{C}^{4}//_{\Theta}(\mathbb{C}^{*}\times\mathbb{C}^{*}), \nonumber \end{equation}
%where $(g_{1},g_{2})\cdot(x_{1},x_{2},x_{3},x_{4})\triangleq(g_{2}g_{1}^{-1}x_{1},g_{2}g_{1}^{-1}x_{2},g_{2}g_{1}^{-1}x_{3},g_{2}g_{1}^{-1}x_{4})$.

%For $\Theta=(\theta_{1},\theta_{2},*,*)$, we have
%\begin{equation}\mathcal{M}_{(1,1,0,0)}^{ss}(Q)\cong \textrm{Proj }\big(\bigoplus_{n=0}^{\infty}
%\{f\in\mathbb{C}[x_{1},x_{2},x_{3},x_{4}]\textrm{ }| \textrm{ }f(g_{2}g_{1}^{-1}\vec{x})=g_{2}^{n\theta_{2}}g_{1}^{n\theta_{1}}f(\vec{x}) \}\big).
%\nonumber \end{equation}
%If $\Theta=(-k,k,*,*)$ with $k\geq1$, then
%\begin{equation}\mathcal{M}_{(1,1,0,0)}^{s}(Q)\cong \textrm{Proj }\big(\bigoplus_{n=0}^{\infty}P_{kn}(x_{1},x_{2},x_{3},x_{4})\big),
%\nonumber \end{equation}
%where $P_{kn}$ denotes the set of degree $kn$ polynomials of variables $x_{1},x_{2},x_{3},x_{4}$.
%By the $k$-uple embedding of projective spaces \cite{hart}, we have $\mathcal{M}_{\textbf{d}}^{s}(Q)\cong\mathbb{CP}^{3}$.

We obviously have $\mathcal{M}_{\textbf{d}}^{s}(Q)\cong \mathcal{M}_{\textbf{d}}^{s}(Q')$,
$Ext^{2}_{A}(W,W)=[(V_{0}^{*}\otimes V_{2})\oplus (V_{2}^{*}\otimes V_{0}) ]^{\oplus6}$.
Meanwhile the obstruction space $Ext^{2}_{A'}(W,W)=(V_{0}^{*}\otimes V_{2})^{\oplus6}$ of $\mathcal{M}_{\textbf{d}}^{s}(Q')$
is a maximal isotropic subspace of $Ext^{2}_{A}(W,W)$. By Lemma \ref{ASD equivalent to max isotropic}, we are done.
\end{proof}
\begin{remark}
The quiver $(Q',I')$ is associated with $\mathbb{P}^{3}$ by the existence of full exceptional collections \cite{bri}. Proposition \ref{NCDT4/3} is a non-commutative version of the $DT_{4}/DT_{3}$ correspondence proved in Theorem \ref{compact supp DT4}.
\end{remark}

\section{Appendix}

\subsection{Local Kuranishi models of $\mathcal{M}_{c}^{bdl}$ }
We review several local Kuranishi models of $\mathcal{M}_{c}^{bdl}$. \\
%They are isomorphic and related by change of variables. \\
${}$ \\
\textbf{A Kuranishi model of $\mathcal{M}_{c}^{bdl}$ with gauge fixing condition $\overline{\partial}_{A}^{*}a''=0$}.
By \cite{fm}, we define
\begin{equation}\kappa: H^{0,1}(X,EndE)\rightarrow H^{0,2}(X,EndE), \nonumber \end{equation}
\begin{equation}\label{kappa}\kappa(\alpha)=\mathbb{H}^{0,2}\big(g^{-1}(\alpha)\wedge g^{-1}(\alpha)\big), \end{equation}
where
\begin{equation}g:\Omega^{0,1}(X,EndE)_{k}\rightarrow\Omega^{0,1}(X,EndE)_{k}, \nonumber \end{equation}
\begin{equation}g(a'')\triangleq a''+\overline{\partial}_{A}^{-1}P_{\overline{\partial}_{A}}(a''\wedge a''). \nonumber \end{equation}
Note that $\overline{\partial}_{A}: Im(\overline{\partial}_{A}^{*})\rightarrow \Omega^{0,2}(X,EndE)_{k-1}$
is an isomorphism onto its image and $\overline{\partial}_{A}^{-1}$ is defined as its inverse.
%$a''\in Ker\overline{\partial}_{A}^{*}\subseteq\Omega^{0,1}(X,EndE)_{k}$  $\alpha=g(a'')\in H^{0,1}(X,EndE)$.
By the standard Kuranishi theory, we have a local isomorphism
\begin{equation} \kappa^{-1}(0)\cong\left\{ \begin{array}{lll}
    a'' \textrm{ }\big{|} & \|a''\|_{k} < \epsilon'', F^{0,2}(\overline{\partial}_{A}+a'')=0 , \textrm{ } \overline{\partial}_{A}^{*}a''=0      %  leave space using \textrm{ }
\nonumber\end{array}\right\}.\end{equation}
${}$ \\
\textbf{A Kuranishi model of $\mathcal{M}_{c}^{bdl}$ with gauge fixing condition $F\wedge \omega^{3}=0$}.
If we use gauge fixing $F\wedge \omega^{3}=0$, $d_{A}^{*}a=0$, we have another
Kuranishi model of $\mathcal{M}_{c}^{bdl}$ at $\overline{\partial}_{A}$: We define
\begin{equation}\widetilde{\kappa}: H^{0,1}(X,EndE)\rightarrow H^{0,2}(X,EndE),  \nonumber \end{equation}
\begin{equation}\label{kappa tilta}\widetilde{\kappa}(\alpha)=\mathbb{H}^{0,2}\big(\widetilde{g}^{-1}(\alpha)\wedge \widetilde{g}^{-1}(\alpha)\big),
\end{equation}
where
\begin{equation}\widetilde{g}: \Omega^{0,1}(EndE)_{k}\rightarrow H^{0,1}(EndE)\oplus {\overline{\partial}_{A}^{*}\Omega^{0,1}(EndE)}_{k}\oplus
{\overline{\partial}_{A}^{*}\Omega^{0,2}(EndE)}_{k-1},  \nonumber \end{equation}
\begin{equation}\label{g tilta}\widetilde{g}(a'')\triangleq\bigg(\mathbb{H}(a''),\overline{\partial}_{A}^{*}a''-\frac{i}{2}\wedge(a'\wedge a''+a''\wedge a'),
\overline{\partial}_{A}^{*}\big(\overline{\partial}_{A}a''+P_{\overline{\partial}_{A}}(a''\wedge a'')
\big)\bigg). \nonumber \end{equation}
By a suitable complex gauge transformation, we have $\widetilde{\kappa}^{-1}(0)\cong \kappa^{-1}(0)$. \\
%We know that $\widetilde{\kappa}^{-1}(0)$ gives a local model of $\mathcal{M}_{c}^{bdl}$ at $\overline{\partial}_{A}$
${}$ \\
\textbf{Another Kuranishi model of $\mathcal{M}_{c}^{bdl}$ with gauge fixing condition $F\wedge \omega^{3}=0$}.
Induced from the $DT_{4}$ equations, we have the following local Kuranishi model of $\mathcal{M}_{c}^{bdl}$ at $\overline{\partial}_{A}$: We define
\begin{equation}\tilde{\tilde{\kappa}}: H^{0,1}(X,EndE)\rightarrow H^{0,2}(X,EndE),  \nonumber \end{equation}
\begin{equation}\label{kappa double tilta}\tilde{\tilde{\kappa}}(\alpha)=\mathbb{H}^{0,2}\big(q^{-1}(\alpha)\wedge q^{-1}(\alpha)\big),
\end{equation}
where
\begin{equation}q: \Omega^{0,1}(EndE)_{k}\rightarrow H^{0,1}(EndE)\oplus {\overline{\partial}_{A}^{*}\Omega^{0,1}(EndE)}_{k}\oplus
{\overline{\partial}_{A}^{*}\Omega^{0,2}(EndE)}_{k-1},  \nonumber \end{equation}
\begin{equation}\label{g double tilta}q(a'')\triangleq\bigg(\mathbb{H}(a''),\overline{\partial}_{A}^{*}a''-\frac{i}{2}\wedge(a'\wedge a''+a''\wedge a'),
\overline{\partial}_{A}^{*}\big(\overline{\partial}_{A}a''+P_{\overline{\partial}_{A}}(a''\wedge a'')+*_{4}P_{\overline{\partial}_{A}^{*}}(a''\wedge a'')\big)\bigg). \nonumber \end{equation}
By Proposition \ref{QA intesect P=0 equals NA} and Definition \ref{real kuranishi map}, $\tilde{\tilde{\kappa}}$ is a Kuranishi map of $\mathcal{M}_{c}^{bdl}$ at $\overline{\partial}_{A}$.

\subsection{Seidel-Thomas twists}
In this section, we recall the Seidel-Thomas twist \cite{st} and how it could be used to identify a moduli space of simple sheaves to
a moduli space of simple holomorphic bundles, which is the work of Joyce and Song \cite{js}.
\begin{definition}
Let $(X,\mathcal{O}_{X}(1))$ be a projective Calabi-Yau $m$-fold with $Hol(X)=SU(m)$. For each $n\in \mathbb{Z}$,
the Seidel-Thomas twist $T_{\mathcal{O}_{X}(-n)}$ by $\mathcal{O}_{X}(-n)$ is the Fourier-Mukai transform from $D(X)$ to $D(X)$ with kernel
\begin{equation}K=cone(\mathcal{O}_{X}(n)\boxtimes\mathcal{O}_{X}(-n)\rightarrow \mathcal{O}_{\Delta}).
\nonumber \end{equation}
\end{definition}
In general, $T_{n}\triangleq T_{\mathcal{O}_{X}(-n)}[-1]$ maps sheaves to complexes of sheaves. But for $n\gg 0$, we have
\begin{theorem}\label{seidel thomas twist lemma}(Joyce-Song, Lemma 8.2 of \cite{js})
Let $U$ be a finite type $\mathbb{C}$-scheme and $\mathcal{F}_{U}$ is a coherent sheaf on $U\times X$ flat over $U$ i.e.
it is a $U$-family of coherent sheaves on $X$. Then for $n\gg 0$,
$T_{n}(\mathcal{F}_{U})$ is also a $U$-family of coherent sheaves on $X$.
\end{theorem}
Sufficiently many compositions of Seidel-Thomas twists map sheaves to vector bundles.
\begin{definition}
For a nonzero coherent sheaf $\mathcal{F}$, the homological dimension $hd(\mathcal{F})$ is the smallest $n\geq0$ for
which there exists an exact sequence in the abelian category $coh(X)$ of coherent sheaves
\begin{equation}0\rightarrow E_{n}\rightarrow E_{n-1}\cdot\cdot\cdot\rightarrow E_{0}\rightarrow \mathcal{F}\rightarrow 0
\nonumber \end{equation}
with $\{E_{i}\}_{i=0,...,n}$ are vector bundles.
\end{definition}
\begin{theorem}(Joyce-Song, Lemma 8.4 of \cite{js})
Let $\mathcal{F}_{U}$, $n\gg0$ be the same as in Theorem \ref{seidel thomas twist lemma}, then for any $u\in U$,
we have $hd(T_{n}(\mathcal{F}_{u}))=max(hd(\mathcal{F}_{u})-1,0)$.
\end{theorem}
\begin{corollary}(Joyce-Song, Corollary 8.5 of \cite{js})
Let $U$ be a finite type $\mathbb{C}$-scheme and $\mathcal{F}_{U}$ is a $U$-family of coherent sheaves on $X$.
Then there exists $n_{1},...n_{m}\gg0$ such that for $T_{n_{m}}\circ\cdot\cdot\cdot\circ T_{n_{1}}(\mathcal{F}_{U})$
is a $U$-family of vector bundles on $X$.
\end{corollary}
Meanwhile, Seidel-Thomas twists are auto-equivalences of derived category $D(X)$, they preserve determinant line bundles (if exists)
of corresponding moduli spaces.
\begin{corollary}(Joyce-Song \cite{js})\label{st preserves L}
Given a coarse moduli space $\mathcal{M}_{X}$ of simple sheaves with fixed Chern classes, we choose sufficiently large integers
$n_{1},...n_{m}\gg0$ such that $\Psi\triangleq T_{n_{m}}\circ\cdot\cdot\cdot\circ T_{n_{1}}$ identifies $\mathcal{M}_{X}$
with a coarse moduli space $\mathfrak{M}^{bdl}_{X}$ of simple holomorphic bundles. Then
\begin{equation}\Psi^{*}\mathcal{L}_{\mathfrak{M}^{bdl}_{X}}\cong\mathcal{L}_{\mathcal{M}_{X}}, \nonumber \end{equation}
where $\mathcal{L}_{\bullet}$ is the determinant line bundle of the corresponding moduli space.

Moreover, if $X$ is a $CY_{2n}$, $\mathcal{L}_{\mathfrak{M}^{bdl}_{X}}$ and $\mathcal{L}_{\mathcal{M}_{X}}$ are endowed with
non-degenerate quadratic forms from Serre duality pairing. The isomorphism $\Psi^{*}\mathcal{L}_{\mathfrak{M}^{bdl}_{X}}
\cong\mathcal{L}_{\mathcal{M}_{X}}$ also preserves
the quadratic forms.
\end{corollary}
By successive Seidel-Thomas twists, we obtain an isomorphism between $\mathcal{M}_{c}$ and some component(s) of a moduli space of simple holomorphic bundles. We also have
%In general, the Gieseker stability is lost after twists.
\begin{equation}Ext^{i}(\mathcal{F},\mathcal{F})\cong Ext^{i}(E_{b},E_{b}), \textrm{ } i=1,2, \nonumber\end{equation}
\begin{equation}\label{seidel thomas twist}\kappa^{-1}(0)\cong \kappa'^{-1}(0), \end{equation}
where $\mathcal{F}\in\mathcal{M}_{c}$, $\kappa$ is a Kuranishi map at $\mathcal{F}$ and
$E_{b}$ is the corresponding holomorphic bundle under Seidel-Thomas twists with $\kappa'$ to be a Kuranishi map at $E_{b}$.

One possible way to prove the vanishing result $\kappa_{+}=0\Rightarrow\kappa=0$ for $\mathcal{M}_{c}$ in Assumption \ref{assumption on gluing} is to use Seidel-Thomas twists transform stable sheaves to holomorphic bundles and then prove it using gauge theory. We will make it work in the next subsection.

\subsection{Comparisons of Borisov-Joyce's virtual fundamental classes with $DT_{4}$ virtual cycles }
We first recall local Kuranishi models used by Borisov and Joyce in their constructions of virtual fundamental classes for moduli spaces of stable sheaves on Calabi-Yau 4-folds.
\begin{theorem}(Brav, Bussi and Joyce \cite{bbj} Corollary 5.20)  ${}$ \\
Suppose $X$ is a Calabi-Yau 4-fold over a field $\mathbb{K}$, and $\mathcal{M}$ is a classical moduli $\mathbb{K}$-scheme of simple coherent sheaves, or simple complexes of coherent sheaves, on $X$. Then for each $[F]\in\mathcal{M}$, there exist a smooth $\mathbb{K}$-scheme
$U$ with $dim U=dim Ext^{1}(F,F)$, a vector bundle $E\rightarrow U$ with $rank E=dim Ext^{2}(F,F)$, a non-degenerate quadratic form $Q$ on $E$, a section $s\in H^{0}(E)$ with $Q(s,s)=0$, and an isomorphism from $s^{-1}(0)\subseteq U$ to a Zariski open neighbourhood of $[F]$ in $\mathcal{M}$.
\end{theorem}
Their original proof is based on the theory of cyclic homology and derived algebraic geometry.

As mentioned before, Borisov and Joyce \cite{bj} used the above local 'Darboux charts', the machinery of homotopical algebra and $C^{\infty}$-algebraic geometry to obtain a compact derived $C^{\infty}$-scheme with the same underlying topological structure as the Gieseker moduli space of stable sheaves. In our language, this $C^{\infty}$-scheme is called a generalized $DT_{4}$ moduli space ($C^{\infty}$-scheme version, see Definition \ref{generalized DT4}). Furthermore, they defined its virtual fundamental class.

%As a matter of fact, the above 'Darboux theorem' is very crucial for their construction of generalized $DT_{4}$ moduli spaces in general (We got generalized $DT_{4}$ moduli spaces as real analytic spaces for some special cases as shown above).
In the next theorem, we give a gauge theoretical proof of an analytic version of the above local 'Darboux theorem' in the case when $\mathbb{K}=\mathbb{C}$ and $\mathcal{M}=\mathcal{M}_{c}$ is a Gieseker moduli space of stable sheaves. We then introduce a weaker condition on their local 'Darboux charts' to include local models induced from $DT_{4}$ equations (the map $\tilde{\tilde{\kappa}}$ in Theorem \ref{Kuranishi str of cpx ASD thm}). It turns out that the weaker condition is already sufficient for the gluing requirement in Borisov and Joyce's work \cite{bj}.
\begin{theorem}\label{darboux thm using gauge theory}
Let $\mathcal{M}_{c}$ be the Gieseker moduli space of stable sheaves with fixed Chern character $c$ on a compact Calabi-Yau 4-fold $X$ with $Hol(X)=SU(4)$. \\
Then for any closed point $\mathcal{F}\in\mathcal{M}_{c}$, there exists an analytic neighborhood $U_{\mathcal{F}}\subseteq\mathcal{M}_{c}$, a holomorphic map near the origin
\begin{equation}\kappa: Ext^{1}(\mathcal{F},\mathcal{F})\rightarrow Ext^{2}(\mathcal{F},\mathcal{F})   \nonumber\end{equation}
such that $Q(\kappa,\kappa)=0$ and $\kappa^{-1}(0)\cong U_{\mathcal{F}}$ as complex analytic spaces possibly with non-reduced structures,
where $Q$ is the Serre duality pairing on $Ext^{2}(\mathcal{F},\mathcal{F})$.
\end{theorem}
\begin{proof}
We use Seidel-Thomas twists \cite{js},\cite{st} transfer the problem to a problem on moduli spaces of holomorphic bundles. We take a connection on $E$ with curvature $F$. By Chern-Weil theory,
%\begin{equation} Tr(F^{2})=-8\pi^{2}ch_{2}(E). \nonumber \end{equation}
%Then
\begin{equation}\label{chern weil formula}
-8\pi^{2}\int ch_{2}(E)\wedge\Omega=\int Tr(F^{0,2}\wedge F^{0,2})\wedge\Omega.  \end{equation}
We take the $(0,1)$ part of the connection to be $\overline{\partial}_{A}+a''$, where $\overline{\partial}_{A}^{2}=0$. Then
\begin{equation}F^{0,2}(\overline{\partial}_{A}+a'')=\overline{\partial}_{A}a''+P_{\overline{\partial}_{A}}(a''\wedge a'')
+P_{\overline{\partial}_{A}^{*}}(a''\wedge a'')+\mathbb{H}^{0,2}(a''\wedge a''). \nonumber\end{equation}
To describe Kuranishi theory, we take $a''$ satisfies $\overline{\partial}_{A}a''+P_{\overline{\partial}_{A}}(a''\wedge a'')=0$. After gauge fixing $\overline{\partial}_{A}^{*}a''=0$, we have
\begin{equation}F^{0,2}(\overline{\partial}_{A}+a'')=P_{\overline{\partial}_{A}^{*}}(g^{-1}(\alpha)\wedge g^{-1}(\alpha))+\mathbb{H}^{0,2}(g^{-1}(\alpha)\wedge g^{-1}(\alpha)), \nonumber\end{equation}
where $\alpha\triangleq g(a'')$, $g(a'')\triangleq a''+\overline{\partial}_{A}^{-1}P_{\overline{\partial}_{A}}(a''\wedge a'')$.
Note that $\alpha\in H^{0,1}(X,EndE)$ and $\kappa(\alpha)\triangleq\mathbb{H}^{0,2}(g^{-1}(\alpha)\wedge g^{-1}(\alpha))$ is a Kuranishi map of moduli spaces of holomorphic bundles (\ref{kappa}). We denote $\Delta(\alpha)\triangleq P_{\overline{\partial}_{A}^{*}}(g^{-1}(\alpha)\wedge g^{-1}(\alpha))$
and apply it into (\ref{chern weil formula}),
\begin{equation}-8\pi^{2}\int ch_{2}(E)\wedge\Omega=\int Tr \big(\kappa(\alpha)\wedge\kappa(\alpha)+2\kappa(\alpha)\wedge  \Delta(\alpha)+\Delta(\alpha)\wedge\Delta(\alpha)\big)\wedge\Omega. \nonumber\end{equation}
As $\kappa(\alpha)\in H^{0,2}(X,EndE)$, $\Delta(\alpha)\in \overline{\partial}_{A}^{*}\Omega^{0,3}(X,EndE)_{k}$, we get
\begin{equation}\int Tr(\kappa(\alpha)\wedge \Delta(\alpha))\wedge\Omega=\int \big(\kappa(\alpha),*_{4}\Delta(\alpha)\big)_{h}\wedge \Omega\wedge\overline{\Omega}=0 \nonumber\end{equation}
by the Hodge decomposition theorem. Similarly, we have $\int Tr\big(\Delta(\alpha)\wedge\Delta(\alpha)\big)\wedge\Omega=0$.
Thus,
\begin{equation}-8\pi^{2}\int ch_{2}(E)\wedge\Omega=\int Tr \big(\kappa(\alpha)\wedge\kappa(\alpha)\big)\wedge\Omega.\nonumber\end{equation}
$ch_{2}(E)\in H^{2,2}(X)$ for holomorphic bundle $E$ which gives $ch_{2}(E)\wedge\Omega=0$.
\end{proof}
\begin{remark}
One can check the Kuranishi map $\tilde{\kappa}$ (\ref{kappa tilta}) also satisfies $Q(\tilde{\kappa},\tilde{\kappa})=0$.
\end{remark}
The above theorem has an application to the unobstructedness of Gieseker moduli spaces.
\begin{corollary}
If for any closed point $\mathcal{F}\in\mathcal{M}_{c}$, $dim_{\mathbb{C}}Ext^{2}(\mathcal{F},\mathcal{F})\leq 1$, then $\mathcal{M}_{c}$ is smooth, i.e. all Kuranishi maps are zero.
\end{corollary}
%\begin{proof}
%By Theorem \ref{darboux thm using gauge theory} and the assumption, we have $\kappa=0$.
%\end{proof}
Then it is natural to ask whether the Kuranishi map $\tilde{\tilde{\kappa}}$ (\ref{kappa double tilta}) induced from the $DT_{4}$ equations satisfies $Q(\tilde{\tilde{\kappa}},\tilde{\tilde{\kappa}})=0$.
\begin{proposition}
The Kuranishi map $\tilde{\tilde{\kappa}}$ in (\ref{kappa double tilta}) satisfies $Q(\tilde{\tilde{\kappa}},\tilde{\tilde{\kappa}})\geq0$.
\end{proposition}
\begin{proof}
We denote the $(0,1)$ part of a unitary connection to be $\overline{\partial}_{A}+a''$, where $\overline{\partial}_{A}^{2}=0$. Then
\begin{equation}F^{0,2}(\overline{\partial}_{A}+a'')=\overline{\partial}_{A}a''+P_{\overline{\partial}_{A}}(a''\wedge a'')
+P_{\overline{\partial}_{A}^{*}}(a''\wedge a'')+\mathbb{H}^{0,2}(a''\wedge a''). \nonumber\end{equation}
To describe $\tilde{\tilde{\kappa}}$, we take $a''$ satisfies
$\overline{\partial}_{A}a''+P_{\overline{\partial}_{A}}(a''\wedge a'')+*_{4}P_{\overline{\partial}_{A}^{^{*}}}(a''\wedge a'')=0$.
After gauge fixing,
\begin{equation}\overline{\partial}_{A}^{*}a''-\frac{i}{2}\wedge(a'\wedge a''+a''\wedge a')=0. \nonumber\end{equation}
We get
\begin{equation}F^{0,2}(\overline{\partial}_{A}+a'')=P_{\overline{\partial}_{A}^{*}}(q^{-1}(\alpha)\wedge q^{-1}(\alpha))-*_{4}P_{\overline{\partial}_{A}^{^{*}}}(q^{-1}(\alpha)\wedge q^{-1}(\alpha))+\mathbb{H}^{0,2}(q^{-1}(\alpha)\wedge q^{-1}(\alpha)),  \nonumber\end{equation}
where
\begin{equation}q: \Omega^{0,1}(EndE)_{k}\rightarrow H^{0,1}(EndE)\oplus {\overline{\partial}_{A}^{*}\Omega^{0,1}(EndE)}_{k}\oplus
{\overline{\partial}_{A}^{*}\Omega^{0,2}(EndE)}_{k-1},  \nonumber \end{equation}
\begin{equation}q(a'')=\bigg(\mathbb{H}(a''),\overline{\partial}_{A}^{*}a''-\frac{i}{2}\wedge(a'\wedge a''+a''\wedge a'),
\overline{\partial}_{A}^{*}\big(\overline{\partial}_{A}a''+P_{\overline{\partial}_{A}}(a''\wedge a'')+*_{4}P_{\overline{\partial}_{A}^{*}}(a''\wedge a'')\big)\bigg) \nonumber \end{equation}
and $q(a'')=\alpha\in H^{0,1}(EndE)$.
Similar as before, we get
\begin{equation}-8\pi^{2}\int ch_{2}(E)\wedge\Omega=\int Tr\big(\tilde{\tilde{\kappa}}(\alpha)\wedge\tilde{\tilde{\kappa}}(\alpha)\big)\wedge\Omega-2\int \Big\|P_{\overline{\partial}_{A}^{*}}\big(q^{-1}(\alpha)\wedge q^{-1}(\alpha)\big)\Big\|_{h}^{2}\wedge\Omega\wedge\overline{\Omega},\nonumber\end{equation}
where $\tilde{\tilde{\kappa}}=\mathbb{H}^{0,2}(q^{-1}(\alpha)\wedge q^{-1}(\alpha))$. Thus $Q(\tilde{\tilde{\kappa}},\tilde{\tilde{\kappa}})\geq0$.
\end{proof}
Although $Q(\tilde{\tilde{\kappa}},\tilde{\tilde{\kappa}})\geq0$ instead of identically zero, it is already enough for the proof of the following promised vanishing result (see also Proposition 12 of \cite{bj}).
\begin{lemma}
We take a Kuranishi map $\kappa$ at $\mathcal{F}\in \mathcal{M}_{c}$,
\begin{equation}\kappa: Ext^{1}(\mathcal{F},\mathcal{F})\rightarrow Ext^{2}(\mathcal{F},\mathcal{F}) \nonumber\end{equation}
such that $Q(\kappa,\kappa)\geq0$, a half dimension real subspace $Ext^{2}_{+}(\mathcal{F},\mathcal{F})\subseteq Ext^{2}(\mathcal{F},\mathcal{F})$ such that $Q|_{Ext^{2}_{+}(\mathcal{F},\mathcal{F})}$ is real and positive definite, where $Q$ is the Serre duality pairing.

Then $\kappa_{+}^{-1}(0)\cong\kappa^{-1}(0)$ as topological spaces,
where $\kappa_{+}\triangleq\pi_{+}\circ\kappa$ and $\pi_{+}$ is projection to $Ext^{2}_{+}(\mathcal{F},\mathcal{F})$.
\end{lemma}
\begin{proof}
By the  above assumption,
\begin{equation}Q(\kappa,\kappa)=Q(\kappa_{+},\kappa_{+})+Q(\kappa_{-},\kappa_{-})+2Q(\kappa_{+},\kappa_{-})\geq0.  \nonumber\end{equation}
If $\kappa_{+}(a)=0$, then $Q(\kappa_{-}(a),\kappa_{-}(a))\geq0$. However, $Q$ is negative definite on $iExt^{2}_{+}(\mathcal{F},\mathcal{F})$.
Thus $Q(\kappa_{-},\kappa_{-})\leq0$ which implies that $\kappa_{-}(a)=0$.
\end{proof}
\begin{remark}
One can check the condition $Q(\kappa,\kappa)=0$ in BBJ's Darboux theorem used by Borisov and Joyce \cite{bj} for gluing can be replaced by the weaker condition $Q(\kappa,\kappa)\geq0$. Thus the local Kuranishi map $\tilde{\tilde{\kappa}}$ (\ref{kappa double tilta}) for moduli spaces of stable bundles induced from $DT_{4}$ equations (differential geometrical nature) fits into the gluing data of \cite{bj}. This then indicates the equivalence of $DT_{4}$ virtual cycles defined using purely gauge theory (Theorem \ref{main theorem}) and Borisov-Joyce's virtual cycles defined using derived $C^{\infty}$-geometry, as their virtual cycles are independent of the choices of local charts and splittings \cite{bj}.
\end{remark}

\subsection{A proof of the orientability of determinant line bundles for moduli spaces of $Spin(7)$ instantons}
In this subsection, we prove the orientability of determinant line bundles on spaces of gauge equivalence classes of connections
on $Spin(7)$ manifolds under certain assumptions. We also prove the existence of orientation data of $DT_{4}$ theory in those cases.
%We remark that we gave some partial results on the orientability of determinant line bundles for
%(generalized) $DT_{4}$ moduli spaces before in \cite{cao}. They are summarized as follows.
%\begin{proposition}(\cite{cao} Corollary 9.5, 9.7) ${}$ \\
%(1) If $\mathcal{M}_{c}^{bdl}\neq\emptyset$, $H^{*}(\mathcal{M}_{c}^{bdl},\mathbb{Z}_{2})$ is finitely generated and $H_{1}(\mathcal{M}_{c}^{bdl},\mathbb{Z})$ does not have torsion of type $\mathbb{Z}_{4k}$, then the determinant line bundle of $\mathcal{M}_{c}^{DT_{4}}$ is orientable. \\
%(2) If the Gieseker moduli space of stable sheaves $\mathcal{M}_{c}$ is smooth and $H_{1}(\mathcal{M}_{c},\mathbb{Z})$ does not have torsion of type $\mathbb{Z}_{4k}$, where $k\geq1$. Then the index bundle of the generalized $DT_{4}$ moduli space is oriented.
%\end{proposition}
%We get rid of the above strong assumption on moduli spaces and prove the following orientability result.
\begin{theorem}\label{orientablity theorem}
For any compact oriented real eight dimensional manifold $X$ with $Hol(X)\subseteq Spin(7)$ and $H^{odd}(X,\mathbb{Z})=0$, and $U(r)$ bundle $E\rightarrow X$, the determinant line bundle $\mathcal{L}$ of the index bundle of twisted Dirac operators over the space $\mathcal{B}^{*}$ is trivial.
\end{theorem}
\begin{proof}
The proof is essentially due to Donaldson, Corollary 3.22 of \cite{d1}. By the standard argument as in \cite{d1}, we are reduced to the case of $SU(N)$ bundles with $N\gg0$.
%We take a loop $\gamma: S^{1}\rightarrow \mathcal{B}^{*}$, the pairing $<w_{1}(\mathcal{L}),\gamma>$ then
%depends only on the class it defines in $[X,SU]$. By the torsion-free assumption, we have
%\begin{equation}[X,SU]\cong H^{3}(X,\mathbb{Z})\oplus H^{5}(X,\mathbb{Z}), \nonumber \end{equation}
%which vanishes by our assumption.
We only need to show $\pi_{0}(\mathcal{G})=0$ as $\pi_{1}(\mathcal{B}^{*})\cong \pi_{0}(\mathcal{G})$. By Atiyah-Bott's Proposition 2.4 \cite{ab1},
we have homotopy equivalence
\begin{equation}B\mathcal{G}\simeq Map_{P}(X,BSU),  \nonumber \end{equation}
where $Map_{P}$ denotes the component of the principal bundle $P$ in the mapping space.
Thus we are left to show $\pi_{1}(Map_{P}(X,BSU))=0$. By Theorem 11.7 of \cite{mccleary},
we have the Federer spectral sequence such that
\begin{equation}E_{2}^{p,q}\cong H^{p}(X,\pi_{p+q}(BSU))\Rightarrow \pi_{q}(Map_{P}(X,BSU)).  \nonumber \end{equation}
By our assumptions and direct calculations, we get $\pi_{1}(Map_{P}(X,BSU))=0$.
\end{proof}
\begin{remark}
See Theorem 2.1 of \cite{caoleung3} for a generalization to spin manifolds of $8k$ dimensions.
\end{remark}
%\begin{corollary}
%If $X$ is a compact simply connected Calabi-Yau 4-fold with $h^{1,2}(X)=0$ and torsion-free $H_{2}(X,\mathbb{Z})$, $H_{3}(X,\mathbb{Z})$, then the determinant line bundle $\mathcal{L}$ of the index bundle of twisted Dirac operators over the space $\mathcal{B}^{*}$ is trivial.
%\end{corollary}
%\begin{proof}
%By the universal coefficient theorem and the classification of simply connected Calabi-Yau manifolds \cite{leung}, we have $H^{odd}(X,\mathbb{Z})=0$. We are then reduced to Theorem \ref{orientablity theorem}.
%\end{proof}
For Calabi-Yau 4-folds with $SU(4)$ holonomy, we can further use Seidel-Thomas twists to prove that the structure group of $(\mathcal{L}_{\mathbb{C}},Q_{Serre})$ over $\mathcal{M}_{c}$ can be reduced to $SO(1,\mathbb{C})$.
\begin{corollary}\label{existence of ori data}
Let $X$ be a compact Calabi-Yau 4-fold with $H_{odd}(X,\mathbb{Z})=0$ and $Hol(X)=SU(4)$, $\mathcal{M}_{c}$ be a Gieseker moduli space
of stable sheaves with determinant line bundle $(\mathcal{L}_{\mathbb{C}},Q_{Serre})$.

Then the structure group of $(\mathcal{L}_{\mathbb{C}},Q_{Serre})$ can be reduced to $SO(1,\mathbb{C})$.
Furthermore, the orientation data of $DT_{4}$ theory (Definition \ref{ori data}) exists in this case.
\end{corollary}
\begin{proof}
By the work of Joyce and Song \cite{js} on the Seidel-Thomas twist \cite{st}, we are reduced to Theorem \ref{orientablity theorem}.
See also \cite{caoleung3} for more details.
\end{proof}
\begin{remark} ${}$ \\
1. See Theorem 2.2 of \cite{caoleung3} for a generalization to $CY_{2n}$. \\
2. Examples of such $CY_{4}$ are given by smooth complete intersections in smooth toric varieties.
\end{remark}

\subsection{Some remarks on Cayley submanifolds and $Spin(7)$ instantons}
By Remark \ref{remark on spin(7) instantons}, $Spin(7)$ instanton countings on Calabi-Yau 4-folds are $DT_{4}$ invariants.
Combining Borisov-Joyce's construction \cite{bj} and our results, we can define $DT_{4}$ invariants for complex vector bundles
when the corresponding Chern character satisfies $c\in \oplus_{k}H^{k,k}(X)$ and also for general stable coherent sheaves
(we need Corollary \ref{existence of ori data} to guarantee invariants to be defined in $\mathbb{Z}$ instead of $\mathbb{Z}_{2}$).
If $c\notin \oplus_{k}H^{k,k}(X)$, we do not have a definition of $DT_{4}$ invariants at this moment.

From the perspective of SYZ mirror transformations \cite{syz}, \cite{lyz} and the blow-up analysis of moduli spaces of
$Spin(7)$ instantons \cite{t}, \cite{lewis}, it would be interesting to consider the counting problem of Cayley submanifolds with ASD connections.
%We restrict ourselves to special Lagrangian submanifolds inside Calabi-Yau 4-folds with ASD connections over them.

We fix $(L,A)$, a special Lagrangian submanifold with an ASD connection on certain vector bundle over $L$. If $b_{1}(L)=0$, $L$ is rigid, i.e. it can't be deformed as a Lagrangian submanifold inside $X$. Then deformations of $(L,A)$ are deformations of $A$ and the component of the moduli space $\mathfrak{M}^{L}_{SLag+ASD}\triangleq\{(L,A) | F_{A}^{+}=0 \}$ can be identified with the moduli space $\mathfrak{M}_{ASD}(L)$ of ASD connections on $L$. %Meanwhile, $\mathfrak{M}_{ASD}(L)$ admits Uhlenbeck compatification, $\overline{\mathfrak{M}}_{ASD}(L)$.
Thus, naively speaking, $DT_{4}\textit{-type inv}\textrm{ }(\mathfrak{M}^{L}_{SLag+ASD})=\textit{Donaldson inv}\textrm{ }(L)$ if $b_{1}(L)=0$.

%The Institute of Mathematical Sciences and Department of Mathematics, The Chinese University of Hong Kong, Shatin, Hong Kong \\
%Email address: ylcao@math.cuhk.edu.hk \\
%Email address: leung@math.cuhk.edu.hk

\end{document}